\documentclass[12pt,a4paper,twoside,reqno]{amsart}

\usepackage{amsmath, amssymb, amsthm, enumerate, enumitem}
\usepackage[normalem]{ulem}
\usepackage{xr}
\makeatletter
\def\l@section{\@tocline{1}{10pt}{1pc}{}{}}
\def\l@subsection{\@tocline{2}{0pt}{1pc}{4.6em}{}}
\def\l@subsubsection{\@tocline{3}{0pt}{1pc}{7.6em}{}}
\renewcommand{\tocsection}[3]{%
  \indentlabel{\@ifnotempty{#2}{\makebox[2.3em][l]{%
    \ignorespaces#1 #2.\hfill}}}\textbf{#3}}
\renewcommand{\tocsubsection}[3]{%
  \indentlabel{\@ifnotempty{#2}{\hspace*{2.3em}\makebox[2.3em][l]{%
    \ignorespaces#1 #2.\hfill}}}#3}
\renewcommand{\tocsubsubsection}[3]{%
  \indentlabel{\@ifnotempty{#2}{\hspace*{4.6em}\makebox[3em][l]{%
    \ignorespaces#1 #2.\hfill}}}#3}
\makeatother

\newcommand{\MM}{\mathcal{M}}

\newcommand{\IR}{\mathbb{R}}

\newcommand{\IH}{\mathbb{H}}

\newcommand{\IZ}{\mathbb{Z}}

\newcommand{\ov}[1]{\overline{#1}}

\newcommand{\td}[1]{\widetilde{#1}}
\DeclareMathOperator{\Klein}{Klein}

\DeclareMathOperator{\Int}{Int}

\DeclareMathOperator{\area}{area}

\DeclareMathOperator{\tr}{tr}

\DeclareMathOperator{\dist}{dist}

\newcommand{\dotcup}{\ensuremath{\mathaccent\cdot\cup}}
\newcommand{\EMPTY}[1]{}

\newtheorem{Theorem}{Theorem}[section]
\newtheorem{Lemma}[Theorem]{Lemma}

\newtheorem{Proposition}[Theorem]{Proposition}
\newtheorem{Definition}[Theorem]{Definition}

\numberwithin{equation}{section}

\title[Long-time behavior of 3d Ricci flow --- C]{Long-time behavior of 3 dimensional Ricci flow\\C: 3-manifold topology and combinatorics of simplicial complexes in $3$-manifolds}
\author{Richard H Bamler}
\address{UC Berkeley, Department of Mathematics, 970 Evans Hall, Berkeley, CA 94720, USA}
\email{rbamler@math.berkeley.edu}
\date{\today}

\begin{document}
\begin{abstract}
In the third part of this series of papers, we establish several topological results that will become important for studying the long-time behavior of Ricci flows with surgery.
In the first part of this paper we recall some elementary observations in the topology of $3$-manifolds.
The main part is devoted to the construction of certain simplicial complexes in a given 3-manifold that exhibit useful intersection properties with embedded, incompressible solid tori.

This paper is purely topological in nature and Ricci flows will not be used.
\end{abstract}

\maketitle
\tableofcontents

\section{Introduction}
In this paper we establish several topological results that will be needed in the last part \cite{Bamler-LT-main} of this series of papers.

In the first part, section \ref{sec:3dtopology}, we recall facts from the topology of 3-manifolds, which will be frequently used in this and the subsequent paper.

In the second part of this paper, section \ref{sec:constructanalysispolygonalcomplex}, we prove a rather combinatorial-topological result (see Proposition \ref{Prop:maincombinatorialresult}).
For clarity, we will first describe a much weaker version of this result.

Consider a closed $3$-manifold $M$ that does not contain any hyperbolic pieces in its geometric decomposition, e.g. a component of the thin part $\MM_{\textnormal{thin}} (t)$ from \cite[Proposition \ref{Prop:thickthindec}]{Bamler-LT-Perelman}.
We claim that then there is a finite, $2$-dimensional simplicial complex $V$, as well as a continuous map $f_0 : V \to M$ with the following property: Whenever $\sigma \subset M$ is an embedded incompressible loop in $M$ (meaning that the fundamental group of $\sigma$ injects into the fundamental group of $M$), then $\sigma$ intersects the image of every continuous map $f : V \to M$ that is homotopic to $f_0$.

In the case in which $M$ is a $3$-torus, this statement follows in an elementary way: 
we can choose $V$ to be the disjoint union of three $2$-tori and $f_0 : V \to M$ to be an embedding that maps those $2$-tori to $2$-tori that generate the second homology of $M$.
The fact that every homotope of $f_0$ intersects every non-contractible loop of $M$ can then be seen by a standard intersection number argument.

For the purposes of \cite{Bamler-LT-main}, we will however need a somewhat stronger statement, which as it turns out, is far more difficult to prove.
This statement is also captured by Proposition \ref{Prop:maincombinatorialresult}.
We will now describe this stronger statement in a somewhat restricted setting.

Assume again that $M$ is a closed $3$-manifold that cannot be covered by a $2$-torus bundle over a circle.
Then we claim that there is a finite, $2$-dimensional simplicial complex $V$ and a continuous map $f_0 : V \to M$ such that the following holds:
Consider an arbitrary solid torus $S \subset M$, $S \approx S^1 \times D^2$ that is incompressible in $M$ (i.e. whose $S^1$-fibers are incompressible in $M$) and a map $f : V \to M$ that is homotopic to $f_0$.
Moreover, consider an arbitrary Riemannian metric $g$ on $M$.
Then there is a ``compressing domain'' for $S$ whose area, with respect to $g$, is bounded in terms of the area of $f$
By this we mean that there is a smooth domain $\Sigma \subset \IR^2$ and a smooth map $h : \Sigma \to M$ such that $h(\partial \Sigma) \subset \partial S$ and such that $h$ restricted to the exterior circle of $\Sigma$ is non-contractible in $\partial S$ and $h$ restricted to each interior circle of $\Sigma$ is contractible in $\partial S$.
This map $h$ can be chose such that its area with respect to $g$ satisfies the bound
\[ \area_g h < C \area_g f, \]
where $C$ is a constant that only depends on the topology of $M$.
Essentially, $h$ will arise from the intersection $f(V) \cap S$, taking into account multiple overlaps and multiply counted faces.
Note that this intersection can a priori be arbitrarily complex, without any bound on the number of edges, and it is a difficult task to extract a compressing domain out of it whose area is sufficiently controlled.

It is not known to the author whether such a statement remains true if $M$ is covered by a $2$-torus bundle over a circle.
In this case, however, we can make use of the special topology of $M$ and we will be able to prove a different statement, which will be sufficient for the arguments in \cite{Bamler-LT-main}:
More specifically, we will construct a sequence of continuous maps
\[ f_1, f_2, \ldots : V \to M \]
with the same simplicial complex $V$ as a domain such that for every embedded, incompressible loop $\sigma \subset M$, every $n \geq 1$ and every map $f'_n : V \to M$ that is homotopic to  $f_n : V \to M$, the image $f'_n (V)$ intersects $\sigma$ at least $n$ times.
By this we mean that $f^{\prime -1}_n(\sigma)$ contains at least $n$ points.

We refer to \cite{Bamler-LT-Introduction} for historical remarks and acknowledgements.

In the following we will assume that all manifolds are orientable and $3$ dimensional, unless stated otherwise.

\section{Preliminaries on $3$-dimensional topology} \label{sec:3dtopology}
In this section we present important topological facts, which we will frequently use in the course of this and the subsequent paper \cite{Bamler-LT-main}.
A more elaborate discussion of most of these results can be found in \cite{Hat}.
In the following, all manifolds are assumed to be connected and $3$-dimensional.

\begin{Definition}[prime manifold]
A manifold $M$ is called \emph{prime}, if it cannot be represented as a connected sum $M = M_1 \# M_2$ of two manifolds $M_1$, $M_2$ that are not diffeomorphic to spheres ($\approx S^3$).
\end{Definition}

\begin{Definition}[irreducible manifold] \label{Def:irreducible}
A manifold $M$ is called \emph{irreducible}, if every smoothly embedded $2$-sphere $S \subset M$ bounds a smoothly embedded $3$-disk $D \subset M$, $\partial D = S$.
\end{Definition}

Recall that a manifold is prime if and only if it is either irreducible or diffeomorphic to $S^1 \times S^2$ (cf \cite[Proposition 1.4]{Hat}).
We furthermore have

\begin{Proposition} \label{Prop:pi2irred}
An orientable manifold $M$ is irreducible if and only if $\pi_2 (M) = 0$.
\end{Proposition}

\begin{proof}
The backward direction follows from the Sphere Theorem (cf \cite[Theorem 3.8]{Hat}).
For the forward direction suppose that $\pi_2(M) = 0$ and let $S \subset M$ be a smoothly embedded sphere.
So $S$ is contractible and by \cite[Proposition 3.10]{Hat}, it bounds a compact contractible submanifold $N \subset M$.
Following \cite[Proposition 3.7]{Hat}, we conclude that if we attach a $3$-disk to $N$, we obtain a closed, simply-connected manifold $M'$.
By the resolution of the Poincar\'e Conjecture, $M' \approx S^3$ and hence $N$ is a $3$-disk.
\end{proof}

\begin{Definition}[incompressibility] \label{Def:incompressible}
Let $X$ be a topological space and $Y \subset X$ a connected subspace.
Then we call $Y$ \emph{(algebraically) incompressible in $X$} if the induced map $\pi_1(Y) \to \pi_1(X)$ is injective.
Otherwise, we call $Y$ \emph{(algebraically) compressible}.
\end{Definition}

\begin{Proposition}[Dehn's Lemma] \label{Prop:spanningdiskexists}
Let $M$ be a manifold with boundary.
If $C \subset \partial M$ is an embedded loop that is nullhomotopic in $M$, then $C$ bounds an embedded disk $D$ in $M$, i.e. $\partial D = C$ and $D \cap \partial M = C$.
\end{Proposition}
\begin{proof}
See \cite[Corollary 3.2]{Hat}.
\end{proof}

\begin{Proposition} \label{Prop:incompressibleequiv}
Let $M$ be a manifold (possibly with boundary) and $S \subset M$ a $2$-sided embedded, connected surface.
Then $S$ is algebraically compressible if and only if there is an embedded loop $C \subset S$ that is homotopically non-trivial in $S$ and that bounds an embedded \emph{compressing disk} $D \subset M$ which meets $S$ only in its boundary, i.e. $\partial D = C$ and $D \cap S = C$.

In particular, the statement holds if $S = \partial M$.
\end{Proposition}
\begin{proof}
See \cite[Corollary 3.3]{Hat}.
\end{proof}

We can now define what we understand by a geometric decomposition.

\begin{Definition}[Geometric decomposition of irreducible manifold] \label{Def:geomdec}
Let $M$ be a compact, orientable and irreducible 3-manifold whose boundary consists of $2$-tori.
A \emph{geometric decomposition} of $M$ is a collection of pairwise disjoint, smoothly embedded $2$-tori $T_1, \ldots, T_m \subset M$ such that
\begin{enumerate}[label=(\roman*)]
\item each torus $T_i$ is incompressible in $M$ (see Definition \ref{Def:incompressible}) and
\item each component of $M \setminus (T_1 \cup \ldots \cup T_m)$ is either \emph{hyperbolic} (i.e. it can be endowed with a complete metric of constant negative sectional curvature and finite volume) or it is \emph{Seifert} (i.e. it carries a Seifert fibration that can be extended regularly to the closure of each of its ends\footnote{Note that it may happen that the boundaries of two ends of such a component coincide.
In this case it may not be possible to extend a (or any) Seifert fibration of this component to its closure, because the extensions of the fibration on the closure of each end may not be the same.
Note also that a Seifert fibration on an orientable manifold can only have exceptional fibers of cone-type.}).
\end{enumerate}
The decomposition is called \emph{minimal} if no smaller subcollection of tori satisfies properties (i) and (ii).

If all components of $M \setminus (T_1 \cup \ldots \cup T_m)$ are Seifert, then the manifold is called \emph{(prime) graph manifold} and the decomposition is called a \emph{Seifert decomposition}.
\end{Definition}

Note that for a minimal geometric decomposition, none of the components of $M \setminus (T_1 \cup \ldots \cup T_m)$ is diffeomorphic to $T^2 \times \IR$ unless $m=1$ and $T_1$ is non-separating.
Moreover, the following is true for a minimal geometric decomposition:
Choose Seifert fibrations on each Seifert component of the decomposition, and consider a torus $T_i$ that is adjacent to a Seifert component on both sides.
Then the Seifert fibrations on either side of $T_i$ extend to two non-isotopic fibrations on $T_i$.

Lastly, we mention that a minimal geometric decomposition is unique up to isotopy (see \cite[Theorem 1.9]{Hat}).
So it is reasonable to speak of \emph{the} (minimal) geometric decomposition of a manifold.

The statement of the Geometrization Conjecture is now the following

\begin{Theorem}[Geometrization Conjecture]
Every closed, orientable, irreducible manifold admits a minimal geometric decomposition.
\end{Theorem}

Next, we will show that $3$-manifolds that are not diffeomorphic to spherical space forms or $S^2 \times S^1$ have a sufficiently complex topology and hence cannot be covered by or decomposed into certain elementary pieces.

\begin{Lemma} \label{Lem:compressingtorus}
Let $M$ be a closed, irreducible manifold and let $T \subset M$ be an embedded, $2$-sided, compressible torus.
Then $T$ separates $M$ into two components $U$, $V$ (i.e. $M = U \cup V$ and $U \cap V = T$) and we can distinguish the following cases:
\begin{enumerate}[label=(\alph*)]
\item Neither of the components $U$ or $V$ is diffeomorphic to a solid torus $S^1 \times D^2$.
Then the compressing disks $D$ for $T$ either all lie in $U$ or in $V$ and for each such $D$ a tubular neighborhood of $D \cup V$ or $D \cup U$ (depending on whether $D \subset U$ or $D \subset V$) is diffeomorphic to a $3$-ball.
\item Only one of the components $U$, $V$ is diffeomorphic to a solid torus.
Assume that this component is $U$.
Then $T$ has compressing disks in $U$.
If it also has compressing disks in $V$, then $U$ is contained in an embedded $3$-ball in $M$ and $U$ is compressible in $M$ (i.e. the map $\IZ \cong \pi_1 (U) \to \pi_1 (M)$ is not injective).
\item Both $U$ and $V$ are diffeomorphic to solid tori.
Then $M$ is diffeomorphic to a spherical space-form.
\end{enumerate}
\end{Lemma}
\begin{proof}
For the first part see \cite[p. 11]{Hat}.
Let $D$ be a compressing disk for $T$ and assume that $D \subset U$.
Again by \cite[p. 11]{Hat}, we know that either $U$ is a solid torus or a tubular neighborhood  of $D \cup V$ is diffeomorphic to a $3$-ball.
So if in case (a) there are compressing disks for $T$ in both $U$ and $V$, then $M$ is covered by two embedded $3$-balls and we have $M \approx S^3$ by Lemma \ref{Lem:coverMbysth}(a) (observe that the proof of Lemma \ref{Lem:coverMbysth}(a) does not make use of this Lemma).
However, this contradicts the fact that an embedded $2$-torus in $S^3$ bounds a solid torus on at least one side (see \cite[p. 11]{Hat}).
Case (b) follows similarly.

Consider now case (c).
Let $K_1, K_2 \subset \pi_1(T) \cong \IZ^2$ be the kernels of the projections $\pi_1(T) \to \pi_1(U)$ and $\pi_1(T) \to \pi_1(V)$.
If $K_1 = K_2$, then $M \approx S^1 \times S^2$ contradicting the assumptions on $M$.
So $K_1 \not= K_2$.
Let $a_i \in K_i$ be generators.
By an appropriate choice of coordinates, we can assume that $a_1 = (1,0) \in \IZ^2$ and $a_2 = (p,q) \in \IZ^2$ where $0 \leq p < q$.
Then $M$ is diffeomorphic to the lens space $L(p,q)$.
\end{proof}

\begin{Lemma} \label{Lem:coverMbysth}
Let $M$ be a closed manifold and assume that $M = U \cup V$.
Then
\begin{enumerate}[label=(\alph*)]
\item If $U$ and $V$ are diffeomorphic to a ball, then $M \approx S^3$.
\item If $U$ is diffeomorphic to a solid torus $S^1 \times D^2$ and $V$ is diffeomorphic to a ball, then $M \approx S^3$.
\item If $U$ and $V$ are diffeomorphic to a solid torus $\approx S^1 \times D^2$, then $M$ is either not irreducible or it is diffeomorphic to a spherical space form.
\end{enumerate}
\end{Lemma}
\begin{proof}
In case (a), we can assume that $U$ and $V$ are the interiors of compact embedded $3$-disks.
So $\partial U \subset V$.
By Alexander's Theorem (cf \cite[Theorem 1.1]{Hat}), $\partial U$ bounds a $3$-disk in $V$.
So $\partial U$ bounds a $3$-disk on both sides and hence $M \approx S^3$.

Case (b) follows along the lines; note that every embedded sphere in a solid torus bounds a ball.

For case (c) we can assume that $M$ is irreducible.
Moreover, by adding collar neighborhoods, we can assume that $\partial U \cap \partial V = \emptyset$.
Let $T = \partial U$ and $V' = M \setminus \Int U$.
Then $T$ is compressible in $V$ and by Proposition \ref{Prop:incompressibleequiv}, we find a spanning disk $D \subset \Int V$.
If also $D \subset U$, then $U \setminus D$ is a $3$-ball and $M = ( U \setminus D ) \cup V$ and we are done by case (b).
So assume that $D \subset V'$.
Then by Lemma \ref{Lem:compressingtorus}(b), either $V'$ is a solid torus or $U$ is contained in an embedded $3$-ball $B \subset M$.
In the latter case $M = B \cup V$ and we are again done by case (b).
Finally, if $V'$ is a solid torus, we are done by Lemma \ref{Lem:compressingtorus}(c).
\end{proof}

\begin{Lemma} \label{Lem:Kleinandsolidtorus}
Let $M$ be a manifold and $T \subset M$  an embedded $2$-torus that separates $M$ into two connected components whose closures $U, V \subset M$ are diffeomorphic to $\Klein^2 \td\times I$ and $S^1 \times D^2$ each.
Then $M$ is either not irreducible or it is diffeomorphic to a spherical space form.
\end{Lemma}
\begin{proof}
Consider the double cover $\widehat{U} \to U$ for which $\widehat{U} \approx T^2 \times I$.
This cover extends to a double cover $\widehat{M} \to M$.
Let $T' \subset \widehat{M}$ be the torus that projects to the zero section in $\Klein^2 \td\times I$.
Then as in the last part of the proof of Lemma \ref{Lem:compressingtorus}(c) we can write $\widehat{M} = S_1 \#_{T'} S_2$ where $S_1$ and $S_2$ are solid tori.
So $\widehat{M}$ is either diffeomorphic to $S^1 \times S^2$ or a lens space.
In the first case, $M$ is either diffeomorphic to $S^1 \times S^2$ or $\IR P^3 \# \IR P^3$ and in the second case, $M$ is still spherical (see also \cite{Asa}).
\end{proof}

The following Lemma will be important in the proof of \cite[Lemma \ref{Lem:SStori}]{Bamler-LT-main}.

\begin{Lemma} \label{Lem:Seifertfiberincompressible}
Let $M$ be compact, orientable, irreducible manifold (possibly with boundary) that is not diffeomorphic to a spherical space form.
Consider a compact, connected $3$-dimensional submanifold $N \subset M$ whose boundary components are tori and that carries a Seifert fibration.
Assume that each boundary component $T \subset \partial N$ that is compressible in $M$, either bounds a solid torus $\approx S^1 \times D^2$ on the other side or $T$ separates $M$ into two components and is incompressible in the component of $M \setminus T$ that does not contain $N$ (if $T \subset \partial M$, then this component is empty).

Then there are two cases: 
In the first case there is one boundary torus $T \subset \partial N$ that bounds a solid torus on the same side as $N$.
In the second case every boundary component of $N$ either bounds a solid torus on the side opposite to $N$ or it is even incompressible in $M$.
Moreover, in the second case, the generic Seifert fibers of $N$ are incompressible in $M$.
\end{Lemma}
\begin{proof}
Some of the following arguments can also be found in \cite{Faessler} and \cite{MorganTian}.
Denote the boundary tori of $N$ by  $T_1, \ldots, T_m$.
Assume that there is a component $T_i$ that bounds a solid torus $S_i$ on the side opposite to $N$ such that the Seifert fibers in $T_i$ are incompressible in $S_i$.
Then we can extend the Seifert fibration of $N$ to $S_i$.
So assume in the following that for any $T_i$ that bounds a solid torus $S_i$ on the other side, the Seifert fibers of $T_i$ are nullhomotopic in $S_i$.
Denote by $O$ the base orbifold of the Seifert fibration on $N$ and call the projection $\pi : N \to O$.
We remark that since $M$ is orientable, the only singular points of $O$ are cone points.
Each $T_i$ corresponds to a boundary circle $C_i = \pi(T_i) \subset \partial O$.

We first show that that there is at most one $T_i$ that bounds a solid torus $S_i$ on the side opposite to $N$ (we will call it from now on $T_1$):
Assume, there were two such components $T_1$ and $T_2$ and denote the respective solid tori by $S_1$ and $S_2$.
Let $\alpha \subset O$ be an embedded arc connecting $C_1$ and $C_2$ that does not meet any singular points.
The preimage $Z_{\alpha} = \pi^{-1}(\alpha) \subset N$ is an annulus whose boundary components are each nullhomotopic in $S_1$ or $S_2$, respectively.
Let $D_1 \subset S_1$ and $D_2 \subset S_2$ be compressing disks for $Z_{\alpha} \cap \partial S_1$ and $Z_{\alpha} \cap \partial S_2$, respectively.
Then $\Sigma_\alpha = D_1 \cup Z_\alpha \cup D_2$ is an embedded $2$-sphere.
Since $D_1$ and $D_2$ are non-separating in $S_1$ and $S_2$, respectively, we conclude that $\Sigma_\alpha$ is non-separating in $M$.
This contradicts the assumption that $M$ is irreducible.

Next, we show that if $T_1$ bounds a solid torus $S_1$ on the side opposite to $N$, then the topological surface underlying $O$ is a planar domain:
Assume not.
Then there is an embedded, non-separating arc $\alpha \subset O$ whose endpoints are distinct and lie in $C_1$.
As before, this arc yields a non-separating sphere $\Sigma_\alpha \subset M$ contradiction the irreducibility assumption of $M$.

Assume now for the rest of the proof that none of the tori $T_i$ bound a solid torus on the same side as $N$.
We will show in the following that then none of the tori $T_i$ bounds a solid torus on either side and that all $T_i$ as well as that the generic Seifert fibers on $N$ are incompressible in $M$.

First assume that $T_1$ bounds a solid torus $S_1$ (on the side opposite to $N$).
So the topological surface underlying $O$ is a planar domain.
We can find a collection of pairwise disjoint, embedded arcs $\alpha_1, \ldots, \alpha_k \subset O$ with endpoints in $C_1$ that do not meet any singular points and that cut $O$ into smaller pieces, each of which contain at most one singular point or one boundary component, and that are bounded by at most two of the arcs $\alpha_i$ and parts of $C_1$.
The corresponding spheres $\Sigma_{\alpha_1}, \ldots, \Sigma_{\alpha_k} \subset M$ bound closed $3$-balls $B_1, \ldots, B_k \subset M$.
Any two such balls are either disjoint or one is contained in the other.
Hence, either there is one $B_i$ containing all other balls or there are two balls $B_i$, $B_j$ such that any ball is contained in one of them.
In the first case set $U = S_1 \cup B_i$ and in the second case set $U = S_1 \cup B_i \cup B_j$.
From the position of the balls relatively to $S_1$ we conclude that $U$ is diffeomorphic to a solid torus.
Moreover, we can find a component $P \subset O \setminus (\alpha_1 \cup \ldots \cup \alpha_k )$ whose boundary contains the arc $\alpha_i$ or the arcs $\alpha_i$ and $\alpha_j$ (depending on whether $U = S_1 \cup B_i$ or  $U = S_1 \cup B_i \cup B_j$) such that $N =  (B_i \cap N) \cup \pi^{-1} (P)$ or $N= ((B_i \cup B_j) \cap N) \cup \pi^{-1} (P)$, respectively.
We can now distinguish the following cases:
\begin{enumerate}[label=$-$]
\item If $P$ contains an orbifold singularity, then $\pi^{-1} (P)$ is diffeomorphic to a solid torus and $M = U \cup \pi^{-1} (P)$ and we obtain a contradiction using Lemma \ref{Lem:coverMbysth}(c).
\item If $P$ contains a boundary component $C_l$ of $O$, then we argue as follows:
In this case $P$ is diffeomorphic to an annulus.
Let $\alpha' \subset P$ be an arc connecting $C_l$ with $C_1$ and choose a compressing disk $D' \subset S_1$ for the arc $Z_{\alpha'} \cap \partial S_1$.
Then $Z_{\alpha'} \cup D'$ is a compressing disk for $T_l$.
By our assumptions, $T_l$ does not bound a solid torus.
So by Lemma \ref{Lem:compressingtorus}(a), a tubular neighborhood of $T_l \cup Z_{\alpha'} \cup D'$ is diffeomorphic to a $3$-ball.
Since $P$ is diffeomorphic to an annulus, this tubular neighborhood can be extended to a tubular neighborhood whose boundary is contained in $U$.
This implies that $M$ is covered by a solid torus and a ball and Lemma \ref{Lem:coverMbysth}(b) gives us a contradiction.
\end{enumerate}
Hence, none of the $T_i$ bound a solid torus on either side.

We argue that the generic Seifert fibers of $N$ are incompressible in $N$:
Using Lemma \ref{Lem:coverMbysth}(c), we find that $O$ cannot be a bad orbifold (i.e. the tear drop or the football) or a quotient of the $2$-sphere.
So, we can find a (possibly non-compact) cover $\widehat O \to O$ such that $\widehat O$ is smooth and corresponding to this a cover $\widehat N \to N$ such that we have an $S^1$-fibration $\widehat N \to \widehat O$.
Observe that $\widehat O$ is not a $2$-sphere, because otherwise by Lemma \ref{Lem:coverMbysth}(c) $\widehat N \approx S^3$ in contradiction to our assumptions.
Using the long exact homotopy sequence and the fact that $\pi_2(\widehat O) = 0$, we conclude that a lift of any generic $S^1$-fiber $\gamma \subset N$ is incompressible in $\widehat N$ implying that $\gamma$ is incompressible in $N$.

Next we show that any generic $S^1$-fiber $\gamma$ of $N$ is incompressible in $M$:
Assume that there is a nullhomotopy $f : D^2 \to M$ for a non-zero multiple of $\gamma$.
By a small perturbation, we can assume that $f$ is transversal to the boundary tori $T_1, \ldots, T_m$.
So $f^{-1} ( T_1 \cup \ldots \cup T_m)$ consists of finitely many circles.
Look at one of those circles $\gamma' \subset D^2$ that is innermost in $D^2$ and assume $f(\gamma') \subset T_i$.
If $f|_{\gamma'}$ is homotopically trivial in $T_i$, then we can alter $f$ such that $\gamma'$ can be removed from the list.
So assume that $f|_{\gamma'}$ is homotopically non-trivial in $T_i$.
Let $D' \subset D^2$ be the disk that is bounded by $\gamma'$.
Then by Proposition \ref{Prop:incompressibleequiv} and Lemma \ref{Lem:compressingtorus}(b) we have $f(D') \subset N$.
Since the generic Seifert fibers of $N$ are incompressible in $N$, $f|_{\gamma'}$ cannot be homotopic to such a fiber, so it projects down to an arc that is homotopic to a non-zero multiple of the boundary circle $C_i$ under $\pi$.
Hence, a non-zero multiple of $C_i$ is homotopically trivial in $\pi_{1, \textnormal{orbifold}}(O)$.
We conclude that $O$ can only be a disk with possibly one orbifold singularity.
But this implies that $N$ is diffeomorphic to a solid torus, in contradiction to our assumptions.

It remains to show that all tori $T_i$ are incompressible in $M$.
By Lemma \ref{Lem:compressingtorus}(a), we conclude that if $T_i$ is compressible in $M$, then $T_i$ is contained in an embedded $3$-ball.
But this however contradicts the fact that the generic Seifert fibers of $N$ are incompressible in $M$.
\end{proof}

\section{Construction and analysis of simplicial complexes in $M$} \label{sec:constructanalysispolygonalcomplex}
\subsection{Setup and statement of the results} \label{subsec:CombinatorialSetup}
In this section, we construct a simplicial complex $V$ that will be used in \cite{Bamler-LT-main} in combination with the area evolution result from \cite{Bamler-LT-simpcx}.
We moreover analyze the intersections of images of $V$ with solid tori in $M$.
The results of this section are topological, however, we will need to make use of some combinatorial geometric arguments in the proofs.

We first recall the notion of simplicial complexes (compare also with \cite[Definition \ref{Def:simplcomplex}]{Bamler-LT-simpcx}).
\begin{Definition}[simplicial complex] \label{Def:simplcomplexrecall}
A \emph{($2$-dimensional) simplicial complex} $V$ is a topological space that is the union of embedded, closed $2$-simplices (triangles), $1$-simplices (intervals) and $0$-simplices (points) such that any two distinct simplices are either disjoint or their intersection is equal to another simplex whose dimension is strictly smaller than the maximal dimension of both simplices.
$V$ is called \emph{finite} if the number of these simplices is finite.

In this paper, we assume $V$ moreover to be \emph{locally finite} and \emph{pure}.
The first property demands that every simplex of $V$ is contained in only finitely many other simplices and the second property states that every $0$ or $1$-dimensional simplex is contained in a $2$-simplex.
We will also assume that all $2$ and $1$-simplices are equipped with differentiable parameterizations that are compatible with respect to restriction.

We will often refer to the $2$-simplices of $V$ as \emph{faces}, the $1$-simplices as \emph{edges} and the $0$-simplices as \emph{vertices}.
The \emph{$1$-skeleton} $V^{(1)}$ is the union of all edges and the \emph{$0$-skeleton} $V^{(0)}$ is the union of all vertices of $V$.
The \emph{valency} of an edge $E \subset V^{(1)}$ denotes the number of adjacent faces, i.e. the number of $2$-simplices that contain $E$.
The \emph{boundary} $\partial V$ is the union of all edges of valency $1$.
\end{Definition}

Next let $M$ be a closed, orientable, irreducible $3$-manifold that is not a spherical space form.
Consider a (not necessarily minimal) geometric decomposition of $T_1, \ldots, T_m \subset M$ of $M$, i.e. the components of $M \setminus (T_1 \cup \ldots \cup T_m)$ are either hyperbolic or Seifert (see Definition \ref{Def:geomdec} for more details).
We will assume from now on that the decomposition has been chosen such that no two hyperbolic components are adjacent to one another.
This can always be achieved by adding a parallel torus next to a torus between two hyperbolic components and hence adding another Seifert piece $\approx T^2 \times (0,1)$.
Let $M_{\textnormal{hyp}}$ be the union of the closures of all hyperbolic pieces of this decomposition and $M_{\textnormal{Seif}}$ the union of the closures of all Seifert pieces.
Then $M = M_{\textnormal{hyp}} \cup M_{\textnormal{Seif}}$ and $M_{\textnormal{hyp}} \cap M_{\textnormal{Seif}} = \partial M_{\textnormal{hyp}} = \partial M_{\textnormal{Seif}}$ is a disjoint union of embedded, incompressible $2$-tori.
Note that this construction parallels the ``thick-thin decomposition'' from \cite[Proposition \ref{Prop:thickthindec}]{Bamler-LT-Perelman}.

The goal of this section is to establish the following Proposition.
In this Proposition, we need to distinguish the cases in which $M$ is covered by a $T^2$-bundle over a circle (i.e. in which $M$ is the quotient of a $3$-torus, the Heisenberg manifold or the Solvmanifold) and in which it is not.
It is not known to the author whether part (a) of the Proposition actually holds in both cases.

\begin{Proposition} \label{Prop:maincombinatorialresult}
There is a finite simplicial complex $V$ and a constant $C < \infty$ such that the following holds:
\begin{enumerate}[label=(\alph*)]
\item In the case in which $M$ is not covered by a $T^2$-bundle over a circle there is a map
\[ f_0 : V \to M \qquad \text{with} \qquad f_0 (\partial V) \subset \partial M_{\textnormal{Seif}} \]
that is a smooth immersion on $\partial V$ such that the following holds:
Let $S \subset \Int M_{\textnormal{Seif}}$, $S \approx S^1 \times D^2$ be an embedded solid torus whose fundamental group injects into the fundamental group of $M$ (i.e. $S$ is incompressible in $M$).
Let moreover $f : V \to M$ be a piecewise smooth map that is homotopic to $f_0$ relative $\partial V$ and $g$ a Riemannian metric on $M$.
Then $f(V) \cap S \neq \emptyset$ and we can find a compact, smooth domain $\Sigma \subset \IR^2$ and a smooth map $h : \Sigma \to S$ such that $h(\partial \Sigma) \subset \partial S$ and such that $h$ restricted to the interior boundary circles of $\Sigma$ of  is contractible in $\partial S$ and $h$ restricted to the exterior boundary circle of $\Sigma$ is non-contractible in $\partial S$ and such that
\[ \area  h < C \area f. \]
\item In the case in which $M$ is covered by a $T^2$-bundle over a circle the following holds:
$\partial V = 0$ and there are continuous maps 
\[ f_1, f_2, \ldots : V \to M \]
such that for every $n \geq 1$, every map $f'_n : V \to M$ that is homotopic to $f_n$ and every embedded loop $\sigma \subset M$, with the property that all non-trivial multiples of $\sigma$ are non-contractible in $M$, the map $f'_n$ intersects $\sigma$ at least $n$ times, i.e. ${f'}_n^{-1} (\sigma)$ contains at least $n$ points.
\end{enumerate}
\end{Proposition}

We will first establish part (a) of the Proposition in subsections \ref{subsec:reducifnoT2bundle}--\ref{subsec:proofofeasiermaincombinatorialresult} and then part (b) in subsection \ref{subsec:CaseT2bundleoverS1}.

\subsection{Preliminary considerations for the case in which $M$ is not covered by a $T^2$-bundle over a circle} \label{subsec:reducifnoT2bundle}
Assume in this subsection that $M$ is not covered by a $T^2$-bundle over a circle.
In order to establish part (a) of Proposition \ref{Prop:maincombinatorialresult}, it suffices to construct a simplicial complex $V$ and a map $f_0 : V \to M$ with the desired properties for every component $M' \subset M_{\textnormal{Seif}}$, i.e. $f_0 (\partial V) \subset \partial M'$, and check that the inequality involving the areas holds for every solid torus $S \subset M'$ and every homotope $f$ of $f_0$ .
We will hence from now on fix a single component $M' \subset M_{\textnormal{Seif}}$.

The next Lemma ensures that we can pass to a finite cover of $M'$ and simplify the structure of $M'$.
This simplification is not really needed in the following analysis, but it makes its presentation more comprehensible.

\begin{Lemma} \label{Lem:ProductstructinFiniteCovering}
Under the assumptions of this subsection there is a finite cover $\widehat\pi' : \widehat{M}' \to M'$ such that the following holds:
There is a Seifert decomposition $\widehat{T}_1, \ldots, \widehat{T}_m \subset \widehat{M}'$ such that the components of $\Int \widehat{M}' \setminus (\widehat{T}_1 \cup \ldots \cup \widehat{T}_m)$ are diffeomorphic to the interiors of manifolds $\widehat{M}_j = \Sigma_j \times S^1$ for $j= 1, \ldots, k$, where each $\Sigma_j$ is a compact orientable surface (possibly with boundary).
The diffeomorphisms can be chosen in such a way that they can be smoothly extended to the boundary tori.

Moreover, one of the following cases holds:
\begin{enumerate}[label=(\Alph*)]
\item $\widehat{M}'$ is diffeomorphic to $T^2 \times I$ and $m = 0$, $k = 1$.
\item $\widehat{M}'$ is closed and diffeomorphic to an $S^1$-bundle over a closed, orientable surface $\Sigma$ with $\chi(\Sigma) < 0$.
In particular, we may assume that $m = k = 1$ and the surface $\Sigma$ arises from $\Sigma_1$ by gluing together its two boundary circles.
\item $\Sigma_j$ has at least one boundary component and $\chi(\Sigma_j) < 0$ for all $j = 1, \ldots, k$ and at each torus $T_i$ the fibers coming from the $S^1$-fibration induced from either side are not homotopic to one another.
\end{enumerate}
\end{Lemma}

\begin{proof}
The arguments in this proof are similar to those in \cite[Proposition 4.4]{Luecke-Wu-93}.

Let $T_1, \ldots, T_m \subset M'$ be a Seifert decomposition of $M'$, i.e. $T_1$, \ldots, $T_m$ are pairwise disjoint, embedded, incompressible $2$-tori such that the components of $\Int M' \setminus (T_1 \cup \ldots \cup T_m)$ are diffeomorphic to the interiors of compact Seifert spaces $M'_1, \ldots, M'_m$ of $M' \setminus (T_1 \cup \ldots \cup T_m)$ whose quotient spaces are compact orbifolds $O_1, \ldots, O_m$ (possibly with boundary) whose singularities are of cone type.

We first analyze the $2$-orbifolds $O_1, \ldots, O_m$.
By Lemma \ref{Lem:compressingtorus}(c) and the fact that $M$ is aspherical we conclude that each $O_j$ is good, i.e. its interior is diffeomorphic to an isometric quotient of $S^2, \IR^2$ or $\IH^2$ (observe that otherwise we would be able to cover $M$ by two solid tori).
By the same argument and the fact that every orbifold covering of $O_j$ induces a covering of $M_j$, it follows that $O_j$ can also not be a quotient of $S^2$.
So each $O_j$ is an isometric quotient of $\IR^2$ or $\IH^2$.

If $\Int O_j$ is diffeomorphic to an isometric quotient of $\IR^2$, then there is a finite covering $\widehat{O}_j \to O_j$ such that $\widehat{O}_j$ is diffeomorphic to a torus or an annulus.
Let $\widehat{M}'_j \to M'_j$ be the induced covering.
In the first case $m = j = 1$ and $\widehat{M} = \widehat{M}' = \widehat{M}'_j$ carries an $S^1$-fibration over $T^2$.
Since $T^2$ fibers over a circle, this would however imply that $\widehat{M}$ fibers over a circle with $T^2$-fibers, in contradiction to our assumptions.
So $\widehat{O}_j$ is diffeomorphic to an annulus and $\widehat{M}'_j \approx T^2 \times I$.
We mention the following fact, which we will use later in the proof:
For every natural number $N \geq 1$, the covering $\widehat{O}_j \to O_j$ can be chosen such that its restriction to every boundary component of $\widehat{O}_j$ is an $N$-fold covering over a circle.
We can moreover pass to a covering $\widehat{M}_j \to \widehat{M}'_j$, $\widehat{M}_j \approx T^2 \times I$ such that the composition $\widehat{M}_j \to \widehat{M}'_j \to M'_j$ over each boundary torus of $M'_j$ is an $N^2$-fold covering of $n_j := 1$ or $n_j := 2$ tori over a torus that is induced by the sublattice $N \IZ^2 \subset \IZ^2$.

If $\Int O_j$ is diffeomorphic to an isometric quotient of $\IH^2$, then by an argument from \cite[Lemma 4.1]{Luecke-Wu-93} for every large enough $N \geq 2$ we can find a finite orbifold covering $\widehat{O}_j \to O_j$ such that $\widehat{O}_j$ is a manifold and such that the covering map restricted to each boundary component of $\widehat{O}_j$ is an $N$-fold covering of the circle.
Consider the induced covering $\widehat{M}'_j \to M'_j$ where $\widehat{M}'_j$ is an $S^1$-bundle over $\Int \widehat{O}_j$.
If $\widehat{O}_j$ is closed, then we are in case (B) of the Lemma, so assume in the following that none of the $\widehat{O}_j$ is closed.
The $S^1$-fibration on each $\widehat{M}'_j$ can then be trivialized, i.e. $\widehat{M}'_j = \widehat{O}_j \times S^1$.
We can hence pass to a further $N$-fold covering $\widehat{M}_j \to \widehat{M}'_j$ using an $N$-fold covering of the $S^1$-factor.
Then for some $n_j \geq 1$ the composition $\widehat{M}_j \to \widehat{M}'_j \to M'_j$ over each boundary torus of $M'_j$ is the disjoint union of $n_j$ many $N^2$-fold coverings over the torus, induced by a sublattice $N \IZ^2 \subset \IZ^2$.

Now choose $N$ large enough such that the construction of the last two paragraphs can be carried out for every $j = 1, \ldots, m$.
Observe that the coverings over every $T_i$ coming from the coverings over the two adjacent $M'_j$ consist of equivalent pieces.
Let $N_0 = n_1 \cdots n_k$ and consider $\frac{N_0}{n_j}$ many disjoint copies of $\widehat{M}_j$ for each $j = 1, \ldots, k$.
Then these copies can be glued together along their boundary to obtain a covering $\widehat{M}' \to M'$.
The Seifert decomposition on $M'$ induces a Seifert decomposition $\widehat{T}'_1, \ldots, \widehat{T}'_{m'}$ of $\widehat{M}'$ all of whose pieces are products.

We are now almost done.
As a final step we successively remove tori $\widehat{T}'_i$ that are adjacent to Seifert components $\approx T^2 \times (0,1)$.
Since $\widehat{M}$ cannot be a $T^2$-bundle over a circle, these Seifert components can never be adjacent to such a torus $\widehat{T}'_i$ from both sides.
At the end of this process, we are either left with a single piece $\approx T^2 \times I$ and we are in case (A) of the Lemma or none of the Seifert pieces are diffeomorphic to $T^2 \times I$.
In the latter case we also remove tori $\widehat{T}'_i$ for which the $S^1$-fibers coming from either side are homotopic to one another.
This will either result in two distinct Seifert components getting joined together or in identifying two boundary tori of a single Seifert component.
If at any point in this process the new Seifert component is closed, then we undo the last step and we are in case (B).
Otherwise, we are in case (C).
\end{proof}

We will now show that Proposition \ref{Prop:maincombinatorialresult}(a) is implied by the following Proposition.

\begin{Proposition} \label{Prop:easiermaincombinatorialresultCaseb}
Let $M_0$ be an arbitrary $3$-manifold with $\pi_2 (M_0) = 0$ and $M \subset M_0$ be an embedded, connected, orientable, compact $3$-manifold with incompressible toroidal boundary components such that the fundamental group of $M$ injects into the fundamental group of $M_0$.

Assume that $M$ satisfies one of the following conditions:
\begin{enumerate}[label=(\Alph*)]
\item $M \approx T^2 \times I$.
\item $M$ is the total space of an $S^1$-bundle over a closed, orientable surface $\Sigma$ with $\chi (\Sigma) < 0$.
\item $M$ admits a Seifert decomposition $T_1, \ldots, T_m \subset M$ such that the components of $\Int M \setminus (T_1 \cup \ldots \cup T_m)$ are diffeomorphic to the interiors of $M_j = \Sigma_j \times S^1$ for $j = 1, \ldots, k$, where each $\Sigma_j$ is a compact orientable surface with at least one boundary component and $\chi (\Sigma_j) < 0$.
The diffeomorphisms can be chosen in such a way that they can be smoothly extended to the boundary tori.
Moreover, at each $T_i$ the fibers of the $S^1$-fibrations induced from the manifold $M_j$ on either side are not homotopic to one another.
\end{enumerate}
Then there is a constant $C < \infty$, a simplicial complex $V$ and a continuous map
\[ f_0 : V \to M \qquad \text{with} \qquad f_0 (\partial V) \subset \partial M \]
that is a smooth immersion on $\partial V$ such that the following holds:

Let $S \subset \Int M$, $S \approx S^1 \times D^2$ be an embedded solid torus whose fundamental group injects into the fundamental group of $M$ (i.e. $S$ is incompressible in $M$).
Let moreover $f : V \to M_0$ be a piecewise smooth map that is homotopic to $f_0$ relative to $\partial V$ in $M_0$ and $g$ a Riemannian metric on $M_0$.
Then $f(V) \cap S \neq \emptyset$ and we can find a compact, smooth domain $\Sigma \subset \IR^2$ and a smooth map $h : \Sigma \to S$ such that $h(\partial \Sigma) \subset \partial S$ and such that $h$ restricted to the interior boundary circles of $\Sigma$ of  is contractible in $\partial S$ and $h$ restricted to the exterior boundary circle of $\Sigma$ is non-contractible in $\partial S$ and such that
\[ \area  h < C \area f. \]
\end{Proposition}

\begin{proof}[Proof that Proposition \ref{Prop:easiermaincombinatorialresultCaseb} implies Proposition \ref{Prop:maincombinatorialresult}(b)]
Let $M = M_{\textnormal{hyp}} \cup M_{\textnormal{Seif}}$ be a closed, orientable, irreducible manifold as defined in subsection \ref{subsec:CombinatorialSetup} and $M'$ a component of $M_{\textnormal{Seif}}$.
By van Kampen's Theorem the fundamental group of $M'$ injects into that of $M$.
Consider now the finite covering $\widehat\pi' : \widehat{M}' \to M'$ from Lemma \ref{Lem:ProductstructinFiniteCovering}.
Choose $p \in \widehat{M}'$ and consider the push forward $\widehat\pi'_* (\pi_1 (\widehat{M}', p))$ inside $\pi_1 ( M, \widehat\pi'_* (p))$.
This subgroup induces a covering $\widehat\pi : \widehat{M} \to M$, which can be seen as an extension of $\widehat\pi' : \widehat{M}' \to M'$.
Still, the fundamental group of $\widehat{M}'$ injects into that of $\widehat{M}$.

The cases (A)--(C) of Lemma \ref{Lem:ProductstructinFiniteCovering} for $\widehat{M}'$ correspond to the conditions (A)--(C) in Proposition \ref{Prop:easiermaincombinatorialresultCaseb}.
So we can apply Proposition \ref{Prop:easiermaincombinatorialresultCaseb}  for $M \leftarrow \widehat{M}'$, $M_0 \leftarrow \widehat{M}$ and obtain a simplicial complex $V$ and a map $\widehat{f}_0 : V \to \widehat{M}'$ (observe that $\pi_2(\widehat{M}) = \pi_2(M) = 0$ by Proposition \ref{Prop:pi2irred} and by the fact that $M$ is irreducible).
Set $f_0 = \widehat\pi \circ \widehat{f}_0 : V \to M$.
Then we can lift any homotopy between $f_0$ and a map $f : V \to M$ to a homotopy between $\widehat{f}_0$ and $\widehat{f} : V \to \widehat{M}$ such that $f = \widehat\pi \circ \widehat{f}$.
Consider now an incompressible solid torus $S \subset M'$ and choose a component $\widehat{S} \subset \widehat\pi^{-1} (S) \cap \widehat{M}'$.
Since $\widehat\pi'$ is a finite covering, we find that $\widehat{S}$ is a solid torus as well, which is incompressible in $\widehat{M}'$.
So Proposition \ref{Prop:easiermaincombinatorialresultCaseb} provides a compact, smooth domain $\Sigma \subset \IR^2$ and a map $\widehat{h} : \Sigma \to \widehat{M}$ such that $\widehat{h}$ restricted to the exterior boundary circle of $\Sigma$ is non-contractible in $\partial \widehat{S}$, but $\widehat{h}$ restricted to the other boundary circles is contractible in $\partial \widehat{S}$.
Therefore, $h = \widehat\pi \circ \widehat{h}$ has the desired topological properties and we have
\[ \area h  = \area \widehat{h} < C \area \widehat{f} = C \area f. \]
This finishes the proof.
\end{proof}

In the following four subsections, we will frequently refer to the conditions (A)--(C).
We first finish off the case in which $M$ satisfies condition (A).

\begin{Proposition} \label{Prop:CasAiseasy}
Proposition \ref{Prop:easiermaincombinatorialresultCaseb} holds if $M$ satisfies condition (A).
\end{Proposition}

\begin{proof}
Observe that $M \approx T^2 \times I \approx S^1 \times S^1 \times I$.
Denote by $A_1, A_2$ the two embedded annuli of the form
\[ \{ \textnormal{pt} \} \times S^1 \times I, \;\; S^1 \times \{ \textnormal{pt} \} \times I \subset M. \]
Let $V$ be their disjoint union and $f_0 : V \to M$ be the map that restricts to the inclusion on each component of $V$.
Moreover, fix a triangulation on $V$ to give it the structure of a simplicial complex.

Note that every non-contractible loop $\sigma \subset \Int M$ has non-zero intersection number with one of the maps $f_0 |_{A_1}$ or $f_0 |_{A_2}$.
Consider now the solid torus $S \subset M$ and let $\sigma \subset \Int S$ be a non-contractible loop inside $S$ (and hence also inside $M$).
Choose $i \in \{ 1, 2 \}$ such that $f_0 |_{A_i}$ has non-zero intersection number with $\sigma$.
Then so does $f |_{A_i}$.
Let $f' : A_i \to M$ be a small perturbation of $f |_{A_i}$ that is transversal to $\partial S$ and for which $\area f' < 2 \area f$.
Still, $f'$ has non-zero intersection number with $\sigma$.

Denote the components of $f^{\prime -1} (S)$ by $Q_1, \ldots, Q_p \subset A_i \approx S^1 \times I$.
The sum of the intersection numbers of $f' |_{Q_j}$ with $\sigma$ is non-zero.
Moreover, by the choice of $i$ none of these components $Q_j$ can contain a circle that is non-contractible in $A_i$.
So each $Q_j$ is contained in a closed disk $Q'_j \subset A_i$ with $\partial Q'_j \subset \partial Q_j$, which arises from filling in all its interior boundary circles.
Note that any two such disks, $Q'_{j_1}, Q'_{j_2}$ are either disjoint or one is contained in the other.
By a maximality argument, we can choose $j \in \{ 1, \ldots, p \}$ such that the intersection number of $f' |_{Q_j}$ with $\sigma$ is non-zero, but such that $Q'_j$ does not contain any other $Q_{j'}$ with the same property.
Then $f'$ has to have zero intersection number with $\sigma$ on every component of $Q'_j \setminus Q_j$.
Hence, $f'$ restricted to every circle of $\partial Q_j \setminus \partial Q'_j$ is contractible in $\partial S$ and $f'$ restricted to $\partial Q'_j$ is non-contractible.
So if we choose $\Sigma = Q_j \subset Q'_j \approx D^2 \subset \IR^2$ and $h = f' |_{Q_j}$, then the desired properties are fulfilled and $\area h < \area f' < 2 \area f$.
\end{proof}

It remains to prove Proposition \ref{Prop:easiermaincombinatorialresultCaseb} in the cases in which $M$ satisfies condition (B) or (C).
Its proof in these two cases will be carried out in subsection \ref{subsec:proofofeasiermaincombinatorialresult}.
The proof makes use of a simplicial complex $V$, which will be constructed and analyzed in the following subsection and relies on a certain combinatorial convexity estimate on $V$, which will be derived in subsection \ref{subsec:combconvexincaseC} for case (C) and in subsection \ref{subsec:combconvexincaseB} for case (B).

\subsection{Combinatorial geometry of $\td{M}$ if $M$ satisfies condition (B) or (C)} \label{subsec:constructionofV}
In this subsection we will set up the proof of Proposition \ref{Prop:easiermaincombinatorialresultCaseb}.
In particular, we will construct the simplicial complex $V$ and introduce the tools that will be needed in the following two subsections.

Assume that $M$ satisfies condition (B) or (C) in Proposition \ref{Prop:easiermaincombinatorialresultCaseb}, i.e. $M$ is a compact, connected, orientable $3$-manifold with incompressible toroidal boundary components.
If $M$ satisfies condition (C), we fix the Seifert decomposition $T_1, \ldots, T_m$ of $M$ as well as the identifications of the components of $\Int M \setminus (T_1 \cup \ldots \cup T_m)$ with the interiors of the products $M_j \approx \Sigma_j \times S^1$ ($j = 1, \ldots, k$).
Here $\Sigma_1, \ldots, \Sigma_m$ are compact surfaces with at least one boundary component and negative Euler characteristic.
If $M$ satisfies condition (B), then we set $m = k = 1$ and we can find a torus $T_1 \subset M$ such that $M \setminus T_1$ is diffeomorphic to the interior of the product $\Sigma_1 \times S^1$, where $\Sigma_1$ is a compact, orientable surface with two boundary circles, which can be obtained from $\Sigma$ by cutting along a non-separating, embedded loop.
Moreover, $\chi(\Sigma_1) = \chi(\Sigma) < 0$.
In either case, we assume that the diffeomorphisms that identify the interior of each $M_j$ with the corresponding component of $\Int M \setminus (T_1 \cup \ldots \cup T_m)$ can be continued smoothly up to the boundary tori.
If $M$ satisfies condition (C), then the fibrations coming from either side of each torus $T_i$ are assumed to be non-homotopic to one another and in case (B) we assume that the fibration on $M_1$ has been chosen such that both fibrations agree.

We will mainly be working in the universal covering $\td{M}$ of $M$.
Let $\pi : \td{M} \to M$ be the covering projection.

\begin{Definition}[chambers]
The closures $K \subset \td{M}$ of components of the preimages of components of $M \setminus (T_1 \cup \ldots \cup T_m)$ under $\pi$ are called \emph{chambers} and the set of chambers is denoted by $\mathcal{K}$.
\end{Definition}

\begin{Definition}[walls]
The components $W$ of $\partial \td{M}$ and of the preimages $\pi^{-1} (T_i)$, $i = 1, \ldots, m$ are called \emph{walls} and the set of walls is denoted by $\mathcal{W}$.
We say that two distinct chambers $K_1, K_2 \in \mathcal{K}$ are \emph{adjacent} if they share a common wall.
\end{Definition}

By van Kampen's Theorem every chamber $K \in \mathcal{K}$ can be viewed as the universal cover of $M_{j_K}$ for a unique $j_K \in \{1, \ldots, k\}$.
So $K \approx \td\Sigma_{j_K} \times \IR$.
The boundary of $K$ is a disjoint union of walls that cover exactly the tori $T_i$ and the boundary tori of $M$ that are adjacent to $M_{j_K}$, and these tori stand in one-to-one correspondence with the boundary circles of $\Sigma_{j_K}$.
Moreover, every wall is diffeomorphic to $\IR^2$.
For later purposes, we will replace the $j$-index by $K$ and write for example $M_K = M_{j_K}$ and $\Sigma_K = \Sigma_{j_K}$.
Note that the interior of every chamber is disjoint from the union of all walls.
So the complement of the union of all walls in $\td{M}$ is equal to the union of the interiors of all chambers.

\begin{Lemma} \label{Lem:KKistree}
Every wall $W \in \mathcal{W}$, $W \not\subset \partial \td{M}$ separates $\td{M}$ into two components.
So every two distinct chambers $K_1, K_2 \in \mathcal{K}$ can only intersect in at most one wall $W = K_1 \cap K_2$ and the adjacency graph of $\mathcal{K}$ is a tree. 
\end{Lemma}

\begin{proof}
If $W \in \mathcal{W}$ did not separate $\td{M}$, then we could find a loop $\gamma \subset \td{M}$ that intersects $W$ transversally and exactly once, i.e. its intersection number with $W$ is $1$.
However $\gamma \subset \td{M}$ must be contractible.
\end{proof}

On each torus $T_i$ and boundary torus of $M$ we fix an affine structure and a point $e_i \in T_i$ for the remainder of this subsection.
These affine structures induce an affine structure on all walls $W \in \mathcal{W}$.
We can assume that the product structures on each $M_j \approx \Sigma_j \times S^1$ are chosen such that the circle fibers on each boundary component of $M_j$ coming from the $S^1$-factor and the boundary circle of the $\Sigma_j$ are geodesic circles in the corresponding torus $T_i$.

Now, for each $j = 1, \ldots, k$ we choose an embedded section $S_j \subset M_j \approx \Sigma_j \times S^1$ of the form $\Sigma_j \times \{ \text{pt} \}$.
Next, we choose embedded and pairwise disjoint arcs inside each $\Sigma_j$, whose endpoints lie in the boundary of $\Sigma_j$ and that cut the interior of $\Sigma_j$ into a topological ball.
Denote their union by $C^*_j \subset \Sigma_j$ and set $C_j = C^*_j \times S^1$.
Let now
\[ V = T_1 \cup \ldots \cup T_m \cup S_1 \cup \ldots \cup S_k \cup C_1 \cup \ldots \cup C_k. \]
The complement $\Int M \setminus V$ is a disjoint union of $k$ topological balls $\approx (\Sigma_1 \setminus C^*_1) \times (0,1), \ldots, (\Sigma_k \setminus C^*_k) \times (0,1)$.

By construction, $V$ can be seen as an embedded, finite and pure 2-dimensional polygonal complex with $\partial V \subset \partial M$.
The notion of a ``polygonal complex'' generalizes the notion of a ``simplicial complex'' from Definition \ref{Def:simplcomplexrecall}, by allowing the faces to be polygons instead of 2-simplices.
The $1$-skeleton $V^{(1)}$ of $V$, seen as a polygonal complex, is the union of $\partial S_j$, $\partial C_j$ and $C_j \cap S_j$ for all $j = 1, \ldots, k$.
And the set of vertices $V^{(0)}$ of $V$ is contained in $T_1 \cup \ldots \cup T_m \cup \partial M$.
By subdividing the polygonal faces of $V$ into triangles, we can give $V$ the structure of a finite and pure simplicial complex.
In doing so, we increase the number of edges and faces of $V$, but the topology of $V$ still remains the same.
In the following, we will fix this simplicial structure, and the structure of $V$ as a polygonal complex will not be essential for us anymore.

Consider now the universal covering $\pi : \td{M} \to M$ and set $\td{V} = \pi^{-1} (V) \subset \td{M}$.
Then $\td{V}$ inherits the structure of an infinite simplicial complex with $\partial \td{V} \subset \partial \td{M}$ and the components of $\Int \td{M} \setminus \td{V}$ are topological balls on which $\pi$ is injective.
Their boundaries are diffeomorphic to simplicial $2$-spheres.

\begin{Definition}[cells]
The closure $Q$ of any component component of $\td{M} \setminus \td{V}$ is called a \emph{cell} and the set of cells is denoted by $\mathcal{Q}$.
Two distinct cells are called \emph{adjacent} if their intersection contains a point of $\td{V} \setminus \td{V}^{(1)}$.
Here $\td{V}^{(1)}$ denotes the 1-skeleton of $\td{V}$ viewed as a simplicial complex.
(Note that the notion of adjacent cells would be the same, if $\td{V}^{(1)}$ denoted the 1-skeleton of $\td{V}$ with respect to the inherited polygonal structure on $\td{V}$ as opposed to the simplicial structure.)
\end{Definition}

So every chamber $K \in \mathcal{K}$ is equal to the union of cells $Q \subset K$.
Identify $K$ with $\td\Sigma_K \times \IR$ as before and set $C_K = C_{j_K}$ and $C^*_K = C^*_{j_K}$.
The structure of $\td{V}$ in $K$ can then be understood as follows:
Let $\td{C}^*_K$ be the preimage of $C^*_K$ under the universal covering map $\td\Sigma_K \to \Sigma_K$.
Then $\td{V} \cap K$ is equal to the union of $\pi^{-1} (C_K) \cap K \approx \td{C}^*_K \times \IR$ with $\pi^{-1} (S_K) \cap K$ and $\partial K$.
So the arrangement of the cells $Q \subset K$ is reflected by the following identity
\begin{equation} \label{eq:arrangementQinK}
\bigcup_{Q \in \mathcal{Q}, \; Q \subset K} \Int Q = (\Int \td\Sigma_K \setminus \td{C}^*_K) \times (\IR \setminus \IZ).
\end{equation}
We will always refer to the first factor in this cartesian product as the \emph{horizontal} direction and to the second factor as the \emph{vertical} direction.
In the next definition we group cells that share the same horizontal coordinates.

\begin{Definition}[columns]
Consider a chamber $K \in \mathcal{K}$ and choose the identification $K \cong \td\Sigma_K \times \IR$ as in the last paragraph.
Then the closure $E$ of each component of $(\Int \td\Sigma_K \setminus \td{C}^*_K) \times \IR$ is called a \emph{column}.
The set of columns of $K$ is denoted by $\mathcal{E}_K$.

We say that two distinct columns $E_1, E_2 \in \mathcal{E}_K$ are \emph{adjacent} if they intersect.
An ordered tuple $(E_0, \ldots, E_n)$ of columns for which $E_i$ is adjacent to $E_{i+1}$ is called a \emph{chain between $E_0$ and $E_n$} and $n$ is called its \emph{length}.
It is called \emph{minimal} if its length is minimal amongst all chains between the same columns.
\end{Definition}

So each chamber $K \in \mathcal{K}$ is equal to the union of all its columns $E \in \mathcal{E}_K$ and every such column $E$ consists of cells $Q \subset E$, which are arranged in a linear manner.
Next, we define distance functions with respect to the horizontal and vertical direction in (\ref{eq:arrangementQinK}).

\begin{Definition}[horizontal and vertical distance within a chamber]
Let $K \in \mathcal{K}$ be a chamber and $E_1, E_2 \in \mathcal{E}_K$ two columns.
We define their \emph{horizontal distance  $\dist^H_K (E_1, E_2)$ (within $K$)} to be the minimal length of a chain between $E_1$ and $E_2$.
For two cells $Q_1, Q_2 \subset K$ with $Q_1 \subset E_1$ and $Q_2 \subset E_2$ we define the \emph{horizontal distance $\dist^H_K (Q_1, Q_2) = \dist^H_K (E_1, E_2)$ (within $K$)}.
We say that $Q_1, Q_2$ are \emph{vertically aligned (within $K$)} if $\dist^H_K (Q_1, Q_2) = 0$, i.e. if $Q_1, Q_2$ lie in the same column.

For two cells $Q_1, Q_2 \subset K$ we define the \emph{vertical distance $\dist^V_K (Q_1, Q_2)$ (within $K$)} by the minimal number of times that an arc $\gamma : [0,1] \to K$ with $\gamma(0) \in \Int Q_1$ and $\gamma(1) \in \Int Q_2$ intersects $\pi^{-1} (S_K)$, i.e. the number if integers between the second coordinates of both cells in (\ref{eq:arrangementQinK}).
We say that $Q_1, Q_2$ are \emph{horizontally aligned (within $K$)} if $\dist^V_K (Q_1, Q_2) = 0$.
\end{Definition}

Obviously, both distance functions satisfy the triangle inequality.
Two cells $Q_1, Q_2 \subset K$ are adjacent if and only if $\dist^H_K(Q_1, Q_2) + \dist^V_K(Q_1, Q_2) = 1$.
And they are disjoint if and only if this sum is $\geq 2$ and not both summands are equal to $1$.

\begin{Lemma} \label{Lem:EEinKisatree}
Assume that $M$ satisfies condition (B) or (C).
Consider a chamber $K \in \mathcal{K}$.
Then the set of columns $\mathcal{E}_K$ together with the adjacency relation describes a tree with constant valency $\geq 4$.
So for every two columns $E_1, E_2 \in \mathcal{E}_K$, there is a unique minimal chain between $E_1, E_2$ and a chain between $E_1, E_2$ is the minimal one if and only if it  contains each column not more than once.
Moreover, for every three columns $E_1, E_2, E_3 \in \mathcal{E}_K$ there is a unique column $E^* \in \mathcal{E}_K$ that lies on all three minimizing chains between every pair of $E_1, E_2, E_3$.

Finally, for every two columns $E_1, E_2 \in \mathcal{E}_K$ with $\dist^H_K (E_1, E_2) \geq 2$ there is at most one wall $W \in \mathcal{W}$ that is adjacent to both $E_1$ and $E_2$.
\end{Lemma}

\begin{proof}
By an intersection number argument as in the proof of Lemma \ref{Lem:KKistree}, we find that a loop in $\td\Sigma_K$ cannot cross a component of $\td{C}^*_K \subset \td\Sigma_K$ exactly once.
This establishes the tree property.

Now assume that there are two distinct boundary components $B_1, B_2 \subset \partial \td\Sigma_K$ that are adjacent to two distinct components $U_1, U_2 \subset \td\Sigma_K \setminus \td{C}^*_K$ at the same time.
Since $\td\Sigma_K$ is simply connected, the closure of the set $B_1 \cup B_2 \cup U_1 \cup U_2$ separates $\td\Sigma_K$ into two open components $A_1, A_2$ one of which, say $A_1$, has compact closure.
So $A_1$ only contains finitely many components of $\td\Sigma_K \setminus \td{C}^*_K$ and all these components are only adjacent to each other or to $U_1$ or $U_2$.
This however contradicts the tree property.
\end{proof}

In the following we want to understand the adjacency structure of $\mathcal{Q}$ on $\td{M}$.
As a first step we analyze its structure near walls.

\begin{Lemma} \label{Lem:QQestimatesatW}
There is a constant $C_0 < \infty$ such that the following holds:

Let $W \in \mathcal{W}$, $W \not\subset \partial M$ be a wall and let $K, K' \in \mathcal{K}$ be the chambers that are adjacent to $W$ from either side.
Then the columns $E \in \mathcal{E}_K$, $E' \in \mathcal{E}_{K'}$ intersect $W$ in affine strips $E \cap W$, $E' \cap W$ (i.e. domains bounded by two parallel straight lines).
In case in which $M$ satisfies condition (B), these strips are all parallel and if $M$ satisfies condition (C), each pair of strips coming from $K$ and $K'$ are not parallel to one another; so they intersect in a non-empty compact set.

We furthermore have the following estimates between the horizontal and vertical distance functions in $K$ and $K'$:
\begin{enumerate}[label=(\alph*)]
\item Assume that $M$ satisfies condition (B) or (C) and let $Q_1, Q_2 \subset K$ be cells that are adjacent to a common cell $Q' \subset K'$.
Then
\[ \dist^H_K (Q_1, Q_2), \; \dist^V_K (Q_1, Q_2) < C_0. \]
\item Assume that $M$ satisfies condition (C) and let $Q_1, Q_2 \subset K$, $Q'_1, Q'_2 \subset K'$ be cells such that $Q_1, Q'_1$ and $Q_2, Q'_2$ are adjacent and such that $Q'_1, Q'_2$ are vertically aligned.
Then
\[ \dist^V_K (Q_1, Q_2), \; \dist^V_{K'}(Q'_1, Q'_2) < C_0 \dist^H_K (Q_1, Q_2) + C_0. \]
\item Assume that $M$ satisfies condition (C) and consider four cells $Q_1$, $Q_2$, $Q_3$, $Q_4 \subset K$.
Assume that $Q_1, Q_2$ and $Q_3, Q_4$ are vertically aligned and assume that there are columns $E'_1, E'_2 \in \mathcal{E}_{K'}$ such that $Q_1, Q_4$ are adjacent to some cells in $E'_1$ and $Q_2, Q_3$ are adjacent to some cells in $E'_2$.
Then
\[ \big| \dist^V_K (Q_1, Q_2) - \dist^V_K (Q_3, Q_4) \big| < C_0. \]
\item Assume that $M$ satisfies condition (C) and consider cells $Q_1, Q_2 \subset K$ and $Q'_1, Q'_2 \subset K'$ such that $Q_1, Q'_1$ and $Q_2, Q'_2$ are adjacent and that $\dist^H_K (Q_1, Q_2), \dist^H_{K'} (Q'_1, Q'_2) \leq 3$.
Then
\[ \dist^V_K (Q_1, Q_2), \; \dist^V_{K'} (Q'_1, Q'_2) < C_0. \]
\item Assume that $M$ satisfies condition (B) and consider cells $Q_1, Q_2 \subset K$ and $Q'_1, Q'_2 \subset K'$ such that $Q_1, Q'_1$ and $Q_2, Q'_2$ are adjacent.
Then
\[ \dist^V_K (Q_1, Q_2) < \dist^V_{K'} (Q'_1, Q'_2) + C_0 \dist^H_K (Q_1, Q_2) + C_0. \]
\item Assume that $M$ satisfies condition (B) or (C) and consider cells $Q_1, Q_2 \subset K$ and $Q'_1, Q'_2 \subset K'$ such that $Q_1, Q'_1$ and $Q_2, Q'_2$ are adjacent.
Then
\[ \qquad\quad\; \dist^H_K (Q_1, Q_2), \; \dist^V_K(Q_1, Q_2) < C_0 \dist^H_{K'} (Q'_1, Q'_2) + C_0 \dist^V_{K'} (Q'_1, Q'_2) + C_0. \]
\item Assume that $M$ satisfies condition (B) or (C) and consider cells $Q_1, \linebreak[1] Q_2, \linebreak[1] Q_3, \linebreak[1] Q_4 \subset K$ and $Q'_1, Q'_2, Q'_3, Q'_4 \subset K'$ such that $Q_i$ and $Q'_i$ are adjacent for all $i = 1, \ldots, 4$.
Assume moreover that $\dist^H_K(Q_1, Q_2) = \dist^H_K (Q_3, Q_4)$ and $\dist^V_K(Q_1, Q_2) = \dist^V_K (Q_3, Q_4)$ in an oriented sense, i.e. the cells $Q_1, Q_2, Q_3, Q_4$ form a ``parallelogram'' along $W$.
Then
\[ \qquad\quad \big| \dist^H_{K'} (Q'_1, Q'_2) - \dist^H_{K'} (Q'_3, Q'_4) \big|, \; \big| \dist^V_{K'} (Q'_1, Q'_2) - \dist^V_{K'} (Q'_3, Q'_4) \big| < C_0. \]
\end{enumerate}
\end{Lemma}

\begin{proof}
Note first, that in all cases (a)--(g) the cells $Q_i$ and $Q'_i$ are adjacent to $W$ (meaning that $Q_i$ and $Q'_i$ intersect $W$).
The cells of $K$ and $K'$ that are adjacent to $W$ are arranged in a doubly periodic pattern along $W$.
So we can introduce euclidean coordinates $(x_1, x_2) : W \to \IR^2$ such that for any two cells $Q_1, Q_2 \subset K$ that are adjacent to $W$ and any points $p_1 \in Q_1 \cap W$, $p_2 \in Q_2 \cap W$ we have $|\dist^H_K (Q_1, Q_2) - |x_1(p_1) - x_1(p_2)| | < C$ and $|\dist^V_K (Q_1, Q_2) - |x_2(p_1)-x_2(p_2)| | < C$ for some uniform constant $C$.
Similarly, we can find Euclidean coordinates $(x'_1, x'_2) : W \to \IR^2$ with the analogous behavior for the cells of $K'$ that are adjacent to $W$ such that the origins of $(x_1, x_2)$ and $(x'_1, x'_2)$ agree.
The transformation matrix $A \in \IR^{2 \times 2}$ with $A (x_1, x_2) = (x'_1, x'_2)$ is invertible.
In case (C) we have $A_{12} \neq 0$ and in case (B) we have $A_{12} = 0, A_{11} \neq 0$ and $A_{22} = 1$.
All assertions of the Lemma now follow from the corresponding statements for these two coordinate systems.
\end{proof}

Next, consider a smooth arc $\gamma : [0,1] \to \td{M}$.

\begin{Definition}[general position]
We say that $\gamma$ is in \emph{general position} if its endpoints $\gamma(0), \gamma(1) \not\in \td{V}$ and if $\gamma$ intersects $\td{V}$ transversally and only in $\td{V} \setminus \td{V}^{(1)}$.
If $Q_1, Q_2 \in \mathcal{Q}$ are two cells with $\gamma(0) \in Q_1$ and $\gamma(1) \in Q_2$, then we say that $\gamma$ \emph{connects $Q_1$ with $Q_2$}.
\end{Definition}

Let $\eta, H > 0$ be constants whose value will be determined later in subsection \ref{subsec:proofofeasiermaincombinatorialresult}.
In the course of the following three subsections, we will need to assume that $\eta$ is small enough and $H$ is large enough to make certain arguments work out.

\begin{Definition}[length and distance]
The \emph{(combinatorial) length} $|\gamma|$ of an arc $\gamma : [0,1] \to \td{M}$ in general position is defined as
\begin{multline*}
 |\gamma| = \eta \; \big(\textnormal{number of intersections of $\gamma$ with $\pi^{-1} (S_1 \cup \ldots \cup S_k)$} \big) \\
 + H \; \big(\textnormal{number of intersections of $\gamma$ with $\pi^{-1} (T_1 \cup \ldots \cup T_m)$} \big) \\
  + \big(\textnormal{number of intersections of $\gamma$ with $\pi^{-1} (C_1 \cup \ldots \cup C_k)$} \big) 
\end{multline*}
The \emph{(combinatorial) distance} $\dist (Q_1, Q_2)$ between two cells $Q_1, Q_2 \in \mathcal{Q}$ is the minimal combinatorial length of all arcs in general position between $Q_1$ and $Q_2$.
An arc $\gamma : [0,1] \to \td{M}$ in general position is said to be \emph{(combinatorially) minimizing} if its length is equal to the combinatorial distance between the two cells that contain its endpoints.
\end{Definition}

Observe that $(\mathcal{Q}, \dist)$ is a metric space.
On a side note, it is an interesting ``coincidence'' that this metric space approximates the conjectured geometric behavior of the Ricci flow metric $t^{-1} g_t$ lifted to the universal cover $\td{M}$.

Our main characterization of combinatorially minimizing arcs will be stated in Proposition \ref{Prop:characterizationminimizinggamma} in case (C) and in Proposition \ref{Prop:characteriztionminimizinggammaCaseB} in case (B).
We conclude this subsection by pointing out three basic properties of combinatorially minimizing arcs.

\begin{Lemma} \label{Lem:subsegmentofminimizing}
If $\gamma : [0,1] \to \td{M}$ is combinatorially minimizing, then so is every orientation preserving or reversing reparameterization and every subsegment of $\gamma$ whose endpoints don't lie in $\td{V}$.
\end{Lemma}

\begin{proof}
Obvious.
\end{proof}

\begin{Lemma}
For any cell $Q \in \mathcal{Q}$, the preimage $\gamma^{-1} (Q)$ under a combinatorially minimizing arc $\gamma : [0,1] \to \td{M}$ is a closed interval, i.e. $\gamma$ does not reenter $Q$ after exiting it.
\end{Lemma}

\begin{proof}
Otherwise we could replace $\gamma$ by a shorter arc.
\end{proof}

\begin{Lemma} \label{Lem:gammainsidesingleK}
Assume that $\gamma : [0,1] \to \td{M}$ is combinatorially minimizing and stays within some chamber $K \in \mathcal{K}$.
Let $E_0, \ldots, E_n \in \mathcal{E}_K$ be the columns that $\gamma$ intersects in that order.
Then $(E_0, \ldots, E_n)$ is a minimal chain in $\mathcal{E}_K$.

Moreover, $\gamma$ intersects each component of $\pi^{-1} (S_1 \cup \ldots \cup S_k)$ at most once.
So if the endpoints of $\gamma$ lie in cells $Q_1, Q_2 \in \mathcal{Q}$, then
\[ |\gamma|  = \dist (Q_1, Q_2) = \dist^H_K (Q_1,Q_2) + \eta \dist^V_K(Q_1, Q_2). \]
Finally, for any two cells $Q_1, Q_2 \subset K$ we have $\dist(Q_1, Q_2) \leq \dist^H_K (Q_1, Q_2) + \eta \dist^V_K (Q_1, Q_2)$.
\end{Lemma}

\begin{proof}
This follows from the cell structure of $K$ (see also (\ref{eq:arrangementQinK})).
\end{proof}

\subsection{A combinatorial convexity estimate if $M$ satisfies condition (C)} \label{subsec:combconvexincaseC}
In this subsection we assume that $M$ satisfies condition (C) in Proposition \ref{Prop:easiermaincombinatorialresultCaseb}.
We will analyze the combinatorial distance function on $\mathcal{Q}$ in this case.
The main result in this section will be the combinatorial convexity estimate in Proposition \ref{Prop:combinatorialconvexityseveralSeifert}.

\begin{Lemma} \label{Lem:gammainsidetwoKKs}
There are constants $\eta^* > 0$ and $H^* < \infty$ such that if $\eta \leq \eta^*$ and $H \geq H^*$, the following holds:

Consider two chambers $K, K' \in \mathcal{K}$ that are adjacent to a common wall $W = K \cap K'$ from either side and assume that $\gamma : [0,1] \to K \cup K'$ is combinatorially minimizing.
Then $\gamma$ intersects $W$ at most twice.
\begin{enumerate}[label=(\alph*)]
\item If $\gamma$ intersects $W$ exactly once, then there is a unique column $E \in \mathcal{E}_K$ in $K$ that is both adjacent to $W$ and that intersects $\gamma$.
The same is true in $K'$.
\item If $\gamma$ intersects $W$ exactly twice and $\gamma(0), \gamma(1) \in K$, then there is a unique column $E' \in \mathcal{E}_{K'}$ such that $\gamma$ is contained in $K \cup E'$.
Moreover there are exactly two columns $E_1, E_2 \in \mathcal{E}_K$ that are adjacent to $W$ and that intersect $\gamma$.
And we have $\dist^H_K (E_1, E_2) > H$.
\item If $\gamma$ does not intersect $W$, but intersects two columns $E_1, E_2 \in \mathcal{E}_K$ that are both adjacent to $W$, then $\dist^H_K (E_1, E_2) < 3H$.
\end{enumerate}
\end{Lemma}

\begin{proof}
We first establish assertion (a).
Assume without loss of generality that $\gamma(0) \in K$ and $\gamma(1) \in K'$.
Let $Q \subset K$ be the last cell that $\gamma$ intersects inside $K$ and $Q' \subset K'$ the first cell in $K'$.
So $Q, Q'$ are adjacent.
Let $E \in \mathcal{E}_K$ be the column that contains $Q$ and $E'$ the column that contains $Q'$.
Assume that contrary to the assertion there is another column $E_1 \neq E \in \mathcal{E}_K$ that is adjacent to $W$ and intersects $\gamma$.
Choose a cell $Q_1 \subset E_1$ that intersects $\gamma$.
Then by Lemma \ref{Lem:gammainsidesingleK}
\[ \dist(Q_1, Q') = \dist(Q_1, Q) + H = \dist^H_K (E_1, E) + \eta \dist^V_K (Q_1, Q) + H. \]
Let $Q_2 \subset E_1$ be the cell that is horizontally aligned with $Q$.
Then, by the triangle inequality and Lemma \ref{Lem:gammainsidesingleK},
\[ \dist(Q_2, Q') \leq \dist(Q_2, Q) + \dist(Q, Q') \leq \dist_K^H (E_1, E) + H. \]
On the other hand, again by the triangle inequality and Lemma \ref{Lem:gammainsidesingleK},
\begin{multline*}
 \dist(Q_2, Q') + \eta \dist_K^V( Q_1, Q) \geq \dist(Q_2, Q') + \dist(Q_1, Q_2) \\
 \geq \dist(Q_1, Q') = \dist^H_K (E_1, E) + \eta \dist^V_K (Q_1, Q) + H 
\end{multline*}
It follows that
\[ \dist(Q_2, Q') = \dist^H_K (E_1, E) + H. \]
Since $E_1 \cap W$ and $E' \cap W$ are non-parallel strips in $W$, we can find cells $Q_3 \subset E_1$ and $Q'_3 \subset E'$ that are adjacent to each other and by Lemma \ref{Lem:QQestimatesatW}(b) we can estimate
\[ \dist^V_K (Q_2, Q_3) = \dist^V_K (Q, Q_3), \; \dist^V_{K'} (Q', Q'_3) < C_0 \dist^H_K (E_1, E) + C_0. \]
We then conclude
\begin{multline*}
 \dist^H_K (E_1, E) + H = \dist(Q_2, Q') \leq \dist(Q_2, Q_3) + \dist(Q_3, Q'_3) + \dist(Q'_3, Q') \\
 < 2 \eta \big( C_0 \dist^H_K (E_1, E) + C_0 \big) + H.
\end{multline*}
For $\eta < (4 C_0)^{-1}$ this implies $\dist^H_K (E_1, E) < 1$ and hence $E_1 = E$.

Next, we show assertion (b).
Let $E_1, E_2 \in \mathcal{E}_K$ be the the columns that $\gamma$ intersects right before intersecting $W$ for the first time and right after intersecting $W$ for the second time.
Let $E'_1, E'_2 \in \mathcal{E}_{K'}$ be the first and last columns that $\gamma$ intersects inside $K'$.
Assertion (a) applied to the subsegments of $\gamma$ between $\gamma(0)$ and $E'_2$ and between $E'_1$ and $\gamma(1)$, yields that $E' := E'_1 = E'_2$.
Since the subsegment of $\gamma$ that is contained in $K'$ has both of its endpoints in $E'$, it has to be fully contained in it.
Moreover, assertion (a) implies that there are no other columns than $E_1, E_2$ in $K$ that are adjacent to $W$ and intersect $\gamma$.

It remains to show the inequality on the horizontal distance between $E_1, E_2$.
We will do this by comparing the intrinsic and extrinsic distance between these two columns.
Choose $Q_1 \subset E_1$ and $Q'_1 \subset E'$ such that $\gamma$ crosses $W$ between $Q_1$ and $Q'_1$ for the first time and pick $Q_2 \subset E_2$ and $Q'_2 \subset E'$ accordingly.
So $Q_1, Q'_1$ and $Q_2, Q'_2$ are adjacent.
Lemma \ref{Lem:QQestimatesatW}(b) provides the bound $\dist^V_K (Q_1, Q_2) < C_0 \dist^H_K (Q_1, Q_2) + C_0$.
So
\begin{multline*}
 2H \leq \dist(Q_1, Q_2) \leq \dist^H_K (Q_1, Q_2) + \eta \dist^V_K (Q_1, Q_2) \\
  < (1 + \eta C_0) \dist^H_K (Q_1, Q_2) + \eta C_0 .
\end{multline*}
The desired inequality follows for $\eta < (2C_0)^{-1}$ and $H >2$.
This finishes the proof of assertion (b).

We can now show that $\gamma$ intersects $W$ at most twice.
Assume not.
After passing to a subsegment and possibly reversing the orientation, we may assume that $\gamma$ intersects $W$ exactly three times and that $\gamma(0) \in K$, $\gamma(1) \in K'$.
By assertion (b) applied to subsegments of $\gamma$ which intersects $W$ exactly twice, we find that there are columns $E_1, E_3 \in \mathcal{E}_K$ and $E'_2, E'_4 \in \mathcal{E}_{K'}$, which are all adjacent to $W$ such that $\gamma$ crosses $W$ first between $E_1$ and $E'_2$, then between $E'_2$ and $E_3$ and finally between $E_3$ and $E'_4$.
Choose cells $Q_1 \subset E_1$, $Q'_1, Q'_2 \subset E'_2$, $Q_2, Q_3 \subset E_3$, $Q'_3 \subset E'_4$ such that $\gamma$ crosses $W$ first between $Q_1$ and $Q'_1$, then between $Q'_2$ and $Q_2$ and finally between $Q_3$ and $Q'_3$.
So
\begin{multline*}
 \dist(Q_1, Q'_3) = \dist(Q_1, Q'_1) + \dist(Q'_1, Q'_2) + \dist(Q'_2, Q_2) \\ + \dist(Q_2, Q_3) + \dist(Q_3, Q'_3) = 3H + \eta \dist^V_{K'} (Q'_1, Q'_2) + \eta \dist^V_K (Q_2, Q_3).
\end{multline*}
Since $E_1 \cap W$ and $E'_4 \cap W$ are non-parallel strips in $W$, we can find cells $Q^* \subset E_1$ and $Q^{*\prime} \subset E'_4$ that are adjacent to each other.
By Lemma \ref{Lem:QQestimatesatW}(c)
\begin{multline*}
 \dist^V_K (Q_1, Q^*) < \dist^V_K (Q_2, Q_3) + C_0 \qquad \text{and} \\ \quad
\dist^V_{K'} (Q^{*\prime}, Q'_3) < \dist^V_{K'} (Q'_1, Q'_2) + C_0.
\end{multline*}
So
\begin{multline*}
\dist (Q_1, Q'_3) \leq \dist(Q_1, Q^*) + \dist(Q^*, Q^{* \prime}) + \dist(Q^{* \prime}, Q'_3) \\
< \eta \dist^V_K (Q_2, Q_3) + \eta C_0 + H + \eta \dist^V_{K'} (Q'_1, Q'_2) + \eta C_0 \\
= \dist(Q_1, Q'_3) - 2H + 2 \eta C_0.
\end{multline*}
We obtain a contradiction for $\eta C_0 < H$.

Finally, we show assertion (c).
Assume now that $\gamma$ does not intersect $W$ and choose cells $Q_1 \subset E_1$ and $Q_2 \subset E_2$ that intersect $\gamma$.
Since $\gamma$ stays within $K$ we have
\[ \dist(Q_1, Q_2) = \dist^H_K (E_1, E_2) + \eta \dist^V_K (Q_1, Q_2). \]
Let $Q_3 \subset E_2$ be the cell that is horizontally aligned with $Q_1$.
By the triangle inequality and by previous equation
\[ \dist(Q_1, Q_3) = \dist^H_K (E_1, E_2). \]
Let $Q'_1 \subset K'$ be a cell that is adjacent to $Q_1$ and let $E' \in \mathcal{E}_{K'}$ be the column that contains $Q'_1$.
Since $E' \cap W$ and $E_2 \cap W$ are non-parallel strips, we can find cells $Q'_2 \subset E'$ and $Q''_2 \subset E_2$ that are adjacent.
By Lemma \ref{Lem:QQestimatesatW}(b), we have $\dist^V_K(Q''_2, Q_3), \dist^V_{K'} (Q'_1, Q'_2) < C_0 \dist^H_K (E_1, E_2) + C_0$.
So
\begin{multline*}
 \dist^H_K (E_1, E_2) = \dist(Q_1, Q_3) \leq \dist(Q_1, Q'_1) + \dist(Q'_1, Q'_2) \\+ \dist(Q'_2, Q''_2) + \dist(Q''_2, Q_3) 
 < 2 H + 2 \eta C_0 \dist^H_K (E_1, E_2) + 2 \eta C_0.
\end{multline*}
The desired inequality follows for $2 \eta C_0 < \frac1{10}$ and $H > 1$.
\end{proof}

The next Proposition provides an accurate characterization of the behavior of a minimizing arc.

\begin{Proposition} \label{Prop:characterizationminimizinggamma}
Assume that $M$ satisfies condition (C).
There are constants $\eta^* > 0$ and $H^* < \infty$ such that if $\eta \leq \eta^*$ and $H \geq H^*$, the following holds:

Consider a combinatorially minimizing arc $\gamma : [0,1] \to \td{M}$.
Then
\begin{enumerate}[label=(\alph*)]
\item For every chamber $K \in \mathcal{K}$ and every column $E \in \mathcal{E}_K$, the preimage $\gamma^{-1} (E)$ is a connected interval, i.e. $\gamma$ does not exit and reenter $E$.
\item $\gamma$ intersects every wall $W \in \mathcal{W}$ at most twice.
Assume that $K, K' \in \mathcal{K}$ are two chambers that are adjacent to a wall $W \in \mathcal{W}$ from either side.
Then
\begin{enumerate}[label=(b\arabic*)]
\item If $\gamma$ intersects $W$ exactly once, then there is a unique column $E \in \mathcal{E}_K$ that is both adjacent to $W$ and intersects $\gamma$.
Moreover, for every column $E^* \in \mathcal{E}_K$ that also intersects $\gamma$, the minimal chain between $E$ and $E^*$ intersects $W$ in at most two columns.
\item If $\gamma$ intersects $W$ twice and its endpoints lie on the same side of $W$ as $K$, then within both intersections it stays inside a column $E' \in \mathcal{E}_{K'}$ adjacent to $W$.
Moreover, there are exactly two columns $E_1, E_2 \in \mathcal{E}_K$ that intersect $\gamma$ in this order and that are adjacent to $W$ and we have $\dist^H_K (E_1, E_2) > H$.
The arc $\gamma$ leaves $K$ through $W$ right after $E_1$ and reenters $K$ through $W$ right before $E_2$.
\item If $\gamma$ does not intersect $W$, but intersects two columns $E_1, E_2 \in \mathcal{E}_K$ that are both adjacent to $W$, then $\dist^H_K (E_1, E_2) < 3H$.
\end{enumerate}

\item Consider a chamber $K \in \mathcal{K}$ and let $E_1, \ldots, E_n \in \mathcal{E}_K$ be the columns of $K$ that $\gamma$ intersects in that order.
Then there are columns $E^*_1, \ldots, E^*_k \in \mathcal{E}_K$ such that:
\begin{enumerate}[label=(c\arabic*)]
\item $E^*_1 = E_1$ and $E^*_n = E_n$
\item $\dist^H_K (E^*_i, E_i) \leq 1$ for all $i = 1, \ldots, n$.
\item $E^*_1, \ldots, E^*_n$ are pairwise distinct and lie on the minimal chain between $E_1$ and $E_n$ in that order.
\item If $E^*_i \neq E_i$, then there are two walls $W, W' \subset \partial K$ that both intersect $\gamma$ twice such that $\gamma$ exits $K$ through $W'$ right after $E_{i-1}$, enters $K$ through $W'$ right before $E_i$, exits $K$ through $W$ right after $E_i$ and enters $K$ through $W$ right before $E_{i+1}$.
In particular $E_i$ does not lie on the minimal chain between $E_1, E_n$ and $E_i$ is not adjacent to $E_{i-1}$ or $E_{i+1}$.
\item If $E_i, E_{i+1}$ are adjacent, then $\gamma$ stays within $E_i \cup E_{i+1}$ between $E_i$ and $E_{i+1}$.
\item If $E_i, E_{i+1}$ are not adjacent, then there is a wall $W \subset \partial K$ such that $\gamma$ exits $K$ through $W$ right after $E_i$ and enters $K$ through $W$ right before $E_{i+1}$.
The columns $E_i, E^*_i, E^*_{i+1}, E_{i+1}$ lie in that order (some of these columns might be the same) on a minimal chain that runs along $W$.
\item If $i_1 < i_2$ and $E_{i_1}, E_{i_2}$ are adjacent to a common wall $W \subset \partial K$, then either $(E_{i_1}, \ldots, E_{i_2})$ form a minimal chain or $i_2 = i_1 + 1$ and $\gamma$ intersects $W$ right after $E_{i_1}$ and right before $E_{i_2}$.
\end{enumerate}
\end{enumerate}
\end{Proposition}

\begin{proof}
The proof uses induction on the combinatorial length $|\gamma|$ of $\gamma$.
The case $|\gamma| = 0$ is obvious, so assume that $|\gamma| > 0$ and that all assertions of the Proposition hold for all combinatorially shorter minimizing arcs.

Let $W \in \mathcal{W}$ be a wall and $K, K' \in \mathcal{K}$ the chambers that are adjacent to $W$ from either side.
We first check the first statement of assertion (b).
Assume that $\gamma$ intersects $W$ three times or more.
Then by assertion (b2) of the induction hypothesis applied to every subsegment of $\gamma$ that intersects $W$ exactly twice, we obtain that $\gamma$ stays within $K \cup K'$ between its first and last intersection with $W$.
This however contradicts Lemma \ref{Lem:gammainsidetwoKKs}.

Assertion (b2) follows similarly.
Assume that $\gamma$ intersects $W$ twice and that both endpoints lie on the same side of $W$ as $K$.
Let $E \in \mathcal{E}_K$ be the first column that $\gamma$ intersects in $K$ and $E^* \in \mathcal{E}_K$ the last.
By assertion (b1) of the induction hypothesis applied to the subsegment of $\gamma$ from $\gamma(0)$ to $E^*$, we find that $E = E^*$.
By assertion (a) of the induction hypothesis, $\gamma$ remains within $E$ between both intersections with $W$.
So we obtain again that $\gamma$ stays within $K \cup K'$ between its first and second intersection with $W$.
The rest follows with Lemma \ref{Lem:gammainsidetwoKKs}(b).

For assertion (b1), observe that the complete assertion (c) holds for $\gamma$ in the case in which $\gamma$ crosses a wall exactly once, because in this case assertion (c) is only concerned with proper subsegments of $\gamma$.
So consider the columns $E_1, \ldots, E_n, E^*_1, \ldots, E^*_n \in \mathcal{E}_K$.
Without loss of generality, we may assume that $\gamma(1)$ lies on the same side of $W$ as $K$.
This implies that $E_1$ is adjacent to $W$.
If $E_i$ for some $2 \leq i \leq n$ was adjacent to $W$ as well, then by assertions (c7) and (b2) the arc $\gamma$ must be contained in $K$ in between $E_1$ and $E_i$.
This is however impossible by Lemma \ref{Lem:gammainsidetwoKKs}(a) applied to the subsegment of $\gamma$ between the last column in $K'$ and $E_i$.
So the first part of (b1) holds.
Consider now the minimal chain between $E_1$ and some $E_i$ and assume that three of its columns are adjacent to $W$.
Those columns need to be the first three columns in this chain.
We can assume that $E_i = E_n$, because otherwise we could pass to a subsegment of $\gamma$.
Since by what we have already shown, $\gamma$ cannot intersect any column that is adjacent to $W$ other than $E_1$, it cannot happen that $E_2$ is adjacent to $E_1$ (compare with assertions (c4), (c5)).
So by assertion (c6) there is a wall $W' \subset \partial K$, $W' \neq W$ that is intersected twice by $\gamma$ in between $E_1$ and $E_2$.
So if $H$ is sufficiently large, assertion (c6) implies that the first three columns on the minimal chain between $E_1$ and $E^*_2$, i.e. the first three columns on the minimal chain between $E_1$ and $E_n$ are adjacent to both $W'$ and $W$.
This however contradicts Lemma \ref{Lem:EEinKisatree} and finishes the proof of assertion (b1).

We now show assertions (c1)--(c7), (a) and (b3).
Observe that by the induction hypothesis, it suffices to restrict our attention to the case in which $\gamma(0), \gamma(1) \in K$.
Consider now the columns $E_1, \ldots, E_n$ as defined in the proposition.
If $n \leq 2$, we are done with the choice $E^*_1 = E_1$ and $E^*_2 = E_2$ by assertion (b2) and assuming $H > 2$.
So assume that $n \geq 3$.
Assertion (c5) and the first part of (c6) follows immediately by passing to the subsegment between $E_i, E_{i+1}$ and using the induction hypothesis.
We will now distinguish the cases of when $E_{n-1}$ lies on the minimal chain between $E_1, E_n$ or not and establish assertions (c1)--(c4) and the second part of assertion (c6) in each case.
Based on these assertions we will next conclude assertion (c7) in both cases.

Consider first the case in which $E_{n-1}$ lies on the minimal chain between $E_1$ and $E_n$.
Then we can apply the induction hypothesis to the subsegment of $\gamma$ between $E_1$ and $E_{n-1}$ and obtain the columns $E^*_1, \ldots, E^*_{n-1}$ on the minimal chain between $E_1$ and $E_{n-1}$.
Moreover, we set $E^*_n = E_n$.
Assertions (c1)--(c6) follow immediately.

Next consider the case in which $E_{n-1}$ does not lie on the minimal chain between $E_1$ and $E_n$.
Define $E^*_1, \ldots, E^*_{n-2}$ using the induction hypothesis applied to the subsegment of $\gamma$ between $E_1$ and $E_{n-1}$.

Assume first that $E_n$ and $E_{n-1}$ are adjacent.
Then $E_n$ must lie on the minimal chain between $E_1$ and $E_{n-1}$ (by our assumption and the tree property, see Lemma \ref{Lem:EEinKisatree}).
So $E_{n-2}$ cannot be adjacent to $E_{n-1}$, because that would imply by assertion (c4) of the induction hypothesis that $E_{n-2} = E_n$ and it is elementary that $\gamma$ cannot reenter a column without exiting $K$.
This means (by assertion (c6) of the induction hypothesis) that there is a wall $W \subset \partial K$ that intersects $\gamma$ twice and that is adjacent to $E_{n-2}, E^*_{n-2}, E_{n-1}$ and hence also $E_n$.
This however contradicts assertion (b2).

So $E_n$ and $E_{n-1}$ are not adjacent and by assertion (b2) both columns are adjacent to a wall $W \subset \partial K$ such that $\gamma$ intersects $W$ right after $E_{n-1}$ and right before $E_n$.
By the tree property of $\mathcal{E}_K$ there is a column $E^* \in \mathcal{E}_K$ that lies on the three minimal chains between $E_{n-1}, E_n$ and $E_1, E_{n-1}$ and $E_1, E_n$.
So $E^*$ is adjacent to $W$ and horizontally lies between $E_{n-1}, E_n$.
By our earlier assumption $E^* \neq E_{n-1}$.
Assertion (b1) applied to a subsegment of $\gamma$ implies that $\dist^H_K (E^*, E_{n-1}) \leq 1$; so $E^*$ is adjacent to $E_{n-1}$.
If $E_{n-2}$ was adjacent to $E_{n-1}$, then $E_{n-2} = E^*$ contradicting assertion (b2).
So by assertion (c6) of the induction hypothesis $E_{n-2}, E^*_{n-2}, E_{n-1}$ are adjacent to a wall $W' \subset \partial K$ such that $\gamma$ intersects $W'$ twice between $E_{n-2}, E_{n-1}$.
This implies $W \neq W'$.
Set $E^*_{n-1} = E^*$ and $E^*_n = E_n$.
Assertions (c1)--(c3) follow immediately.
Assertion (c4) and the second part of (c6) hold with the walls $W, W'$ that we have just defined.

We now establish assertion (c7) in the general case (i.e. independently on whether $E_{n-1}$ lies on the minimal chain between $E_1$ and $E_n$ or not).
Assume that $(E_{i_1}, \ldots, E_{i_2})$ does not form a minimal chain.
Then $\gamma$ has to leave $K$ in between $E_{i_1}$ and $E_{i_2}$, i.e. by assertion (c5) there is a $j \in \{ i_1, \ldots, i_2 -1 \}$ such that $E_j, E_{j+1}$ are not adjacent and hence by assertions (c6) $\gamma$ has to intersect a wall $W' \subset \partial K$ in between $E_j$ and $E_{j+1}$.
The columns on the minimal chain between $E^*_j, E^*_{j+1}$ are adjacent to both $W$ and $W'$ and for $H > 10$ there are at least $3$ such columns.
So by Lemma \ref{Lem:EEinKisatree} $W = W'$ and by assertion (b2) we must have $E_{i_1} = E^*_j$, $E_{i_2} = E^*_{j+1}$.

Finally, assertion (a) is a direct consequence of assertion (c7) and assertion (b3) follows from assertion (c7) and Lemma \ref{Lem:gammainsidetwoKKs}(c).
\end{proof}

Next, we analyze the relative behavior of two combinatorially minimizing arcs.

\begin{Lemma} \label{Lem:distg1g2afterW}
There are constants $\eta^* > 0$ and $H^* < \infty$ such that if $\eta \leq \eta^*$ and $H \geq H^*$, then the following holds:

Let $\gamma_1, \gamma_2 : [0,1] \to \td{M}$ be two combinatorially minimizing arcs and consider a wall $W \in \mathcal{W}$ that is adjacent to two chambers $K, K' \in \mathcal{K}$ on either side.
Assume that $\gamma_1, \gamma_2$ intersect $W$ exactly once and that $\gamma_1(0), \gamma_2(0)$ lie in a common chamber on the same side of $W$ as $K$.
If that chamber is $K$, we additionally require that the cells that contain these points are vertically aligned.
Similarly, assume that $\gamma_1(1), \gamma_2(1)$ lie in a common chamber on the same side of $W$ as $K'$.
If that chamber is $K'$, we also require that the cells that contain these points are vertically aligned.

Let $Q_1, Q_2 \subset K$ be the cells that $\gamma_1, \gamma_2$ intersect right before crossing $W$ and let $Q'_1, Q'_2 \subset K'$ be the cells that $\gamma_1, \gamma_2$ intersect right after crossing $W$.
Then every pair of the cells $Q_1, Q_2, Q'_1, Q'_2$ has combinatorial distance bounded by $4$ or  $4 + H$ depending on whether they lie on the same side of $W$ or not.
\end{Lemma}

\begin{proof}
Let $E_1, E_2 \in \mathcal{E}_K$ be the columns that contain $Q_1, Q_2$.
We first show that $\dist^H_K (E_1, E_2) \leq 3$ (in fact, we can show that $\dist^H_K (E_1, E_2) \leq 1$, but we don't need this result here).

Define $E^*_1, E^*_2 \in \mathcal{E}_K$ to be the first columns in $K$ that are intersected by $\gamma_1, \gamma_2$.
In the case $\gamma_1 (0), \gamma_2 (0) \in K$ we have $E^*_1 = E^*_2$.
So in either case, we can find a wall $W^* \subset \partial K$ with $W^* \neq W$ that is adjacent to both $E^*_1, E^*_2$.
Consider the minimal chain between $E_1, E^*_1$ and let $E^{**}_1$ be the last column on that chain that is adjacent to $W$.
Define $E^{**}_2$ accordingly.
By Proposition \ref{Prop:characterizationminimizinggamma}(b1) $\dist^H_K (E_1, E^{**}_1), \dist^H_K(E_2, E^{**}_2) \leq 1$.
We need to show that $\dist^H_K (E^{**}_1, E^{**}_2) \leq 1$.

If $W, W^*$ are adjacent to a common column, then both $E^{**}_1, E^{**}_2$ have to be adjacent to $W^*$ since in that case a minimal chain between $E_1, E^*_1$ first runs along $W$ and then along $W^*$.
Hence in that case $\dist^H_K (E^{**}_1, E^{**}_2) \leq 1$ by Lemma \ref{Lem:EEinKisatree}.
If $W, W^*$ are not adjacent to a common column, we follow the minimal chain between $E^{**}_1, E^*_1$, then the minimal chain between $E^*_1, E^*_2$ (along $W^*$) and finally the minimal chain between $E^*_2, E^{**}_2$, to obtain a chain that connects $E^{**}_1$ with $E^{**}_2$ and that intersects $W$ only in its first and last column.
By the tree property of $\mathcal{E}_K$ this chain covers the minimal chain between $E^{**}_1, E^{**}_2$ and hence it has to include all columns along $W$ between $E_1^{**}, E_2^{**}$.
So $\dist^H_K (E^{**}_1, E^{**}_2) \leq 1$.

It follows that $\dist^H_K (Q_1, Q_2) = \dist^H_K (E_1, E_2) \leq 3$.
Analogously $\dist^H_K (Q'_1, Q'_2) \leq 3$.
It now follows from Lemma \ref{Lem:QQestimatesatW}(d) that $\dist^V_K (Q_1, Q_2), \dist^V_K (Q'_1, Q'_2) < C_0$.
This establishes the claim for $\eta < C_0^{-1}$.
\end{proof}

\begin{Lemma} \label{Lem:interchangeorderinK}
There are constants $\eta^* > 0$ and $H^*, C_1 < \infty$ such that if $\eta \leq \eta^*$ and $H \geq H^*$, the following holds: 

Let $K \in \mathcal{K}$ be a chamber and $Q_1, \ov{Q}_1, Q_2, \ov{Q}_2 \subset K$ be cells such that $Q_1, \ov{Q}_1$ and $Q_2, \ov{Q}_2$ are vertically aligned in $K$.
Assume that the vertical order of $Q_1, \ov{Q}_1$ is opposite to the one of $Q_2, \ov{Q}_2$ (i.e. $Q_1$ is ``above'' $\ov{Q}_1$ and $Q_2$ is ``below'' $\ov{Q}_2$ or the other way round).
Let $\gamma, \ov\gamma : [0,1] \to \td{M}$ be minimizing arcs from $Q_1$ to $Q_2$ and from $\ov{Q}_1$ to $\ov{Q}_2$.
Then we can find cells $Q', \ov{Q}' \subset K$ such that $Q'$ intersects $\gamma$, $\ov{Q}'$ intersects $\ov\gamma$ and such that $\dist^H_K (Q', \ov{Q}') < 3H$, $\dist^V_K (Q', \ov{Q}') < C_1 H$ and $\dist ( Q', \ov{Q}' ) < 4H$.
\end{Lemma}

\begin{proof}
Note that the last inequality follows from the first two inequalities if $\eta^* < C_1^{-1}$.

Let $E_0, E_\omega \in \mathcal{E}_K$ be the columns that contain $Q_1, \ov{Q}_1$ and $Q_2, \ov{Q}_2$.
We first invoke Proposition \ref{Prop:characterizationminimizinggamma}(c) on $\gamma$ to obtain columns $E_1, \ldots, E_n$, $E^*_1, \ldots, E^*_n \in \mathcal{E}_K$ with $E_1 = E^*_1 = E_0$ and $E_n = E^*_n = E_\omega$.
Then $E^*_1, \ldots, E^*_n$ lie on the minimal chain $L$ between $E_0$ and $E_\omega$.
Let $S \subset L \cup E_1 \cup \ldots \cup E_n$ be the union of all cells in $K$ that intersect $\gamma$ and all cells in $L \cup E_1 \cup \ldots \cup E_n$ that are adjacent to those cells outside $K$ which intersect $\gamma$.
Then $Q_1, Q_2 \subset S$ and by the results of Proposition \ref{Prop:characterizationminimizinggamma} these two cells lie in the same connected component of $S$.
Based on the set $S$ we construct another set $S' \subset L$ as follows: $S'$ is the union of $S \cap L$ with all cells in each $E^*_i$ that are horizontally aligned with a cell in $S \cap E_i$.
Then again $Q_1, Q_2 \subset S'$ and both cells lie in the same connected component of $S'$.
Similarly, we can invoke Proposition \ref{Prop:characterizationminimizinggamma}(c) on $\ov\gamma$, obtaining columns $\ov{E}_1, \ldots, \ov{E}_{\ov{n}} \in \mathcal{E}_K$ and $\ov{E}^*_1, \ldots, \ov{E}^*_{\ov{n}}$ on $L$ and we can define $\ov{S}$ and $\ov{S}'$ in the same way.
So $\ov{Q}_1, \ov{Q}_2 \subset \ov{S}'$ and both cells lie in the same connected component of $\ov{S}'$.
Since the cells on $L$ are arranged on a rectangular lattice and the cells $Q_1, \ov{Q}_1$ and $Q_2, \ov{Q}_2$ lie on opposite sides of $L$ and have opposite vertical order, we conclude that the sets $\ov{S}$ and $\ov{S}'$ have to intersect.
Let $Q^\circ \subset S' \cap \ov{S}'$ be a cell in the intersection and $E^\circ \subset L$ the column containing $Q^\circ$.
So we can find cells $Q, \ov{Q} \in \mathcal{Q}$ that intersect $\gamma, \ov\gamma$ such that the following holds:
Either $Q \subset L$ and $Q^\circ = Q$, or $Q \subset (E_1 \cup \ldots \cup E_n) \setminus \Int L$ and $Q$ is adjacent and horizontally aligned with $Q^\circ$, or $Q \not\subset K$ and $Q$ is either adjacent to $Q^\circ$ or $Q^\circ \subset E^*_i$ for some $i \in \{ 1, \ldots, n \}$ for which $E^*_i \neq E_i$ and $Q$ is adjacent to a cell in $E_i$ that is adjacent to $Q^\circ$ and horizontally aligned with it.
In the first two cases we set $Q' := Q$.
In the third case we will define $Q'$ later.
So if $Q \subset K$, then $Q'$ intersects $\gamma$ and $\dist^H_K (Q', Q^\circ) \leq 1$ and $\dist^V_K (Q', Q^\circ) = 0$.
The analogous characterization holds for $\ov{Q}$ and we define $\ov{Q}'$ in the same way if $\ov{Q} \subset K$.

We now consider the case in which $\dist^H_K (E^\circ, E^*_i) \leq 1$ for some $i \in \{ 1, \ldots, n \}$, and we establish the existence of a cell $Q' \subset K$ that intersects $\gamma$ and that is within bounded distance from $Q^\circ$.
If $Q \subset K$, then we are done by the previous paragraph.
So assume that $Q \not\subset K$.
Let $K' \in \mathcal{K}$ be the chamber that contains $Q$ and let $W = K \cap K' \in \mathcal{W}$ be the wall between $K$ and $K'$.
So $\gamma$ intersects $W$ twice and $E^\circ$ is adjacent to $W$.
Choose $i' \in \{1, \ldots, n-1 \}$ such that $\gamma$ intersects $W$ between $E_{i'}, E_{i'+1}$.
If $E^\circ$ lies between $E^*_{i'}, E^*_{i'+1}$, then $\dist^H_K (E^\circ, E^*_{i'}) \leq 1$ or $\dist^H_K (E^\circ, E^*_{i'+1}) \leq 1$, by our initial assumption.
If $E^\circ$ lies on $L$ not between $E^*_{i'}, E^*_{i'+1}$, then we can conclude by applying Proposition \ref{Prop:characterizationminimizinggamma}(b1) to subsegments of $\gamma$ which intersect $W$ exactly once, that we still have $\dist^H_K (E^\circ, E^*_{i'}) \leq 1$ or $\dist^H_K (E^\circ, E^*_{i'+1}) \leq 1$ (note that there is an $i'' \neq i', i'+1$ such that the minimal chain from $E_{i'}$ or $E_{i'+1}$ to $E_{i''}$ contains the columns $E^*_{i'}$ or $E^*_{i'+1}$ and $E^\circ$ in that order).
So in either case $\dist^H_K ( E^\circ, E^*_{i'}) \leq 1$ or $\dist^H_K (E^\circ, E^*_{i'+1}) \leq 1$ and, after possibly rechoosing $i$, we may assume that $i = i'$ or $i = i'+1$.
Let now $Q' \subset E_i$ be the cell that is intersected by $\gamma$ right before or right after $W$, depending on whether $i = i'$ or $i = i'+1$.
Then $\dist^H_K (Q^\circ, Q') \leq 2$ and by Lemma \ref{Lem:QQestimatesatW}(b) we get $\dist^V_K (Q^\circ, Q') < 3C_0 + 1$.

Combining the previous conclusion with the analogous conclusion for $\ov\gamma$ and the triangle inequality yields the desired result in the case in which there are indices $i \in \{ 1, \ldots, n \}$ and $\ov{i} \in \{ 1, \ldots, \ov{n} \}$ such that $\dist^H_K (E^\circ, E^*_i) \leq 1$ and $\dist^H_K (E^\circ, \ov{E}^*_{\ov{i}}) \leq 1$.
So, after possibly interchanging the roles of $\gamma$ and $\ov\gamma$, it remains to consider the case in which there is an index $i \in \{ 1, \ldots, n-1 \}$ such that $E^\circ$ lies strictly in between $E^*_i, E^*_{i+1}$ and such that $E^\circ$ is not adjacent to either of these columns.
We will henceforth always assume that.
Let $W \subset \partial K$ be the wall that $\gamma$ intersects between $E_i, E_{i+1}$ and let $K' \in \mathcal{K}$ be the chamber on the other side.
Then $E_i, E^*_i, E^\circ, E^*_{i+1}, E_{i+1}$ are arranged along $W$ in that order and by Lemma \ref{Lem:EEinKisatree} we must have $Q \subset K'$ (note that every wall that intersects $\gamma$ twice is adjacent to a subchain of $L$ that is not contained in the chain between $E_i^*$ and $E^*_{i+1}$); let $E \in \mathcal{E}_{K'}$ be the column that contains $Q$.
Finally, let $Q' \subset E_i$ be the cell that $\gamma$ intersects right before $W$.

Consider the columns on $L$ between $\ov{E}^*_{\ov{i}}, \ov{E}^*_{\ov{i}+1}$ for each $\ov{i} = 1, \ldots, \ov{n}-1$.
If for some $\ov{i}$ there are at least $3$ such columns that are also (not strictly) between $E^*_i$ and $E^*_{i+1}$, we must have $\dist^H_K (\ov{E}^*_{\ov{i}}, \ov{E}^*_{\ov{i}+1}) > 1$ and all columns between $\ov{E}^*_{\ov{i}}$ and $\ov{E}^*_{\ov{i}+1}$ have to be adjacent to $W$ by Lemma \ref{Lem:EEinKisatree}.
However, this situation can only occur for at most one index $\ov{i}$.
So there are two different cases:
Either there is no such $\ov{i}$ and hence all columns that are strictly between $E^*_i$ and $E^*_{i+1}$ are contained in $\ov{E}^*_1 \cup \ldots \cup \ov{E}^*_{\ov{n}}$.
In this case all columns of this union that are strictly between $E^*_i$ and $E^*_{i+1}$ have to be adjacent to one another, and hence by Proposition \ref{Prop:characterizationminimizinggamma}(c4) $\gamma$ has to intersect all these columns.
In the second case there is exactly one such $\ov{i}$ and each column that is strictly between $E^*_i$ and $E^*_{i+1}$ either lies (not strictly) between $\ov{E}^*_{\ov{i}}$ and $\ov{E}^*_{\ov{i}+1}$ or in one of the unions $\ov{E}^*_0 \cup \ldots \cup \ov{E}^*_{\ov{i}-1}$ and $\ov{E}^*_{\ov{i}+1} \cup \ldots \cup \ov{E}^*_{\ov{i}+1}$.
Those two subsets can only cover the remaining columns, which are strictly between $E^*_i$ and $E^*_{i+1}$, if the columns of $\ov{E}^*_0, \ldots, \ov{E}^*_{\ov{n}}$ that lie strictly between $E^*_i$ and $E^*_{i+1}$ are adjacent to one another, which implies that either $\ov{E}^*_{\ov{i}}$ is adjacent to $\ov{E}^*_{\ov{i}-1}$ and/or $\ov{E}^*_{\ov{i}+1}$ is adjacent to $\ov{E}^*_{\ov{i+1}}$, depending on the side on which the remaining columns lie.
So we conclude using Proposition \ref{Prop:characterizationminimizinggamma}(c4) that if not all columns that are strictly between $E^*_i$ and $E^*_{i+1}$ are also (not strictly) between $\ov{E}^*_{\ov{i}}$ and $\ov{E}^*_{\ov{i}+1}$, then $\ov{E}^*_{\ov{i}}$ and/or $\ov{E}^*_{\ov{i}+1}$ lies strictly between and we have $\ov{E}^*_{\ov{i}} = \ov{E}_{\ov{i}}$ and/or $\ov{E}^*_{\ov{i}+1} = \ov{E}_{\ov{i}+1}$, respectively.
Now by Proposition \ref{Prop:characterizationminimizinggamma}(b1) applied to the minimal chain between $E_0$ and $\ov{E}_{\ov{i}}$ and/or the minimal chain between $E_n$ and $\ov{E}_{\ov{i}+1}$, we conclude that $\ov{E}_{\ov{i}}$ has to be adjacent to $E^*_i$ and/or $\ov{E}_{\ov{i}+1}$ has to be adjacent to $E^*_{\ov{i}+1}$, depending on which of these columns lie strictly between $E^*_i$ and $E^*_{i+1}$.
So, to summarize our findings: Either all columns that are strictly between $E^*_i$ and $E^*_{i+1}$ intersect $\gamma$ or they lie (not strictly) between $\ov{E}^*_{\ov{i}}$ and $\ov{E}^*_{\ov{i}+1}$.
In the second case, we can apply the same argument reversing the roles of $\gamma$ and $\ov\gamma$ to conclude that there is no other index $i' \in \{ 1, \ldots, n-1 \}$, $i' \neq i$ such that there are more than $2$ columns that are between $E^*_{i'}, E^*_{i'+1}$ and $\ov{E}^*_{\ov{i}}, \ov{E}^*_{\ov{i}+1}$.
This implies that $\dist^H_K (E^*_i, \ov{E}^*_{\ov{i}}), \dist^H_K (E^*_{i+1}, \ov{E}^*_{\ov{i}+1}) \leq 1$ in the second case.

In the first case, we use Proposition \ref{Prop:characterizationminimizinggamma}(b3) and (c7) to find that $\ov\gamma$ intersects fewer than $3H$ columns that are adjacent to $W$.
So $\dist^H_K (\ov{E}^*_{\ov{i}}, \ov{E}^*_{\ov{i}+1}) \leq 3H$.
Since $\ov\gamma$ intersects $E^\circ$ we have $\ov{Q}' := \ov{Q} = Q^\circ$.
Hence $\dist^H_K (Q', \ov{Q}') < 3H$ and Lemma \ref{Lem:QQestimatesatW}(b) yields $\dist^V_K (Q', \ov{Q}') < 3C_0H + C_0$ and we are done.

In the second case, $E^\circ$ lies strictly between $\ov{E}^*_{\ov{i}}, \ov{E}^*_{\ov{i}+1}$.
So $\ov{Q} \not\subset K$ and $\ov\gamma$ intersects $W$ between $\ov{E}_{\ov{i}}, \ov{E}_{\ov{i}+1}$ (by Lemma \ref{Lem:EEinKisatree}).
Let $\ov{K}' \in \mathcal{K}$ be the chamber that contains $\ov{Q}$ and $\ov{W} = K \cap \ov{K}' \in \mathcal{W}$ the wall between $K$ and $\ov{K}'$.
We will now show that $\ov{K}' = K'$ and $\ov{W} = W$.
If not, then there must be an index $\ov{i}' \in \{ 1, \ldots, n-1 \}, \ov{i}' \neq \ov{i}$ such that $\ov{W}$ is adjacent to $\ov{E}^*_{\ov{i}'}, \ov{E}^*_{\ov{i}'+1}, E^\circ$.
If $\ov{i}' < \ov{i}$, then $\ov{W}$ is also adjacent to $E^*_i$ (which is then between $\ov{E}^*_{\ov{i}'}$ and $E^\circ$), contradicting Lemma \ref{Lem:EEinKisatree}.
If $\ov{i}' > \ov{i}$, then $\ov{W}$ is also adjacent to $E^*_{i+1}$, contradicting Lemma \ref{Lem:EEinKisatree} as well.
So indeed $Q, \ov{Q} \subset \ov{K}' = K'$; let $\ov{E} \in \mathcal{E}_{K'}$ be the column that contains $\ov{Q}$.
Let now $\ov{Q}' \subset \ov{E}_{\ov{i}}$ be the cell that $\ov\gamma$ intersects right before $W$.

Recall that $Q' \subset E_i$ and $\ov{Q}'\subset \ov{E}_{\ov{i}}$, that $Q'$ is adjacent to $E$, $\ov{Q}'$ is adjacent to $\ov{E}$ and that $E, \ov{E}$ are both adjacent to $Q^\circ$.
Moreover, by our previous conclusions $\dist^H_K (Q', \ov{Q}') \leq 3$.
Let $Q'' \subset E_i$ be a cell that is adjacent to $\ov{E}$.
Then by Lemma \ref{Lem:QQestimatesatW}(b) $\dist^V_K (\ov{Q}', Q'') < 4C_0$ and by Lemma \ref{Lem:QQestimatesatW}(c) $\dist^V_K (Q'', Q') < C_0$.
Hence $\dist^V_K (Q', \ov{Q}') < 5C_0$.
This finishes the proof of the Lemma.
\end{proof}

The next Lemma is a preparation for the combinatorial convexity estimate stated in Proposition \ref{Prop:combinatorialconvexityseveralSeifert}. 

\begin{Lemma} \label{Lem:combconvexseveralSeifertpreparation}
There are constants $\eta^* > 0$ and $H^* < \infty$ such that if $\eta \leq \eta^*$ and $H \geq H^*$, then the following holds:

Let $K \in \mathcal{K}$ be a chamber and $Q_0, Q_1, Q_2 \subset K$ cells such that $Q_1$ and $Q_2$ are vertically aligned.
Assume that $\dist (Q_0, Q_1), \dist(Q_0, Q_2) \leq R$ for some $R \geq 0$.
Then for any cell $Q^* \subset K$ between $Q_1$ and $Q_2$, we have $\dist (Q_0, Q^*) < R + 8H$.
\end{Lemma}

\begin{proof}
We prove this Lemma by induction on $R$ (observe that we are only interested in a discrete set of values of $R$) and then on $\dist^V_K(Q_1,Q_2)$.
Consider the action $\varphi : \IZ \curvearrowright \td{M}$ by deck transformations of the universal covering $\td{M} \to M$ that acts as a vertical shift on $K$, leaving $\td{V}$ and hence the cell structure and combinatorial distance function invariant and choose $z \in \IZ$ such that $Q^* = \varphi_z (Q_1)$.
We may assume $z \neq 0$.

Without loss of generality, we can assume that $Q^*$ lies between $Q_1$ and $Q^{**} =\varphi_{-z} (Q_2)$.
Otherwise, we can interchange the roles of $Q_1$ and $Q_2$.
Let $\gamma_1, \gamma_2$ be minimizing arcs between $Q_0$ and $Q_1, Q_2$.
We can now apply Lemma \ref{Lem:interchangeorderinK} to $\varphi_z \circ \gamma_1$ and $\gamma_2$ to obtain cells $Q'_1, Q'_2 \subset K$ on $\varphi_z \circ \gamma_1$ and $\gamma_2$ with $\dist (Q'_1, Q'_2) < 4 H$.
Then $\varphi_{-z} (Q'_1)$ lies on $\gamma_1$ and hence
\begin{equation} \label{eq:Q0phiQ1}
 \dist (Q_0, \varphi_{-z} (Q'_1) ) + \dist (\varphi_{-z} (Q'_1), Q_1) = \dist (Q_0, Q_1) \leq R.
\end{equation}
We also have
\begin{equation} \label{eq:Q0Qs2Q2}
 \dist (Q_0, Q'_2) + \dist (Q'_2, Q_2) = \dist (Q_0, Q_2) \leq R.
\end{equation}
If $\dist (Q_0, Q'_2) + \dist (\varphi_{-z} (Q'_1), Q_1) \leq R + 4H$, then
\[ \dist (Q_0, Q^*) \leq \dist(Q_0, Q'_2) + \dist (Q'_2, Q'_1) + \dist (Q'_1, \varphi_z (Q_1) )
< R + 8H, \]
which proves the desired estimate.
On the other hand, assume that $\dist (Q_0, Q'_2) + \dist (\varphi_{-z} (Q'_1), Q_1) > R + 4H$.
Then (\ref{eq:Q0phiQ1}) and (\ref{eq:Q0Qs2Q2}) give us
\[ \dist (Q_0, \varphi_{-z} (Q'_1)) + \dist (Q'_2, Q_2) < R - 4H . \]
It follows that
\begin{multline*}
 \dist (\varphi_{-z} (Q'_1), Q^{**}) = \dist (Q'_1, Q_2) \leq \dist (Q'_1, Q'_2) + \dist(Q'_2, Q_2) \\
 < 4 H + R - 4H - \dist (Q_0, \varphi_{-z} (Q'_1)) = R - \dist (Q_0, \varphi_{-z} (Q'_1)).
 \end{multline*}
Also by (\ref{eq:Q0phiQ1})
\[ \dist (\varphi_{-z} (Q'_1), Q_1) \leq R - \dist (Q_0, \varphi_{-z} (Q'_1)). \]
So by the induction hypothesis, we find that
\[ \dist (\varphi_{-z}  (Q'_1), Q^*) < R - \dist (Q_0, \varphi_{-z} (Q'_1)) + 8H. \]
This implies
\[ \dist (Q_0, Q^*) \leq \dist (Q_0, \varphi_{-z} (Q'_1)) + \dist (\varphi_{-z} (Q'_1), Q^*) < R + 8 H. \qedhere \]
\end{proof}

\begin{Proposition} \label{Prop:combinatorialconvexityseveralSeifert}
Assume that $M$ satisfies condition (C).
There are constants $\eta^* > 0$ and $H^* < \infty$ such that whenever $\eta \leq \eta^*$ and $H \geq H^*$, then the following holds:

Consider a cell $Q_0 \in \mathcal{Q}$, a chamber $K \in \mathcal{K}$ (not necessarily containing $Q_0$) and cells $Q_1, Q_2 \subset K$ that are vertically aligned within $K$.
Assume that $\dist (Q_0, Q_1)$, $\dist(Q_0, Q_2) \leq R$ for some $R \geq 0$.
Then for any cell $Q^* \subset K$ that is vertically aligned with $Q_1, Q_2$ and vertically between $Q_1$ and $Q_2$, we have $\dist (Q_0, Q^*) < R + 10 H$.
\end{Proposition}

\begin{proof}
If $Q_0 \subset K$, we are done by the previous Lemma.
So assume that $Q_0$ lies outside of $K$ and let $\gamma_1, \gamma_2$ be minimizing arcs from $Q_0$ to $Q_1, Q_2$.

Then there is a unique wall $W \subset \partial K$ through which both $\gamma_1$ and $\gamma_2$ enter $K$.
Let $Q'_1, Q'_2 \subset K$ be the first cells in $K$ that are intersected by $\gamma_1, \gamma_2$.
So both cells are adjacent to $W$.
By Lemma \ref{Lem:distg1g2afterW} we know that $\dist(Q'_1, Q'_2) \leq 4$.

So
\[ \dist (Q'_1, Q_1) \leq R - \dist (Q_0, Q'_1) \]
and
\begin{multline*}
 \dist (Q'_1, Q_2) \leq \dist (Q'_1, Q'_2) + \dist (Q'_2, Q_2) 
 \leq 4 + R - \dist (Q_0, Q'_2) \\
 \leq 4 + R - \dist(Q_0, Q'_1) + \dist(Q'_1, Q'_2)  \leq R + 8 - \dist (Q_0, Q'_1) .
\end{multline*}
We can no apply Lemma \ref{Prop:combinatorialconvexityseveralSeifert} to obtain
\[ \dist(Q'_1, Q^*) < R + 8 - \dist(Q_0, Q'_1) + 8H . \]
So $\dist(Q_0, Q^*) < R + 10 H$ for $H > 4$.
\end{proof}

\subsection{A combinatorial convexity estimate if $M$ satisfies condition (B)} \label{subsec:combconvexincaseB}
Assume now that $M$ satisfies condition (B) in Proposition \ref{Prop:easiermaincombinatorialresultCaseb}, i.e. that $M$ is the total space of an $S^1$-bundle over a closed, orientable surface of genus $\geq 2$.
In this setting we will establish the same combinatorial convexity estimate as in Proposition \ref{Prop:combinatorialconvexityseveralSeifert}.
It will be stated in Proposition \ref{Prop:combinatorialconvexityCaseB}.
Its proof will resemble the proof in the previous subsection, except that most Lemmas will be simpler.

We first let $\mathcal{E} = \bigcup_{K \in \mathcal{K}} \mathcal{E}_K$ be the set of all columns of $\td{M}$.
We say that two columns $E_1, E_2 \in \mathcal{E}$ are \emph{adjacent} their intersection contains a point of $\td{V} \setminus \td{V}^{(1)}$.
In other words, $E_1, E_2$ are adjacent if and only if we can find cells $Q_1 \subset E_1$, $Q_2 \subset E_2$ such that $Q_1, Q_2$ are adjacent.
Observe that in the setting of condition (B) every column $E_1 \in \mathcal{E}$ is adjacent to only finitely many columns $E_2 \in \mathcal{E}$ and every wall $W \in \mathcal{W}$ intersects its adjacent columns $E \in \mathcal{E}$ from either side in parallel strips $E \cap W$ (see Lemma \ref{Lem:QQestimatesatW}).

The first Lemma is an analog of Lemma \ref{Lem:gammainsidetwoKKs}.

\begin{Lemma} \label{Lem:gamminKKsCaseB}
There are constants $\eta^* > 0$ and $H^* < \infty$ such that if $\eta \leq \eta^*$ and $H \geq H^*$, the following holds:

Consider two chambers $K, K' \in \mathcal{K}$ that are adjacent to one another across a wall $W = K \cap K' \in \mathcal{W}$.
Assume that $\gamma : [0,1] \to \td{M}$ is combinatorially minimizing and that its image is contained in $K \cup K'$.
Then the following holds:
\begin{enumerate}[label=(\alph*)]
\item The arc $\gamma$ intersects $W$ at most twice and $\gamma$ does not reenter any column, i.e. $\gamma^{-1} (E)$ is an interval for all $E \in \mathcal{E}$.
\item In the case in which $\gamma$ intersects $W$ exactly once, the following is true:
Then the columns on $\gamma$ that are adjacent to $W$ form two minimal chains in $K$ and $K'$, moving in the same direction, which are adjacent to one another in a unique pair of columns $E \in \mathcal{E}_K$ and $E' \in \mathcal{E}_{K'}$.
\item In the case in which $\gamma$ intersects $W$ exactly twice, the following is true:
Assume that $\gamma(0), \gamma(1) \in K$.
Let $E_1, E_2 \in \mathcal{E}_K$ be the columns that $\gamma$ intersects right before and after $W$.
Then $\gamma$ does not intersect any column of $K$ that is adjacent to $W$ and that horizontally lies strictly between $E_1$ and $E_2$.
Moreover, $\dist^H_K(E_1, E_2) > H$.
\end{enumerate}
\end{Lemma}

\begin{proof}
First note that every subsegment of $\gamma$ that does not intersect $W$ and whose endpoints lie in columns that are adjacent to $W$, stays within columns that are adjacent to $W$ and does not reenter any column.
So we can restrict our attention to the case in which $\gamma$ intersects only columns that are adjacent to $W$.

Assume first that we are in the setting of assertion (b), meaning that $\gamma$ intersects $W$ exactly once and assume without loss of generality that $\gamma(0) \in K$.
Let $E_1, \ldots, E_n \in \mathcal{E}_K$ and $E'_1, \ldots, E'_{n'} \in \mathcal{E}_{K'}$ be the columns that $\gamma$ intersects in that order.
Then both sequences of columns form minimal chains, which move along $W$, and $E_n, E'_1$ are adjacent across $W$.
We now show that $E_i$ can only be adjacent to $E'_{i'}$ if $i = n$ and $i' = 1$.
This will also imply that the directions of both minimal chains agree.
Assume that this was not the case and assume without loss of generality that $E_i$ is adjacent to $E'_{i'}$ for some $i < n$ and $i' \geq 1$ (otherwise we reverse the orientation of $\gamma$).
Let $Q_1 \subset E_1$ be the cell that contains $\gamma(0)$, $Q_2 \subset E_n$, $Q_3 \subset E'_1$ the cells that $\gamma$ intersects right before and after $W$ and $Q_4 \subset E_{i'}$ a cell that intersects $\gamma$.
Choose moreover a cell $Q^* \subset E_i$ which is adjacent to $Q_4$.
By Lemma \ref{Lem:QQestimatesatW}(e)
\[ \dist^V_K (Q_2, Q^*) < \dist^V_{K'} (Q_3, Q_4) + C_0 \dist^H_K (Q_2, Q^*) + C_0. \]
So
\begin{multline*}
 \dist^V_K (Q_1, Q^*) \leq \dist^V_K (Q_1, Q_2) + \dist^V_K (Q_2, Q^*) \\
  < \dist^V_K (Q_1, Q_2) + \dist^V_{K'} (Q_3, Q_4) + C_0 \dist^H_K (Q_2, Q^*) + C_0.
\end{multline*}
Hence
\begin{multline*}
 \dist(Q_1, Q_4) \leq \dist(Q_1, Q^*) + H < \dist^H_K (Q_1, Q^*) + \eta \dist^V_K (Q_1, Q_2) \\ + \eta \dist^V_{K'} (Q_3, Q_4) + \eta C_0 \dist^H_K (Q_2, Q^*) + \eta C_0 + H.
\end{multline*}
On the other hand, the minimizing property of $\gamma$ yields
\begin{multline*}
 \dist(Q_1, Q_4) = \dist(Q_1, Q_2) + H + \dist (Q_3, Q_4) \\ \geq \dist^H_K (Q_1, Q_2) + \eta \dist^V_K (Q_1, Q_2) + \eta \dist^V_{K'} (Q_3, Q_4) + H
\end{multline*}
Combining both inequalities yields
\[ \dist^H_K (Q_1, Q_2) < \dist^H_K (Q_1, Q^*) + \eta C_0 \dist^H_K (Q_2, Q^*) + \eta C_0. \]
Since $\dist^H_K (Q_1, Q_2) = n-1$, $\dist^H_K (Q_1, Q^*) = i-1$ and $\dist^H_K (Q_2, Q^*) = n-i \geq 1$, we obtain
\[ n -1 < i - 1 + \eta C_0 (n-i) + \eta C_0. \]
This yields a contradiction if $\eta < (2C_0)^{-1}$.
So assertion (b) holds.

Assume next that we are in the setting of assertion (c), meaning that $\gamma$ intersects $W$ exactly twice and that $\gamma(0), \gamma(1) \in K$.
Define $E_1, E_2 \in \mathcal{E}_K$ as in the statement of the Lemma.
We now establish the bound $\dist^H_K (E_1, E_2) > H$ (for sufficiently small $\eta$ and large $H$).
By assertion (b) applied to subsegments of $\gamma$, we find that $\gamma$ cannot intersect any column of $K$ that is adjacent to $W$ and lies strictly between $E_1, E_2$.
Let now $Q_1 \subset E_1$ and $Q_2 \subset E_2$ be the cells that $\gamma$ intersects right before and after $W$ and let $Q'_1, Q'_2 \subset K'$ be the cells that $\gamma$ intersects right after $Q_1$ and right before $Q_2$.
Then
\[ \dist(Q_1, Q_2) = 2H + \dist(Q'_1, Q'_2) \geq 2H  + \eta \dist^V_{K'} (Q'_1, Q'_2). \]
By Lemma \ref{Lem:QQestimatesatW}(e)
\[  \dist^V_K (Q_1, Q_2) < \dist^V_{K'} (Q'_1, Q'_2) + C_0 \dist^H_K (Q_1, Q_2) + C_0. \]
So
\[ \dist(Q_1, Q_2) < \dist^H_K (Q_1, Q_2)  + \eta \dist^V_{K'} (Q'_1, Q'_2) + \eta C_0 \dist^H_K (Q_1, Q_2) + \eta C_0. \]
Hence
\[ 2 H < (1 + \eta C_0) \dist^H_K (Q_1, Q_2) + \eta C_0. \]
Assertion (c) follows for $H > 10$ and $\eta < (2 C_0)^{-1}$.

It remains to show assertion (a).
To do this, we first show that $\gamma$ cannot intersect $W$ more than twice.
Assume it does.
By passing to a subsegment and possibly interchanging the roles of $K$ and $K'$, we can assume that $\gamma$ intersects $W$ exactly three times and that $\gamma(0) \in K$.
Let $Q_1, Q_2, Q_3 \subset K$ be the cells in $K$ that $\gamma$ intersects before the first, after the second and before the third intersection with $W$ and let $Q'_1, Q'_2, Q'_3 \subset K'$ be the cells of $K'$ that $\gamma$ intersects after the first, before the second and after the third intersection with $W$.
Then
\begin{multline} \label{eq:distQ1Qs3CaseB}
 \dist(Q_1, Q'_3) = 3H + \dist^H_{K'} (Q'_1, Q'_2) + \eta \dist^V_{K'} (Q'_1, Q'_2) \\
  + \dist^H_K (Q_2, Q_3) + \eta \dist^V_K (Q_2, Q_3).
\end{multline}

Let $Q^* \subset K$ be the cell that is adjacent to $W$ and that is located relatively to $Q_1, Q_2, Q_3$ such that $Q_1, Q_2, Q_3, Q^*$ forms a ``parallelogram'', i.e. 
\begin{multline*}
\dist^H_K (Q_1, Q^*) = \dist^H_K (Q_2, Q_3), \qquad \dist^V_K (Q_1, Q^*) = \dist^V_K (Q_2, Q_3), \\ 
\dist^H_K (Q^*, Q_3) = \dist^H_K (Q_1, Q_2), \qquad \dist^V_K (Q^*, Q_3) = \dist^V_K (Q_1, Q_2)
\end{multline*}
in an oriented sense.
Let moreover $Q^{* \prime} \subset K'$ be a cell that is adjacent to $Q^*$.
Then by Lemma \ref{Lem:QQestimatesatW}(g)
\[ \dist^H_{K'} (Q^{* \prime}, Q'_3 ) < \dist^H_{K'} (Q'_1, Q'_2) + C_0, \quad \dist^V_{K'} (Q^{* \prime}, Q'_3 ) < \dist^V_{K'} (Q'_1, Q'_2) + C_0. \]
So
\begin{multline*}
\dist(Q_1, Q'_3) \leq \dist(Q_1, Q^*) + \dist (Q^*, Q^{* \prime}) + \dist (Q^{* \prime}, Q'_3) \\
\leq \dist^H_K (Q_1, Q^*) + \eta \dist^V_K (Q_1, Q^*) + H + \dist^H_{K'} (Q^{* \prime}, Q'_3) + \eta \dist^V_{K'} (Q^{* \prime}, Q'_3) \\
< H + \dist^H_K (Q_2, Q_3) + \eta \dist^V_K (Q_2, Q_3) + \dist^H_{K'} (Q'_1, Q'_2) + C_0 \\ + \eta \dist^V_{K'} (Q'_1, Q'_2) + \eta C_0.
\end{multline*}
Together with (\ref{eq:distQ1Qs3CaseB}) this yields
\[ 2H < C_0 + \eta C_0 \]
and hence a contradiction for $H > C_0$ and $\eta < 1$.

Next, we show that $\gamma$ does not reenter any column.
If $\gamma$ does not intersect $W$, then this fact is a consequence of Lemma \ref{Lem:gammainsidesingleK}.
The same is true if $\gamma$ intersects $W$ exactly once, by passing to a subsegment.
So it remains to consider the case in which $\gamma$ intersects $W$ exactly twice.
Assume that the assertion was wrong.
By passing to a subsequent, we can assume that there is a column $E \in \mathcal{E}_K$ such that $\gamma(0), \gamma(1) \in E$ and that $\gamma$ intersects $W$ exactly twice.
Let $E_1, E_2 \in \mathcal{E}_K$ be the columns that $\gamma$ intersects right before and after $W$, as in the last part of the Lemma.
Let moreover $E'_1 = E, \ldots, E'_{n'} = E_1 \in \mathcal{E}_K$ and $E''_{1} = E_2, \ldots, E''_{n''} = E \in \mathcal{E}_K$ be the columns of $K$ that $\gamma$ intersects in that order.
By Lemma \ref{Lem:gammainsidesingleK} we know that $L_1 = (E'_1, \ldots, E'_{n'})$ and $L_2 = (E''_{n''}, \ldots, E''_1)$ form minimal chains between $E, E_1$ and $E, E_2$, respectively.
Let $L$ be the minimal chain between $E_1$ and $E_2$.
By assertion (c) of this Lemma $L_1, L$ and $L_2, L$ only intersect in $E_1$ and $E_2$, respectively.
So $L_1 \cup L$ is a minimal chain between $E$ and $E_2$.
By the tree property of $\mathcal{E}_K$, we must have $L_1 \cup L = L_2$, which is impossible, since $L_2$ intersects $L$ only in $E_2$.
This finishes the proof.
\end{proof}

The following Proposition and its proof are similar to Proposition \ref{Prop:characterizationminimizinggamma}.

\begin{Proposition} \label{Prop:characteriztionminimizinggammaCaseB}
Assume that $M$ satisfies condition (B).
There are constants $\eta^* > 0$ and $H^* < \infty$ such that if $\eta \leq \eta^*$ and $H \geq H^*$, the following holds:

Consider a combinatorially minimizing arc $\gamma : [0,1] \to \td{M}$.
Then
\begin{enumerate}[label=(\alph*)]
\item For every column $E \in \mathcal{E}$, the preimage $\gamma^{-1} (E)$ is an interval.
\item $\gamma$ intersects every wall $W \in \mathcal{W}$ at most twice.
Assume that $K, K' \in \mathcal{K}$ are two chambers that are adjacent to a wall $W \in \mathcal{W}$ from either side.
Then
\begin{enumerate}[label=(b\arabic*)]
\item If $\gamma$ intersects $W$ exactly once then the following holds:
Assume that $\gamma(0)$ lies on the same side of $W$ as $K$.
Let $E \in \mathcal{E}_{K}$ be the first column that is intersected by $\gamma$ and that is adjacent to $W$.
Then for every column $E^* \in \mathcal{E}_K$ that $\gamma$ intersects before $E$, the minimal chain between $E^*$ and $E$ intersects $W$ in at most two columns.
\item If $\gamma$ intersects $W$ exactly twice and its endpoints lie on the same side of $W$ as $K$, then $\gamma$ stays within $K'$ between both intersections and the columns $E_1, E_2 \in \mathcal{E}_K$ that $\gamma$ intersects right before and after $W$ satisfy $\dist^H_K (E_1, E_2) > H$.
Moreover, $\gamma$ does not intersect any column of $K$ that is adjacent to $W$ and that horizontally lies strictly between $E_1$ and $E_2$.
\item If $\gamma$ intersects two columns $E_1, E_2 \in \mathcal{E}$ that are both adjacent to $W$, then $\gamma$ stays within $K \cup K'$ in between $E_1, E_2$ and only intersects columns that are adjacent to $W$.
\end{enumerate}
\item Consider a chamber $K \in \mathcal{K}$ and let $E_1, \ldots, E_n \in \mathcal{E}_K$ be the columns of $K$ that $\gamma$ intersects in that order.
Then there are columns $E^*_1, \ldots, E^*_n \in \mathcal{E}_K$ such that assertions (c1)--(c6) of Proposition \ref{Prop:characterizationminimizinggamma} hold.
\end{enumerate}
\end{Proposition}

\begin{proof}
We use again induction on the combinatorial length $|\gamma|$ of $\gamma$.
Assume that $|\gamma| > 0$, since for $|\gamma|=0$ there is nothing to prove.
The first part of assertion (b) follows as in the proof or Proposition \ref{Prop:characterizationminimizinggamma}.

We now establish assertion (b1).
So assume that $\gamma$ intersects $W$ exactly once and that $\gamma(0)$ lies on the same side of $W$ as $K$ and consider the columns $E, E^* \in \mathcal{E}_K$.
Note that it suffices to show that the second or third last element of the minimal chain between $E^*$ and $E$ is not adjacent to $W$.
Apply assertions (b) and (c) of the induction hypothesis to the subsegment of $\gamma$ between $E^*$ and $E$.
We obtain sequences $E_1, \ldots, E_n$ and $E^*_1, \ldots, E^*_n$ with $E_1 = E^*_1 = E^*$ and $E_n = E^*_n = E$.
If $E_{n-1}, E_n$ are adjacent, then $E_{n-1} = E^*_{n-1}$ lies on the minimal chain between $E^*$ and $E$ and by assumption $E_{n-1}$ cannot be adjacent to $W$; so we are done.
If $E_{n-1}, E_n$ are not adjacent, then $\gamma$ intersects a wall $W' \subset \partial K$, $W' \neq W$ twice between $E_{n-1}, E_n$.
All columns on the minimal chain between $E^*_{n-1}$ and $E_n$ are adjacent to $W'$.
By Lemma \ref{Lem:EEinKisatree} at most $2$ of those columns can also be adjacent to $W$.

Assertion (b2) follows from assertion (b1) of the induction hypothesis and Lemma \ref{Lem:gamminKKsCaseB} by passing to subsegments of $\gamma$ that intersect $W$ exactly once and whose endpoints are contained in columns adjacent to $W$.

Next, we establish assertions (c) and (a).
It suffices to consider the case in which $\gamma(0), \gamma(1) \in K$.
Let $E_1, \ldots, E_n \in \mathcal{E}_K$ be as defined in the proposition.
If $n \leq 2$, then we are done using assertion (b2); so assume $n \geq 3$.
If $E_{n-1}$ lies on the minimal chain between $E_1$ and $E_n$, then we are done as in the proof of Proposition \ref{Prop:characterizationminimizinggamma}.
So assume that $E_{n-1}$ does not lie on the minimal chain between $E_1, E_n$.

We show that $E_{n-1}, E_n$ cannot be adjacent.
Otherwise, as in the proof of Proposition \ref{Prop:characterizationminimizinggamma}, $\gamma$ intersects a wall $W \subset \partial K$ twice between $E_{n-2}$ and $E_{n-1}$ and the columns $E_{n-2}, E^*_{n-2}, E_n, E_{n-1}$ lie along $W$ in that order.
This contradicts assertion (b2).

So there is a wall $W \subset \partial K$ that is adjacent to both $E_{n-1}, E_n$ and $\gamma$ crosses $W$ twice between those two columns.
We now proceed as in the proof of Proposition \ref{Prop:characterizationminimizinggamma}, but we have to be careful whenever we make use of assertion (b).
As in this proof, we can find a column $E^* \in \mathcal{E}_K$ that lies on the three minimizing chains between $E_{n-1}, E_n$ and $E_1, E_{n-1}$ and $E_1, E_n$ and $E^* \neq E_{n-1}$.
We also know that $E_{n-2}$ cannot be adjacent to to $E_{n-1}$, since otherwise it would lie on the minimal chain between $E_{n-1}$ and $E^*$ along $W$, in contradiction to assertion (b2).
So by assertion (c6) of the induction hypothesis $E_{n-2}, E^*_{n-2}, E_{n-1}$ are adjacent to a wall $W' \subset \partial K$ such that $\gamma$ intersects $W'$ twice between $E_{n-2}, E_{n-1}$.
This implies $W' \neq W$ by assertion (b).
Now both $W$ and $W'$ are adjacent to all columns on the minimal chain between $E_{n-1}, E^*$ or between $E_{n-1}, E^*_{n-2}$, whichever is shorter.
So by Lemma \ref{Lem:EEinKisatree} we must have $\dist^H_K (E^*, E_{n-1}) = 1$.
Assertion (c1)--(c6) now follow as in the proof of Proposition \ref{Prop:characterizationminimizinggamma}.

Now for assertion (a), we may assume that $\gamma(0), \gamma(1) \in E \in \mathcal{E}_K$ in view of the induction hypothesis.
Then assertion (c) implies that $\gamma$ is fully contained in $E$.

Finally, we establish assertion (b3).
In view of the induction hypothesis it suffices to consider the case in which $E_1, E_2 \in \mathcal{E}_K$ and in which $\gamma$ does not intersect $W$.
Apply assertion (c) to obtain sequences $E'_1, \ldots, E'_n$ and $E^{\prime *}_1, \ldots, E^{\prime *}_n$ with $E^{\prime *}_1 = E'_1 = E_1$ and $E^{\prime *}_n = E'_n = E_2$.
It follows that all columns $E^*_1, \ldots, E^*_n$ are adjacent to $W$.
If $\gamma$ crossed a wall $W' \subset \partial K$ twice in between some $E'_i, E'_{i+1}$, then all columns between $E^{\prime *}_i, E^{\prime *}_{i+1}$ would be adjacent to $W'$ and $W$.
This is impossible by Lemma \ref{Lem:EEinKisatree}.
\end{proof}

The next Lemma is an analog of Lemma \ref{Lem:interchangeorderinK}.
Note that in the setting of condition (B), we don't need to work inside a single chamber.
This fact will later compensate us for the lack of an analog for Lemma \ref{Lem:distg1g2afterW}.

\begin{Lemma} \label{Lem:interchangeorderinKCaseB}
There are constants $\eta^* > 0$ and $H^* < \infty$ such that if $\eta \leq \eta^*$ and $H \geq H^*$,  the following holds:

Let $E^\circ_1, E^\circ_2 \in \mathcal{E}$ be two columns and $Q_1, \ov{Q}_1 \subset E^\circ_1$, $Q_2, \ov{Q}_2 \subset E^\circ_2$ cells such that the vertical orders of $Q_1, \ov{Q}_1$ and $Q_2, \ov{Q}_2$ are opposite to each other.
Let $\gamma, \ov\gamma : [0,1] \to \td{M}$ be minimizing arcs from $Q_1$ to $Q_2$ and from $\ov{Q}_1$ to $\ov{Q}_2$.
Then we can find cells $Q', \ov{Q}' \in \mathcal{Q}$ that intersect $\gamma, \ov\gamma$ and such that $\dist(Q', \ov{Q}') < 3H$.
\end{Lemma}

\begin{proof}
Consider first a wall $W \in \mathcal{W}$ that intersects $\gamma$ (and hence also $\ov\gamma$) exactly once.
Let $K, K' \in \mathcal{K}$ be the chambers that are adjacent to $W$ from either side in such a way that $\gamma(0)$ and $\ov\gamma(0)$ lie on the same side of $W$ as $K$.
Let $E \in \mathcal{E}_K$ be the first column on $\gamma$ that is adjacent to $W$ and choose $\ov{E} \in \mathcal{E}_K$ analogously.
We argue similarly as in the proof of Lemma \ref{Lem:distg1g2afterW} that $\dist^H_K (E, \ov{E}) \leq 3$.
Let $E^* \in \mathcal{E}_K$ be the first column on $\gamma$ inside $K$ and define $\ov{E}^*$ analogously.
Then either $E^* = \ov{E}^* = E^\circ_1$ or $E^\circ_1 \not\in \mathcal{E}_K$.
In both cases there is a wall $W^* \subset \partial K$, $W^* \neq W$ that is adjacent to both $E^*$ and $\ov{E}^*$.
Let $E^{**} \in \mathcal{E}_K$ be the last column on the minimal chain between $E$ and $E^*$ that is adjacent to $W$ and define $\ov{E}^{**} \in \mathcal{E}_K$ analogously.
By Proposition \ref{Prop:characteriztionminimizinggammaCaseB}(b1) we have $\dist^H_K (E, E^{**}), \dist^H_K (\ov{E}, \ov{E}^{**}) \leq 1$.
It now follows as in the proof of Lemma \ref{Lem:distg1g2afterW} that $\dist^H_K (E^{**}, \ov{E}^{**}) \leq 1$ and hence $\dist^H_K (E, \ov{E}) \leq 3$ (observe that this part of the proof only makes use of the tree property of $\mathcal{E}_K$ from Lemma \ref{Lem:EEinKisatree}).

Let now $W_1, \ldots, W_h$ be all the walls that $\gamma$ intersects exactly once in this order.
Then also $\ov\gamma$ intersects each of these walls exactly once in this order.
For each $i = 1, \ldots, h$ let $E'_i \in \mathcal{E}$ be the first and $E''_i$ the last column on $\gamma$ that is adjacent to $W_i$.
Define $\ov{E}'_i$ and $\ov{E}''_i$ accordingly.
By the last paragraph, we obtain that $E'_i, \ov{E}'_i$ and $E''_i, \ov{E}''_i$ have horizontal distance $\leq 3$ in the chamber in which they are contained (the bound on the horizontal distance between $E''_i$ and $\ov{E}''_i$ can be obtained by reversing the orientation of $\gamma$ and $\ov\gamma$).
Choose cells $Q'_i \subset E'_i$, $Q''_i \subset E''_i$ or $\ov{Q}'_i \subset \ov{E}'_i$, $\ov{Q}''_i \subset \ov{E}''_i$ that intersect $\gamma$ or $\ov\gamma$, respectively.

We first consider the case in which there is some $i \in \{ 1, \ldots, h \}$ such that the vertical orders of $Q'_i, \ov{Q}'_i$ and $Q''_i, \ov{Q}''_i$ are different.
Observe that by Proposition \ref{Prop:characteriztionminimizinggammaCaseB}(b3) the arc $\gamma$ only intersects cells adjacent to $W_i$ between $Q'_i$ and $Q''_i$; the same is true for $\ov\gamma$.
Let $S \subset \td{M}$ be the union of all cells that $\gamma$ intersects between $Q'_i, Q''_i$ and define $\ov{S}$ accordingly.
Using Lemma \ref{Lem:gamminKKsCaseB} we find that either $S \cap W_i$ and $\ov{S} \cap W_i$ intersect or there is a cell $Q' \in \mathcal{Q}$ on $\gamma$ with $\dist(Q', \ov{Q}'_i) \leq 3 + H$ or $\dist(Q', \ov{Q}''_i) \leq 3 + H$ or there is a cell $\ov{Q}' \in \mathcal{Q}$ on $\ov\gamma$ with $\dist(\ov{Q}', Q'_i) \leq 3+ H$ or $\dist(\ov{Q}', Q''_i) \leq 3+H$.
In all these cases we are done.

So assume from now on that the vertical orders of $Q'_i, \ov{Q}'_i$ and $Q''_i, \ov{Q}''_i$ are the same for all $i = 1, \ldots, h$.
Choose $i \in \{ 1, \ldots, h \}$ minimal such that the vertical order of $Q'_i, \ov{Q}'_i$ differs from that of $Q_1, \ov{Q}_1$.
If there is no such $i$, then the vertical orders of $Q'_h, \ov{Q}'_h$ and $Q_2, \ov{Q}_2$ are opposite and we can get rid of this case by reversing the orientations of $\gamma$ and $\ov\gamma$.
Let $K \in \mathcal{K}$ be the chamber that contains $Q'_i, \ov{Q}'_i$.
If $i > 1$, the choice of $i$ implies that the vertical order $Q''_{i-1}, \ov{Q}''_{i-1}$ is different from that of $Q'_i, \ov{Q}'_i \subset K$.
If $i = 1$, then the vertical order of $Q_1, \ov{Q}_1 \subset K$ is different from that of $Q'_1, \ov{Q}'_1$.
Apply Proposition \ref{Prop:characteriztionminimizinggammaCaseB}(c) to the subsegment of $\gamma$ between $Q'_{i-1}$ or $Q_1$ and $Q''_i$ to obtain columns $E_1, \ldots, E_n$ and $E^*_1, \ldots, E^*_n \in \mathcal{E}_K$.
Similarly we obtain the columns $\ov{E}_1, \ldots, \ov{E}_{\ov{n}}$ and $\ov{E}^*_1, \ldots, \ov{E}^*_{\ov{n}} \in \mathcal{E}_K$ for the corresponding subsegment of $\ov\gamma$.
Note that $\dist^H_K (E_1, \ov{E}_1), \dist^H_K (E_n, \ov{E}_{\ov{n}}) \leq 3$.

If $\dist^H_K (E_1, E_n), \dist^H_K (\ov{E}_1, \ov{E}_{\ov{n}}) \leq 6$, then by Proposition \ref{Prop:characteriztionminimizinggammaCaseB}(c), all columns $E_i$ and $\ov{E}_i$ have distance $\leq 17$ from one another and hence we can just pick cells $Q', \ov{Q}'$ that are horizontally aligned to show the Lemma.
So assume from now on that this is not the case and let $L$ and $\ov{L}$ be the minimal chains between $E_1, E_n$ and $\ov{E}_1, \ov{E}_{\ov{n}}$.
By the tree property as explained in Lemma \ref{Lem:EEinKisatree}, $L$ and $\ov{L}$ intersect in a minimal chain $L^\circ$ such that every column on $(L \cup \ov{L}) \setminus L^\circ$ has horizontal distance $\leq 3$ from $L^\circ$.

As in the proof of Lemma \ref{Lem:interchangeorderinK} define the sets $S \subset L \cup E_1 \cup \ldots \cup E_n$, $S' \subset L$ and $\ov{S} \subset \ov{L} \cup \ov{E}_1 \cup \ldots \cup \ov{E}_{\ov{n}}$, $\ov{S}' \subset \ov{L}$ .
Observe that $S', \ov{S}'$ lie in different sets and might not intersect as before.
However, we can still find cells $Q^\circ \subset S'$, $\ov{Q}^\circ \subset \ov{S}'$ such that 
\[ \dist^H_K (Q^\circ, \ov{Q}^\circ) \leq 6 \qquad \text{and} \qquad \dist^V_K (Q^\circ, \ov{Q}^\circ) = 0. \]
We will work with these cells now instead of $Q^\circ$ alone.
By the definition of $S'$ there is a cell $Q^{\circ\circ} \subset S$ that is either equal to $Q^\circ$ or adjacent to $Q^\circ$ and horizontally aligned with it, i.e. $\dist^H_K (Q^{\circ\circ}, Q^\circ) \leq 1$ and $\dist^V_K (Q^{\circ\circ}, Q^\circ) = 0$.
Again, by the definition of $S$, there is a cell $Q' \in \mathcal{Q}$ on $\gamma$ that is either equal to $Q^{\circ\circ}$ or adjacent to it across a wall, i.e. $\dist(Q', Q^{\circ\circ}) \leq H$.
Altogether this implies that $\dist (Q', Q^\circ) \leq 1 + H$.
By an analogous argument, we can find a cell $\ov{Q}'$ on $\ov\gamma$ with $\dist(\ov{Q}', \ov{Q}^\circ) \leq 1 + H$.
Hence $\dist(Q', \ov{Q'}) \leq 5 + 2H < 3H$ for large enough $H$.
\end{proof}

\begin{Proposition} \label{Prop:combinatorialconvexityCaseB}
Proposition \ref{Prop:combinatorialconvexityseveralSeifert} also holds in the case in which $M$ satisfies condition (B).
\end{Proposition}

\begin{proof}
We follow the proof of Lemma \ref{Lem:combconvexseveralSeifertpreparation}.
Observe that since $M$ satisfies condition (B), the action $\varphi : \IZ \curvearrowright \td{M}$ acts as a vertical shift on each column of $\td{M}$.
So we do not need to restrict to the case in which the cells $Q_0, Q_1, Q_2$ lie in the same chamber.
Instead of applying Lemma \ref{Lem:interchangeorderinK}, we now make use of Lemma \ref{Lem:interchangeorderinKCaseB} to obtain cells $Q'_1, Q'_2 \subset K$ on $\varphi_z \circ \gamma_1$ and $\gamma_2$ with $\dist(Q'_1, Q'_2) < 3H < 4H$.
The rest of the proof is exactly the same as that of Lemma \ref{Lem:combconvexseveralSeifertpreparation}.
\end{proof}

\subsection{Proof of Proposition \ref{Prop:easiermaincombinatorialresultCaseb} if $M$ satisfies condition (B) or (C)} \label{subsec:proofofeasiermaincombinatorialresult}
We will now apply the combinatorial convexity estimates from Propositions \ref{Prop:combinatorialconvexityseveralSeifert} and \ref{Prop:combinatorialconvexityCaseB} to construct large polyhedral balls in $\td{M}$ which consist of cells.
In the following we will always assume that $M$ satisfies condition (B) or (C) and that $\eta$, $H$ have been chosen smaller/larger than than all constants $\eta^*$, $H^*$, respectively, which appeared the Lemmas and Propositions of subsections \ref{subsec:combconvexincaseC} and \ref{subsec:combconvexincaseB}.

\begin{Lemma} \label{Lem:Sinchambersimplyconnected}
Let $K \in \mathcal{K}$ be a chamber of $\td{M}$ and consider a finite union of cells $S \subset K$ whose interior is connected.
Assume that $S$ has the property that for any two cells $Q_1, Q_2 \subset S$ that are vertically aligned, $S$ also contains all cells that are vertically between $Q_1$ and $Q_2$.
Then $S$ is homeomorphic to a closed $3$-disk and the intersection of $S$ with every wall $W \subset \partial K$ has connected interior in $W$.
More precisely, there is a continuous, injective map $b : D^3 \to \td{M}$ with $b(D^3) = S$ that is an embedding on $B^3 \cup (S^2 \setminus b^{-1} (\td{V}^{(1)}))$ and for all walls $W \subset \partial K$ the preimage $b^{-1} (W)$ is either empty or a (connected) topological disk that is the union of rectangles.
\end{Lemma}

\begin{proof}
The Lemma is obviously true if $S$ only consists of cells that are vertically aligned.
Observe next that the columns of $K$ are bounded by subsets of $\partial K$ and components of $\pi^{-1} (C_K)$.
Those components correspond to arcs of $\td{C}^*_K \subset \td\Sigma_K$, are diffeomorphic to $I \times \IR$ and every two adjacent columns intersect in exactly one such component.
Moreover, each such component separates $K$ into two components.

Consider now such a component $X \subset \pi^{-1} (C_K)$ with the property that not all cells of $S$ lie on one side of $X$.
This is always possible if not all cells of $S$ are vertically aligned.
Let $S_1, S_2 \subset K$ be the closures of the two components of $S \setminus X$.
Then $S_1 \cap S_2$ is a connected rectangle and so the interiors of $S_1, S_2$ must be connected and hence $S_1, S_2$ are homeomorphic to $3$-disks.
Since the interior of $S_1 \cap S_2$ in $X$ is a (connected) disk, we find that $S = S_1 \cup S_2$ is a topological $3$-disk as well.
The fact that $S$ is a topological disk follows from this argument by induction.

Next, let $W \subset \partial K$ be a wall and assume that two cells $Q, Q' \subset S$ are adjacent to $W$.
Let $E, E' \in \mathcal{E}_K$ be the columns that contain $Q, Q'$.
Since $S$ is connected, we can find a chain $(E_0, \ldots, E_n)$ between $E, E'$ such that $E_i$ contains a cell of $S$ for all $i = 0, \ldots, n$.
We may assume that we have picked the chain such that $n$ is minimal.
Thus this chain cannot contain any column twice.
Hence it is minimal and so all its columns are adjacent to $W$.
Note that $E_i \cap S \cap W$ is a rectangle for each $i = 0, \ldots, n$.
By the previous paragraph, the rectangles $E_{i-1} \cap S \cap W$ and $E_i \cap S \cap W$ intersect in more than one point.
It follows that $S \cap W$ is a topological disk.

It follows easily that we can connect $Q$ with $Q'$ through cells in $K$ which are adjacent to $W$ and hence $S \cap W$ is connected.
By the property of $S$, this intersection can only be a topological $2$-disk.

The existence of the map $b$ follows along the lines of this proof.
\end{proof}

Let $Q_0 \in \mathcal{Q}$ be an arbitrary cell and $R > 0$ a positive number.
Then we define
\[ B_R (Q_0) = \bigcup \big\{ Q \in \mathcal{Q} \;\; : \;\; \dist (Q, Q_0) < R \big\}. \]
Next, consider the distance function $\dist^{\mathcal{K}} : \mathcal{K} \times \mathcal{K} \to [0, \infty)$, which assigns to every pair of chambers $K_1, K_2$ the length of the minimal chain between $K_1, K_2$.
This length is equal to the minimal number of intersections of an arc between $K_1, K_2$ with the walls of $\td{M}$.
For two cells $Q_1 \subset K_1, Q_2 \subset K_2$ we set $\dist^{\mathcal{K}} (Q_1, Q_2) = \dist^{\mathcal{K}} (K_1, K_2)$.
Observe that
\[ \dist (Q_1, Q_2) \geq H \dist^{\mathcal{K}} (K_1, K_2). \]
Let $J > 0$ be a large constant whose value we will determine later.
We define a new distance function $\dist' (\cdot, \cdot)$ on $\mathcal{Q}$ as follows
\[ \dist' (Q_1, Q_2) := \dist(Q_1, Q_2) + J \dist^{\mathcal{K}} (Q_1, Q_2). \]
Obviously, $(\mathcal{Q}, \dist')$ is a metric space.
Set moreover
\[ B'_R (Q_0) = \bigcup \big\{ Q \in \mathcal{Q} \;\; : \;\; \dist'(Q, Q_0)  < R \big\}. \]
Finally, we define
\[ P_R (Q_0) = \bigcup \Bigg\{ Q \in \mathcal{Q} \;\; : \;\; \begin{array}{l} Q \subset K \in \mathcal{K} \; \text{and there are cells $Q_1, Q_2 \subset B'_R (Q_0)$} \\ \text{in $K$ such that $Q_1, Q, Q_2$ are vertically aligned} \\ \text{and $Q$ lies vertically between $Q_1, Q_2$} \end{array} \Bigg\}. \]

\begin{Proposition} \label{Prop:distballsinQQareballs}
Assume that $M$ satisfies condition (B) or (C).
Then there are choices for $\eta, H, J$ and a constant $C_2 < \infty$ such that the following holds:

For all $Q_0 \in \mathcal{Q}$ and all $R > 0$ we have
\[ B'_R (Q_0) \subset P_R (Q_0) \subset \Int B'_{R + C_2} (Q_0) \cup \partial \td{M}. \]
Moreover, there is a continuous map $b_{R, Q_0} : D^3 \to \td{M}$ such that $b_{R, Q_0} (D^3) = P_R(Q_0)$ and $b_{R, Q_0} (S^2) = \partial P_R(Q_0)$ and such that $b_{R, Q_0}$ is an embedding on $B^3 \cup (S^2 \setminus b_{R, Q_0}^{-1} (\td{V}^{(1)}))$.

Finally, let $K_0 \in \mathcal{K}$ be the chamber that contains $Q_0$.
Then for all cells $Q \subset B'_R(Q_0) \cap K_0$ we have $\dist^H_{K_0} (Q, Q_0), \dist^V_{K_0} (Q, Q_0) < C_2 R$.
\end{Proposition}

\begin{proof}
We will see that the proposition holds for $J = 11H$.

We first show that
\begin{equation} \label{eq:easyinclusionofcellballs}
 B'_R (Q_0) \subset P_R (Q_0) \subset B'_{R + 10H} (Q_0).
\end{equation}
The first inclusion property is trivial.
For the second inclusion property consider a cell $Q \subset P_R(Q_0)$.
Let $K \in \mathcal{K}$ be the chamber that contains $Q$ and choose cells $Q_1, Q_2 \subset B'_R (Q_0) \cap K$ such that $Q_1, Q, Q_2$ are vertically aligned and $Q$ lies vertically in between $Q_1, Q_2$.
Then $\dist (Q_1, Q_0) = \dist' (Q_1, Q_0) - J \dist^{\mathcal{K}} (K, K_0) < R - J \dist^{\mathcal{K}} (K_0, K)$ and similarly $\dist (Q_2, Q_0) < R - J \dist^{\mathcal{K}} (K, K_0)$.
It follows from Proposition \ref{Prop:combinatorialconvexityseveralSeifert} in case (C) and Proposition \ref{Prop:combinatorialconvexityCaseB} in case (B) that $\dist (Q, Q_0) < R + 10H - J \dist^{\mathcal{K}} (K, K_0)$.
So $\dist'(Q, Q_0) < R + 10H$ and (\ref{eq:easyinclusionofcellballs}) follows.
In order to establish the inclusion property of this proposition, it hence suffices to choose $C_2$ larger than $10H + J$ plus the maximal number of cells that can intersect in one point.

Next, choose a sequence $K_1, K_2, \ldots \in \mathcal{K}$ such that $\mathcal{K} = \{ K_0, K_1, K_2, \ldots \}$ and such that $\dist^{\mathcal{K}} (K_n, K_0)$ is non-decreasing in $n$.
We will first show that the interior of $B'_R(Q_0) \cap (K_0 \cup \ldots \cup K_n)$ is connected for each $n \geq 0$:
Fix $n$, choose a cell $Q \subset B'_R(Q_0) \cap (K_0 \cup \ldots \cup K_n)$, $Q \neq Q_0$, let $K_i$ be the chamber that contains $Q$ and consider a combinatorially minimizing arc $\gamma : [0,1] \to \td{M}$ from $Q_0$ to $Q$.
We show by induction on the number of cells that intersect $\gamma$ that $\Int Q$ lies in the same connected component of $\Int (B'_R(Q_0) \cap (K_0 \cup \ldots \cup K_n))$ as $\Int Q_0$.
Let $Q' \in \mathcal{Q}$ be the cell that $\gamma$ intersects prior to $Q$.
If $Q' \subset K_i$, then we are done by the induction hypothesis since then $\dist'(Q', Q_0) < \dist'(Q, Q_0)$ and hence $Q' \subset B'_R(Q_0) \cap (K_0 \cup \ldots \cup K_n)$.
Assume next that $Q' \subset K_j \in \mathcal{K}$ for $j \neq i$ and hence $\gamma$ crosses a wall $W = K_i \cap K_j \in \mathcal{W}$ in between $Q'$ and $Q$.
Then $\dist (Q', Q_0) = \dist (Q, Q_0) - H$ and $\dist^{\mathcal{K}} (K_j, K_0) = \dist^{\mathcal{K}} (K_i, K_0) \pm 1$.
It suffices to consider the case in which $\dist^{\mathcal{K}} (K_j, K_0) = \dist^{\mathcal{K}} (K_i, K_0) + 1$ since otherwise we are again done by the induction hypothesis.
In this case $\gamma$ must cross $W$ twice and there is a cell $Q'' \subset K_i$ that $\gamma$ intersects right before intersecting $W$ for the first time.
By Proposition \ref{Prop:characterizationminimizinggamma}(b2) in case (C) or Proposition \ref{Prop:characteriztionminimizinggammaCaseB}(b3) in case (B), the arc $\gamma$ only intersects cells that lie in $K_j$ and that are adjacent to $W$ between $Q''$ and $Q'$.
Consider now all cells $Q^* \subset K_i$ that are adjacent to a cell $Q^{**} \subset K_j$ which intersects $\gamma$.
For each such $Q^*$ we have
\[ \dist(Q^*, Q_0) \leq H + \dist(Q^{**}, Q_0) \leq H + \dist(Q', Q_0) = \dist(Q, Q_0) \]
and thus $\dist'(Q^*, Q_0) \leq \dist' (Q, Q_0)$ and $Q^* \subset B'_R(Q_0) \cap K_i$.
It follows that $Q'$ and $Q''$ lie in the same connected component of $B'_R (Q_0) \cap K_i$.
This finishes the induction argument.

So also the interior of $P_R (Q_0) \cap (K_0 \cup \ldots \cup K_n)$ is connected for all $n \geq 0$.
We will now show by induction on $n$ that there is a continuous map $b_n : D^3 \to \td{M}$ whose image is equal to the closure of this interior and which is an embedding when restricted to $B^3 \cup (S^2 \setminus b^{-1}_n (\td{V}^{(1)}))$.
For $n = 0$ this statement follows immediately from Lemma \ref{Lem:Sinchambersimplyconnected} and the fact that the interior of $P_R(Q_0) \cap K_0$ is connected.
Assume now that $n \geq 1$.
There is a unique $i \in \{ 1, \ldots, n-1 \}$ such that $K_i$ is adjacent to $K_n$.
So $\dist^{\mathcal{K}} (K_i, K_0) = \dist^{\mathcal{K}} (K_n, K_0) - 1$.
Let $W = K_i \cap K_n \in \mathcal{W}$ be the wall between $K_i$ and $K_n$.
Observe that for every cell $Q \subset P_R(Q_0) \cap K_n$ that is adjacent to $W$ and every cell $Q' \subset K_i$ that is adjacent to $Q$ we have $\dist(Q', Q_0) \leq \dist(Q, Q_0) + H$.
So by (\ref{eq:easyinclusionofcellballs})
\[ \dist' (Q', Q_0) \leq \dist' (Q, Q_0) + H - J < R+ 11H - J = R. \]
Hence $Q' \subset B'_R (Q_0) \subset P_R (Q_0)$.
This implies
\[ \ov{P_R (Q_0) \cap \Int K_n} \cap W \subset P_R (Q_0) \cap W = \ov{P_R (Q_0) \cap \Int K_i} \cap W = b_{n-1} (D^3) \cap W . \]
By Lemma \ref{Lem:Sinchambersimplyconnected} the set $\ov{P_R (Q_0) \cap \Int K_n}$ is the union of the images of maps $b' : D^3 \to \td{M}$ with the appropriate regularity properties and any two such images intersect in at most an edge of $\td{V}$.
Moreover, the preimage of $W$ under every such map $b'$ is a (connected) topological disk that is contained in $b_{n-1}(D^3) \cap W$.
So we can combine $b_{n-1}$ with the maps $b'$ to obtain a map whose image is equal to the closure of the interior of $P_R (Q_0) \cap (K_0 \cup \ldots \cup K_n)$.
Smoothing this map in the interior of $D^3$ yields $b_n$.
This finishes the induction and proves the second assertion of the Proposition for large $n$.

Finally, we show the last statement.
Let $Q \subset B'_R(Q_0) \cap K_0$.
Then $\dist (Q, Q_0) < R$.
Consider a minimizing arc $\gamma : [0,1] \to \td{M}$ between $Q_0, Q$.
By Proposition \ref{Prop:characterizationminimizinggamma} or Proposition \ref{Prop:characteriztionminimizinggammaCaseB} the arc $\gamma$ stays within the union of $K_0$ with the chambers that are adjacent to $K_0$.
Let $Q_0, Q_1, \ldots, Q_n = Q' \subset K_0$ be the cells of $K_0$ that $\gamma$ intersects in that order.
Then for all $i = 0, \ldots, n-1$ either $\dist^H_{K_0} (Q_i, Q_{i+1}) + \dist^V_{K_0} (Q_i, Q_{i+1}) = 1 \leq \eta^{-1} \dist(Q_i, Q_{i+1})$ or $\gamma$ intersects a wall $W \subset \partial K$ right after $Q_i$ and right before $Q_{i+1}$.
In this case let $K' \in \mathcal{K}$ be the chamber on the other side of $W$ and let $Q'_i, Q'_{i+1} \subset K'$ be the cells that $\gamma$ intersects right after $Q_i$ and right before $Q_{i+1}$.
By Lemma \ref{Lem:QQestimatesatW}(f) we have
\[ \dist^H_{K_0} (Q_i, Q_{i+1}),\; \dist^V_{K_0} (Q_i, Q_{i+1}) < C_0 \eta^{-1} \dist(Q'_i, Q'_{i+1}) + C_0 . \]
If $H > 1$, then the right hand side is bounded by $C_0 \eta^{-1} \dist(Q_i, Q_{i+1})$.
The rest follows from the triangle inequality for $\dist^H_{K_0}$ and $\dist^V_{K_0}$ with $C_2 > C_0 \eta^{-1}$.
\end{proof}

We can finally establish Proposition \ref{Prop:easiermaincombinatorialresultCaseb} and hence Proposition \ref{Prop:maincombinatorialresult}(a) (see subsection \ref{subsec:reducifnoT2bundle}).

\begin{proof}[Proof of Proposition \ref{Prop:easiermaincombinatorialresultCaseb}]
By Proposition \ref{Prop:CasAiseasy}, we may assume that $M$ satisfies condition (B) or (C).

Observe first that the universal covering $\pi : \td{M} \to M$ can be seen as the restriction of the universal covering $\pi : \td{M}_0 \to M_0$ to a component of $\pi^{-1} (M)$.
Consider the simplicial complex $V \subset M$ as defined in subsection \ref{subsec:constructionofV} and let $f_0 : V \to M$ be the inclusion map.
Recall that $f_0$ lifts to the inclusion map $\td{f}_0 : \td{V} \to \td{M}$ in the universal covering $\pi : \td{M} \to M$.
Consider the Riemannian metric $g$ on $M_0$ and the map $f : V \to M$ from the assumptions of the Proposition.
Let $H : V \times [0,1] \to M_0$ be the homotopy between $f_0$ and $f$ and let $L$ be a strict upper bound on the length of the arcs $t \mapsto H(x, t)$ (note that $V$ is compact).
Since this homotopy leaves $\partial V$ invariant and embedded in $\partial M$, we can extend $H$ to a homotopy $H^* : (V \cup \partial M) \times [0,1] \to M_0$ between the inclusion map $f_0^* : V \cup \partial M \to M$ and the extension $f^* : V \cup \partial M \to M_0$ of $f$ such that $H^* (\cdot , t)$ restricted to $\partial M$ is the identity for all $t \in [0,1]$.
Here we view $V \cup \partial M$ as a connected simplicial complex.
The homotopy $H^*$ can be lifted to a homotopy $\td{H}^* : (\td{V} \cup \partial \td{M}) \times [0,1] \to \td{M}_0$ between the inclusion map $\td{f}_0 : \td{V} \cup \td{M} \to \td{M}$ and a lift $\td{f}^* : \td{V} \cup \partial \td{M} \to \td{M}_0$ of $f^*$, i.e. $f^* \circ \pi |_{\td{V} \cup \partial \td{M}} = \pi \circ \td{f}^*$.
Note that $\pi (\td{H}^* (x,t)) = H^* (\pi(x),t)$ for all $(x,t) \in (\td{V} \cup \partial \td{M}) \times [0,1]$.
Still, the lengths of the arcs $t \mapsto \td{H}^*(x, t)$ are bounded by $L$.

Consider the solid torus $S \subset \Int M$ and pick a component $\td{S} \subset \pi^{-1} (S) \subset \Int \td{M}$.
Fix a diffeomorphism $\Phi : S^1 \times D^2 \to S$ and an orientation on the $S^1$-factor and denote by $\sigma = \Phi(S^1 \times \{ 0\}) \subset S$ the core of $S$.
By our assumptions, $\td{S} \approx \IR \times D^2$ and $\pi |_{\td{S}} : \td{S} \to S$ is a universal covering of $S$.
So we can find a lifted diffeomorphism $\td{\Phi} : \IR \times D^2 \to \td{S}$, such that $\pi \circ \td{\Phi} = \Phi \circ \pi_{S^1 \times D^2}$, where $\pi_{S^1 \times D^2} : \IR \times D^2 \to S^1 \times D^2, (u,x) \mapsto (e^{2\pi i u}, x)$ is the standard universal covering map.
Let $F = \td\Phi ( [0,1] \times D^2) \subset \td{S}$.
Then $\pi (F) = S$ and $\pi$ restricted to the interior of $F$ is injective.
So $F$ is a fundamental domain for the universal covering $\pi |_{\td{S}} : \td{S} \to S$.
The central loop $\sigma \approx S^1 \times \{ 0 \} \subset S \approx S^1 \times D^2$ induces a deck transformation $\varphi : \td{M}_0 \to \td{M}_0$, which is an isometry and $\td{S}$ is covered by fundamental domains of the form $\varphi^{(n)} (F)$ where $n \in \IZ$.
Observe also that $\td\sigma = \td\Phi (\IR \times \{ 0 \}) = \pi^{-1} (\sigma) \cap \td{S}$ is a properly embedded, infinite line, which is invariant under $\varphi$.

Choose a chamber $K_0 \in \mathcal{K}$ for which the displacement $\dist^{\mathcal{K}} (K_0, \varphi(K_0))$ is minimal.
Next, if $\varphi(K_0) = K_0$ choose a column $E_0 \in \mathcal{E}_{K_0}$ for which the displacement $\dist^H_{K_0} (E_0, \varphi(E_0))$ is minimal.
If $\varphi(K_0) \neq K_0$, the column $E_0 \in \mathcal{E}_{K_0}$ can be chosen arbitrarily.
Finally, choose an arbitrary cell $Q_0 \subset E_0$.
We will now show that there is a universal constant $c > 0$, which only depends on the structure of $V$ (and not on $S$!) such that for all $n \in \IZ$
\begin{equation} \label{eq:distgrowslinearlyinQQ}
 \dist'(Q_0, \varphi^{(n)} (Q_0)) \geq c |n|.
\end{equation}
If $\varphi (K_0) \neq K_0$, then we argue as follows.
Consider the minimal chain between $K_0$ and $\varphi(K_0)$ in the adjacency graph of $\mathcal{K}$.
The images of this minimal chain under the deck transformations $\varphi^{(0)}, \ldots, \varphi^{(n-1)}$ are each minimal and can be concatenated along $\varphi^{(1)} (K_0), \ldots, \varphi^{(n-1)} (K_0)$ to a chain between $K_0$ and $\varphi^{(n)} (K_0)$.
We now claim that this chain is minimal.
Otherwise, there are elements in this chain that occur at least twice.
Since the adjacency graph of $\mathcal{K}$ is a tree (see Lemma \ref{Lem:KKistree}), there must then be even two consecutive elements in this chain that are equal.
These two elements can only come from two distinct images of the minimal chain between $K_0$ and $\varphi(K_0)$.
So if $K'_0, K''_0 \in \mathcal{K}$ are the second and second last elements on this minimal chain then we must have $\varphi^{(i+1)} (K'_0) = \varphi^{(i)}(K''_0)$ for some $i \in \{ 0, \ldots, n-1 \}$.
But this would imply that $\varphi(K'_0) = K''_0$ and hence $\dist^{\mathcal{K}} ( K'_0, \varphi(K'_0)) = \dist^{\mathcal{K}} ( K'_0, K''_0)  = \dist^{\mathcal{K}} (K_0, \varphi(K_0)) - 2$; contradicting the minimal choice of $K_0$.
So we conclude that the chain in question is minimal and hence $\dist^{\mathcal{K}} (K_0, \varphi^{(n)}(K_0) ) \geq |n|$ for all $n \in \IZ$, which establishes (\ref{eq:distgrowslinearlyinQQ}) assuming $c < H+J$.

If $\varphi(K_0) = K_0$ but $\varphi(E_0) \neq E_0$, then we can draw the same conclusions for $\mathcal{E}_{K_0}$ instead of $\mathcal{K}$ and obtain $\dist^H_{K_0} (Q_0, \varphi^{(n)}(Q_0)) \geq |n|$ for all $n \in \IZ$.
If $\varphi(E_0) = E_0$, then $\dist^V_{K_0} (Q_0, \varphi^{(n)}(Q_0)) \geq |n|$ for all $n \in \IZ$.
So in the latter two cases (\ref{eq:distgrowslinearlyinQQ}) follows by the last assertion of Proposition \ref{Prop:distballsinQQareballs}.

Let $N \geq 1$ be some large natural number whose value we will determine at the end of the proof.
It will depend on $g$, $f$ and $L$.
The sets $\td{S}_+ = F \cup \varphi (F) \cup \ldots \varphi^{(N-1)} (F)$ and $\td{S}_- = \varphi^{(-1)} (F) \cup \ldots \varphi^{(-N)} (F)$ are each diffeomorphic to a solid cylinder $\approx I \times D^2$ and are bounded by annuli inside $\partial \td{S}$ as well as disks $D_0, D_+$ and $D_0, D_-$ where $D_+ = \varphi^{(N)} (D_0)$ and $D_- = \varphi^{(-N)} (D_0)$.
Let $\td\sigma_+, \td\sigma_-$ be the subsegments of $\td\sigma$ that connect $D_0$ with $D_+$ and $D_0$ with $D_-$, i.e. $\td\sigma_+ = \td\sigma \cap \td{S}_+$ and $\td\sigma_- = \td\sigma \cap \td{S}_-$.

Choose $R_0 > 0$ large enough such that $B'_{R_0} (Q_0)$ contains all points of $\td{M}$ that have distance at most $L$ from $F$.
Then for all $n \in \IZ$ the set $B'_{R_0} ( \varphi^{(n)}(Q_0))$ contains all points of $\td{M}$ that have distance at most $L$ from $\varphi^{(n)} (F)$.
Consider for the moment some number $R$ such that
\[ R_0 \leq R \leq c N - C_2 - R_0. \]
($C_2$ is the constant from Proposition \ref{Prop:distballsinQQareballs}.)
Then we have $B'_{R_0} (Q_0) \subset B'_R (Q_0) \subset P_R (Q_0)$, so every point of $\td{M}$ that has distance at most $L$ from $D_0 \subset F$ is contained in $P_R(Q_0)$.
We now claim that for all $n \in \IZ$ with $|n| \geq N$ the set $B'_{R_0} (\varphi^{(n)} (Q_0))$ is disjoint from the interior of $B'_{R + C_2} (Q_0)$.
Assume not and let $Q' \subset B'_{R_0} (\varphi^{(n)} (Q_0)) \cap B'_{R + C_2} (Q_0)$ be a cell in the intersection.
Then we obtain the following contraction using (\ref{eq:distgrowslinearlyinQQ}):
\[ c|n| \leq \dist'(Q_0, \varphi^{(n)} (Q_0)) \leq \dist'(Q_0, Q') + \dist' (Q', \varphi^{(n)} (Q_0)) < R_0 + R + C_2 \leq cN. \]
So $B'_{R_0} (\varphi^{(n)} (Q_0))$ is disjoint from $\Int B'_{R + C_2} (Q_0)$ and hence also from $P_R (Q_0) \subset \Int B'_{R + C_2} (Q_0) \cup \partial \td{M}$.
So $P_R(Q_0)$ does not contain any point of $\td{M}$ that has distance at most $L$ from $\td{S} \setminus (\td{S}_+ \cup \td{S}_-)$ and thus also from $D_+, D_-$.
This implies in particular that the arcs $\td\sigma_+$ and $\td\sigma_-$ have intersection number $1$ with the restriction $b_{R, Q_0} |_{S^2} : S^2 \to \td{M}$, whose image is $\partial P_R(Q_0)$ (see Proposition \ref{Prop:distballsinQQareballs}).

Our conclusions imply that the homotopy $\td{H}^*$ restricted to $\partial P_R(Q_0)$ does not intersect $D_0 \cup D_+ \cup D_-$ or, more generally, that it stays away from $\td{S} \setminus (\td{S}_- \cup \td{S}_+)$.
If we view $b_{R, Q_0} |_{S^2}$ as a map from $S^2$ to $\td{V} \cup \partial \td{M}$, then $(x,t) \mapsto \td{H}^* (b_{R, Q_0} |_{S^2}(x), t)$ is a homotopy between $b_{R, Q_0} |_{S^2} = \td{f}^*_0 \circ b_{R, Q_0} |_{S^2} : S^2 \to \td{M}$ and $s_{R, Q_0} = \td{f}^* \circ b_{R, Q_0}|_{S^2} : S^2 \to \td{M}_0$ whose image is disjoint from $D_0 \cup D_+ \cup D_-$  and $\td{S} \setminus (\td{S}_- \cup \td{S}_+)$.
So $s_{R, Q_0}$ has intersection number $1$ with $\td\sigma_+$ and $\td\sigma_-$.
Choose a small perturbation $s'_{R, Q_0} : S^2 \to \td{M}_0$ of $s_{R, Q_0}$ that intersects $\partial \td{S}$ transversally, that still stays away from $D_0, D_+, D_-$ and and $\td{S} \setminus (\td{S}_- \cup \td{S}_+)$ and that satisfies
\begin{equation} \label{eq:sRQ0ssRQ0}
 \area s'_{R, Q_0} \big|_{{s'}^{-1}_{R, Q_0} (\td{S}_+ \cup \td{S}_-)} < 2 \area s_{R, Q_0} \big|_{s^{-1}_{R, Q_0}(\td{S}_+ \cup \td{S}_-)}.
\end{equation}
(This can always be achieved by perturbing the composition of $s_{R, Q_0}$ with a diffeomorphism of $\td{M}_0$ which slightly expands $\td{S}$.)
Set
\[ X = {s'}_{R, Q_0}^{-1} (\td{S}), \qquad  X_+ = {s'}_{R, Q_0}^{-1} (\td{S}_+), \qquad X_- = {s'}_{R, Q_0}^{-1} (\td{S}_-). \]
Then $X, X_+, X_-$ are compact smooth domains of $S^2$ and we have $X = X_+ \dotcup X_- $, $s'( \partial X) \subset \partial \td{S}$ and $s'_{R, Q_0}$ restricted to $X_+$ and $X_-$ has non-zero intersection number with $\td{\sigma}$.
Let $X'_+ \subset X_+$ be the union of all components of $X_+$ on which $s'_{R, Q_0}$ has non-zero intersection number with $\td\sigma$, define $X'_- \subset X_-$ analogously and set $X' = X'_+ \cup X'_-$.
Then $X'_+, X'_- \neq \emptyset$ and $X'_+, X'_- \neq S^2$ and every component $Y \subset X'$ is bounded by at least one circle $Z \subset \partial Y$ such that $s'_{R, Q_0} |_Z : Z \to \partial \td{S}$ is non-contractible in $\partial \td{S}$.
Each such circle $Z$ bounds two disks $E_1, E_2 \subset S^2$ on either side (one of these disks contains $Y$ and the other one doesn't).
Consider now the set of all such disks $E \subset S^2$ coming from all components $Y$ of $X'$ and all boundary circles $Z \subset \partial Y$ for which $s'_{R, Q_0} |_Z : Z \to \partial \td{S}$ is non-contractible in $\partial \td{S}$.
Any two such disks are either disjoint or one is contained in the other.
We can hence choose a component $Y \subset X'$, a boundary circle $Z \subset \partial X'$ with the aforementioned property and a disk $E \subset S^2$ bounded by $Z$ such that $E$ is minimal with respect to inclusion.
We argue that $s'_{R, Q_0}$ restricted to every other boundary circle $Z' \subset \partial Y$ is contractible in $\td{S}$:
If this was not the case, then $Y$ must be disjoint from the interior of $E$, since otherwise $Z' \subset Y \subset E$ bounds a disk $E' \subsetneq E$.
By the same argument, $E$ cannot contain any other component $Y'$ of $X'$, because otherwise we would find a boundary circle $Z'' \subset \partial Y' \subset \Int E$ such that $s'_{R, Q_0} |_{Z''}$ is non-contractible in $\partial \td{S}$.
So $\Int E$ must be be disjoint from $X'$ and hence $s'_{R, Q_0} |_{E}$ describes a nullhomotopy of a non-contractible loop in $\partial \td{S}$, which does not intersect $\td\sigma$.
Since $\pi_2 (\td{M}_0) = \pi_2 (M_0) = 0$, this nullhomotopy can be homotoped relative boundary to a nullhomotopy that has non-zero intersection number with $\td\sigma$.
This is however impossible and we obtain a contradiction.
So $s'_{R, Q_0}$ restricted to all other boundary components of $Y$ is non-contractible in $\td{S}$ and hence we have shown that $\Sigma = Y$ and $h = \pi \circ s'_{R, Q_0} |_Y$ satisfy all the claims of the Proposition except for the area bound.

In view of (\ref{eq:sRQ0ssRQ0}) it remains to choose $R$ and $N$ such that $\area s_{R, Q_0} |_{s^{-1}_{R, Q_0}(\td{S}_+ \cup \td{S}_-)}$ can be bounded by a uniform multiple of $\area f$.
To do this choose radii $R_i = R_0 + C_2 i$ where $i = 0, \ldots, e$ with $e = \lfloor C_2^{-1} (cN - C_2 - 2 R_0) \rfloor$.
Then
\[ R_0 < R_1 < \ldots < R_e \leq c N - C_2 - R_0. \]
By Proposition \ref{Prop:distballsinQQareballs} we know that $\partial P_{R_0} (Q_0) \setminus \partial \td{M}, \ldots, \partial P_{R_e} (Q_0) \setminus \partial \td{M} \subset \td{V} \subset \td{M}$ are pairwise disjoint.
So since $b_{R_i, Q_0} (S^2) = \partial P_{R_i}(Q_0)$ and $s_{R_i, Q_0} = \td{f}^* \circ b_{R_i, Q_0} |_{S^2}$ and $\td{f}^* (\partial \td{M}) = \partial \td{M}$ we have
\[ \area s_{R_0, Q_0} \big|_{s^{-1}_{R_0, Q_0} (\td{S}_+ \cup \td{S}_-)} + \ldots + \area s_{R_e, Q_0} \big|_{s^{-1}_{R_e, Q_0} (\td{S}_+ \cup \td{S}_-)} \leq \area \td{f}^* \big|_{\td{f}_0^{* -1} (\td{S}_+ \cup \td{S}_-)}. \]
Since $\td{S}_+ \cup \td{S}_- = \varphi^{(-N)} (F') \dotcup \ldots \dotcup \varphi^{(N-1)} (F') \cup D_-$ for the half-open set $F' = F \setminus D_0$, we further have
\[ \area \td{f}^* \big|_{\td{f}^{* -1} (\td{S}_+ \cup \td{S}_-)} = \area \td{f}^* \big|_{\td{f}^{* -1} (\varphi^{(-N)} (F'))} + \ldots + \area \td{f}^* \big|_{\td{f}^{* -1} (\varphi^{(N-1)} (F'))}. \]
Observe now that whenever $\td{f}^* (x) = \td{f}^* (y)$ and $\pi(x) = \pi(y)$ for $x, y \in \td{V}$, then $x = y$, since the arcs $t \mapsto \td{H}^* (x, t), \td{H}^*(y,t)$ have the same endpoint and project to the same arc $t \mapsto H^* (\pi(x), t)$ under $\pi$.
So for all $n \in \IZ$ the projection $\pi$ restricted to $\td{f}^{* -1}(\varphi^{(n)}(F'))$ is injective.
Since $\pi ( \td{f}^{* -1} (\varphi^{(n)}(F'))) \subset f^{*-1} (S)$, we conclude $\area \td{f}^* \big|_{\td{f}^{* -1} (\varphi^{(n)} (F'))} \leq \area f^* |_{f^{* -1} (S)} < \area f$.
Putting all this together yields
\[ \area s_{R_0, Q_0} \big|_{s^{-1}_{R_0, Q_0} (\td{S}_+ \cup \td{S}_-)} + \ldots + \area s_{R_e, Q_0} \big|_{s^{-1}_{R_e, Q_0} (\td{S}_+ \cup \td{S}_-)} < 2N \area f. \]
So we can find an index $i \in \{ 0, \ldots, e \}$ such that
\[ \area s_{R_i, Q_0} \big|_{s^{-1}_{R_i, Q_0} (\td{S}_+ \cup \td{S}_-)} < \frac{2N}{e+1} \area f \leq \frac{2 C_2 N}{cN - C_2 - 2 R_0} \area f. \]
Choosing $N > 2 c^{-1} (C_2 + 2R_0)$ yields
\[ \area s_{R_i, Q_0} \big|_{s^{-1}_{R_i, Q_0} (\td{S}_+ \cup \td{S}_-)} < 4 c^{-1} C_2 \area f. \]
This finishes the proof of Proposition \ref{Prop:easiermaincombinatorialresultCaseb}
\end{proof}

\subsection{The case in which $M$ is covered by a $T^2$-bundle over a circle} \label{subsec:CaseT2bundleoverS1}
We finally present the proof of Proposition \ref{Prop:maincombinatorialresult}(b).

\begin{Lemma} \label{Lem:geomseriesinSL2Z}
Let $A \in SL(2, \IZ)$ be a $2 \times 2$-matrix with integral entries and determinant $1$.
Then for every $k \geq 1$ there is a number $1 \leq d \leq 6^k$ such that
\[ I + A + A^2 + \ldots + A^{d-1} \equiv 0 \quad \mod 3^k. \]
Here $I$ is the identity matrix.
\end{Lemma}

\begin{proof}
We first show the claim for $k = 1$.
Since $\det A = 1$, the Theorem of Cayley-Hamilton yields that $I - (\tr A) A + A^2 = 0$.
So we are done for $\tr A \equiv 2 \mod 3$.
If $\tr A \equiv 0 \mod 3$, then $I + A + A^2 + A^3 = (I + A^2)(I + A) \equiv 0 \mod 3$ and if $\tr A \equiv 1 \mod 3$, then $I + \ldots + A^5 = (I - A + A^2) (I + 2 A + 2 A^2 + A^3) \equiv 0 \mod 3$.

We now apply induction.
Assume that the statement is true for all numbers up to $k \geq 1$.
We will show that it then also holds for $k+1$.
Choose $1\leq d_1 \leq 6$ such that $I + A + \ldots + A^{d_1 - 1}  \equiv 0 \mod 3$.
By the induction hypothesis applied to $A^{d_1} \in SL(2, \IZ)$, there is a number $1 \leq d_2 \leq 6^k$ such that $I + A^{d_1} + A^{2d_1} + \ldots + A^{(d_2-1)d_1} \equiv 0 \mod 3^k$.
So
\begin{multline*}
 I + A  + \ldots + A^{d_1 d_2 - 1} \\
 = (I + A + \ldots + A^{d_1 - 1}) (I + A^{d_1} + \ldots + A^{(d_2-1) d_1}) \equiv 0 \quad \mod 3^{k+1},
\end{multline*}
and $1 \leq d_1 d_2 \leq 6^{k+1}$.
\end{proof}

\begin{Lemma} \label{Lem:Misexpandable}
Assume that $M$ is the total space of a $T^2$-bundle over a circle.
Then for every $n \geq 1$ there is a finite covering map $\pi_n : M \to M$ with the same domain and range such that for every embedded loop $\sigma \subset M$ the preimage $\pi_n^{-1} (\sigma)$ consists of at least $n$ loops.
\end{Lemma}

\begin{proof}
The manifold $M$ is diffeomorphic to a mapping torus of an orientation-preserving diffeomorphism $\phi : T^2 \to T^2$.
The diffeomorphism $\phi$ acts on $\pi_1 (T^2) \cong \IZ^2$ by an element in $A \in SL(2, \IZ)$.
The fundamental group $\pi_1 (M)$ is isomorphic to a semidirect product of $\IZ$ with $\IZ^2$ coming from the action of $A$ on $\IZ^2$.
So $\pi_1(M)$ can be identified with $\IZ^2 \times \IZ$ with the following multiplication
\[ \left( \begin{pmatrix} x_1\\ y_1 \end{pmatrix}, z_1 \right) \cdot \left( \begin{pmatrix} x_2 \\ y_2 \end{pmatrix}, z_2 \right) = \left( \begin{pmatrix} x_1\\ y_1 \end{pmatrix} + A^{z_1} \begin{pmatrix} x_2 \\ y_2 \end{pmatrix}, z_1 + z_2 \right). \]
Since for any $m \geq 1$, the lattice $m \IZ^2 \subset \IZ^2$ is preserved by the action of $A$, the subset
\[ U_m = \left\{ \left( \begin{pmatrix}  m x \\  m y \end{pmatrix}, z \right) \;\; : \;\; x, y, z \in \IZ \right\} \subset \pi_1(M) \]
is a subgroup of $\pi_1(M)$ of index $m^2$.
Note that $U_m$ is isomorphic to $\pi_1(M)$, so if we consider the corresponding $m^2$-fold covering $\pi'_m : \widehat{M}_m \to M$, then $\widehat{M}_m$ is diffeomorphic to $M$.

It remains to compute the number of components of ${\pi'}_m^{-1} (\sigma)$ and to show that this number can be made arbitrarily large for the right choice of $m$.
Set $m = 3^k$ for a number $k \geq 1$, which we will determine later.
Let $\widehat{\sigma} \subset  \pi^{' -1}_m (\sigma)$ be an arbitrary loop in the preimage of $\sigma$.
Then we can find an element $g = ((x, y), z) \in \pi_1(M)$ in the conjugacy class of $[\sigma]$ such that $\sigma$ represents $g$ in $M$ and such that $\widehat{\sigma}$ represents a multiple of $g$ that is contained in $U_m = \pi_1 (\widehat{M}_m)$ in $\widehat{M}_m$.
Then the restriction $\pi'_m |_{\widehat{\sigma}} : \widehat{\sigma} \to \sigma$ is a covering of a circle and its degree is equal to the first exponent $d_0 \geq 1$ for which $g^{d_0} \in U_m$.
We will show that $d_0 \leq 6^k$.
To do this observe that for all $i \geq 1$
\[ g^i =  \left( \begin{pmatrix} x\\ y \end{pmatrix} + A^{z} \begin{pmatrix} x \\ y \end{pmatrix} + \ldots + A^{(i-1) z} \begin{pmatrix} x \\ y \end{pmatrix}, i z \right). \]
By Lemma \ref{Lem:geomseriesinSL2Z}, there is a number $1 \leq d \leq 6^k$ such that the first two entries of $g^d$ are divisible by $m = 3^k$ and hence $g^d \in U_m$.
This implies the desired bound.

Since the choice of $\widehat{\sigma}$ was arbitrary, we conclude that every loop in ${\pi'}_m^{-1} (\sigma)$ covers $\sigma$ at most $6^k$ times and hence the number of such loops is at least
\[ \frac{m^2}{6^k} = \frac{3^{2k}}{6^k} = \Big( \frac32 \Big)^k. \]
So choosing $k$ such that $( \frac32 )^k > n$ yields the desired result.
\end{proof}

\begin{proof}[Proof of Proposition \ref{Prop:maincombinatorialresult}(b)]
We only have to consider the case in which $M$ is a $T^2$-bundle over a circle, since for any finite cover $\widehat\pi : \widehat{M} \to M$ we can compose the maps $\widehat{f}_1, \widehat{f}_2, \ldots : V \to \widehat{M}$ obtained for $\widehat{M}$ with $\widehat\pi$ to obtain the maps $f_1 = \widehat\pi \circ \widehat{f}_1, f_2 = \widehat\pi \circ \widehat{f}_2, \ldots$.

We first establish the assertion for the case $n=1$.
Fix a $T^2$-bundle projection $p : M \to S^1$ and let $T = p^{-1} (\{ \text{pt} \}) \subset M$ be a torus fiber.
Then $M \setminus T$ is diffeomorphic to $T^2 \times (0,1)$ and we can find a smooth local diffeomorphism $\Phi : T^2 \times [0,1] = S^1 \times S^1 \times [0,1] \to M$ such that $\Phi |_{T^2 \times (0,1) }$ is a diffeomorphism onto $M \setminus T$ and such that $\Phi$ restricted to $T^2 \times \{ 0 \}$ and $T^2 \times \{ 1 \}$ is a diffeomorphism onto $T$.
Moreover, we may assume that $\Phi$ is chosen in such a way that the map $(\Phi |_{T^2 \times \{ 1 \} } )^{-1} \circ \Phi |_{T^2 \times \{ 0 \}} : T^2 \times \{ 0 \} \to T^2 \times \{ 0 \}$ is affine with respect to the standard affine structure on $T^2 = S^1 \times S^1$.
Let now $z \in S^1$ be an arbitrary point and set 
\[ V = T \cup \Phi ( S^1 \times \{ z \} \times [0,1] ) \cup \Phi (\{ z \} \times S^1 \times [0,1]). \]
Then $M \setminus V$ is diffeomorphic to a 3-ball and $V$ can be given the structure of a simplicial complex.
Fix such a structure for the rest of the proof and let $f_1 : V \to M$ be the inclusion map of $V$.

Consider the universal covering $\pi : \td{M} \to M$ and choose a lift $F \subset \td{M}$ of $M \setminus V$.
Then $F$ is a connected fundamental domain.
Denote furthermore by $\td{V} = \td\pi^{-1} (V) \subset \td{M}$ the preimage of $V$ and by $\td{f}_1 : \td{V} \to \td{M}$ its inclusion map.
The complement of $\td{V}$ in $\td{M}$ consists of open sets whose closures $Q \subset \td{M}$ are finite polyhedra and which we call \emph{cells}.
Denote the set of cells again by $\mathcal{Q}$.
Observe that every cell is the image of $F$ under a deck transformation of $\pi : \td{M} \to M$.
We say that two cells $Q_1, Q_2 \in \mathcal{Q}$ are adjacent if their intersection contains a point of $\td{V} \setminus \td{V}^{(1)}$.
Choose a cell $Q_0 \in \mathcal{Q}$ and consider for each $k \geq 0$ the union $B_k (Q_0)$ of all cells that have distance $\leq k$ in the adjacency graph of $\mathcal{Q}$.
Then $S_k = \partial B_k(Q_0) \subset \td{V}$ is the image of a continuous map $s_k : \Sigma_k \to \td{V}$ where $\Sigma_k$ is an orientable surface such that $s_k$ is an embedding on $\Sigma_k \setminus s_k^{-1} (\td{V}^{(1)})$.

Choose a component $\td\sigma \subset \pi^{-1} (\sigma)$ that intersects $Q_0$.
Then $\td\sigma \subset \td{M}$ is a non-compact, properly embedded line and there is a non-compact ray $\td\sigma^+ \subset \td\sigma$ that starts in $Q_0$.
This implies that $\td\sigma^+$ has non-zero intersection number with the map $\td{f}_1 \circ s_k : \Sigma_k \to \td{M}$ for each $k \geq 1$.

Consider now the continuous map $f'_1 : V \to M$ that is homotopic to $f_1 : V \to M$ via a homotopy $H : V \times [0,1] \to M$ and use this homotopy to construct a lift $\td{f}'_1 : \td{V} \to \td{M}$.
Let $N$ be a bound on the number of cells that each arc of the form $t \mapsto H (\cdot, t)$ intersects.
Then $H$ induces a homotopy from $\td{f}_1 \circ s_N : \Sigma_N \to \td{M}$ to $\td{f}'_1 \circ s_N : \Sigma_N \to \td{M}$, which is disjoint from $Q_0$.
So both maps have the same, non-zero, intersection number with $\td\sigma^+$.
We conclude that $\td{f}'_1( s_N ( \Sigma_N )) \subset \td{f}'_1 ( \td{V})$ intersects $\td\sigma$.
Hence, $f'_1(V)$ intersects $\sigma$.

We finally show the assertion for all remaining $n \geq 2$.
Fix $n$, consider the covering map $\pi_n : M \to M$ from Lemma \ref{Lem:Misexpandable} and set
\[ f_n = \pi_n \circ f_1 : V \to M. \]
Moreover, the preimage $\sigma_n= \pi_n^{-1} (\sigma)$ is the union of at least $n$ loops which all have the property that all its non-trivial multiples are non-contractible in $M$.

Consider a map $f'_n : V \to M$ and a homotopy between $f_n$ and $f'_n$.
This homotopy can be lifted via $\pi_n : M \to M$ to a homotopy between $f_1$ and a map $f'_1 : V \to M$ such that $f'_n = \pi_n \circ f'_1$.
We now have
\[ {f'}_n^{-1} (\sigma) ={f'}^{-1}_1 ( \pi_n^{-1} (\sigma)) =  {f'}^{-1}_1 ( \sigma_n ) = \bigcup_{\sigma' \subset \sigma_n} {f'}^{-1}_n (\sigma'), \]
where the last union is to be understood as the union over all loops $\sigma'$ of $\sigma_n$.
By our previous conclusion, ${f'}^{-1}_1 (\sigma') \neq \emptyset$ for all such $\sigma'$ and all such sets are pairwise disjoint.
This proves the desired result.
\end{proof}

\end{document}